\documentclass{amsart}
\usepackage{graphicx} 

\usepackage{fullpage,amsmath,amsthm,amsfonts,amssymb,graphicx,
setspace,amscd,float,hyperref,enumitem}
\usepackage[mathscr]{euscript}
\usepackage{color, xcolor}
\usepackage{tikz-cd}
\usepackage{xypic}
\usetikzlibrary{cd}
\usepackage{wrapfig}\usepackage{picins}
\usepackage{lipsum}
\usepackage{graphicx}

\definecolor{ivory}{rgb}{1.0, 1.0, 0.94}

\usepackage{todonotes}

\usepackage{zref-clever} 
\zcsetup{cap = true, nameinlink = true}
\newcommand{\Cref}[1]{\zcref{#1}}

\newtheorem{thm}{Theorem}[section]
\zcRefTypeSetup{thm}{Name-sg = Theorem}

\usepackage{xparse} 

\NewCommandCopy{\newtheoremcopy}{\newtheorem}
\RenewDocumentCommand{\newtheorem}{m O{thm} m}{ 
\newtheoremcopy{#1}[#2]{#3}
\AddToHook{env/#1/begin}{\zcsetup{countertype={thm=#1}}}
\zcRefTypeSetup{#1}{Name-sg = #3}}

\newtheorem{lem}[thm]{Lemma}
\newtheorem{prop}[thm]{Proposition}
\newtheorem{cor}[thm]{Corollary}
\newtheorem{qst}[thm]{Question}
\newtheorem{conj}[thm]{Conjecture}

\theoremstyle{definition}
\newtheorem{df}[thm]{Definition}
\newtheorem{rmk}[thm]{Remark}
\newtheorem{constr}[thm]{Construction}
\newtheorem{ex}[thm]{Example}

\begin{document}

\pagecolor{ivory}

\renewcommand{\AA}{\mathbb{A}}
  \newcommand{\BB}{\mathbb{B}}
  \newcommand{\CC}{\mathbb{C}}
  \newcommand{\DD}{\mathbb{D}}
  \newcommand{\EE}{\mathbb{E}}
  \newcommand{\FF}{\mathbb{F}}
  \newcommand{\HH}{\mathbb{H}}
  \newcommand{\KK}{\mathbb{K}}
  \newcommand{\MM}{\mathbb{M}}
  \newcommand{\NN}{\mathbb{N}}
  \newcommand{\OO}{\mathbb{O}}
  \newcommand{\PP}{\mathbb{P}}
  \newcommand{\QQ}{\mathbb{Q}}
  \newcommand{\RR}{\mathbb{R}}
  \renewcommand{\SS}{\mathbb{S}}
  \newcommand{\TT}{\mathbb{T}}
  \newcommand{\VV}{\mathbb{V}}
  \newcommand{\WW}{\mathbb{W}}
  \newcommand{\XX}{\mathbb{X}}
  \newcommand{\YY}{\mathbb{Y}}
  \newcommand{\ZZ}{\mathbb{Z}}
  \newcommand{\T}{\mathcal{T}}
  \newcommand{\mT}{\mathcal{T}}
\newcommand{\A}{\mathcal{A}}
\newcommand{\mA}{\mathcal{A}}
\newcommand{\mB}{\mathcal{B}}
\newcommand{\mD}{\mathcal{D}}
\newcommand{\I}{\mathcal{I}}
\newcommand{\mI}{\mathcal{I}}
\newcommand{\C}{\mathcal{C}}
\newcommand{\mC}{\mathcal{C}}
\newcommand{\mF}{\mathcal{F}}
\newcommand{\mG}{\mathcal{G}}
\newcommand{\G}{\Gamma}
\newcommand{\aut}{\textup{Aut}(F_r)}
\newcommand{\out}{\textup{Out}(F_r)}
\newcommand{\outt}{\textup{Out}(F_3)}
\newcommand{\os}{\textup{CV}_r}
\renewcommand{\phi}{\varphi} 
\newcommand{\vphi}{\varphi}
\newcommand{\veps}{\varepsilon}
\newcommand{\wtilde}{\widetilde}
\newcommand{\Av}{\mathcal{A}_{\vphi}}
\newcommand{\Tv}{T_{\vphi}^+}
\newcommand{\ft}{F_t}
\newcommand{\La}{\Lambda}

\newcommand{\teich}{Teichm\"{u}ller }

\newcommand{\from}{\colon}

\newcommand{\ind}{\mbox{ind}}
\newcommand{\vol}{\mathrm{vol}}

\newcommand{\gind}{\mbox{ind}_{\rm geom}}

\newcommand{\Ind}{\mbox{ind}}
\newcommand{\ol}{\overline}

\title{A ``cubist'' decomposition of the Handel-Mosher axis bundle}

\author[C. Pfaff]{Catherine Eva Pfaff}
\address{Queen's University}
\email{\tt catherine.pfaff@gmail.com}
\urladdr{https://mast.queensu.ca/~cpfaff/}

\author[C. C. Tsang]{Chi Cheuk Tsang}
\address{Université du Québec à Montréal}
\email{\tt tsang.chi\_cheuk@uqam.ca}
\urladdr{https://sites.google.com/view/chicheuktsang}

\begin{abstract}
We show that the axis bundle of a nongeometric fully irreducible outer automorphism admits a canonical ``cubist'' decomposition into branched cubes that fit together with special combinatorics. From this structure, we locate a canonical finite collection of periodic fold lines in each axis bundle. This can be considered as an analogue of results of Hamenstädt and Agol from the surface setting, which state that the set of trivalent train tracks carrying the unstable lamination of a pseudo-Anosov map can be given the structure of a CAT(0) cube complex, and that there is a canonical periodic fold line in this cube complex. This work also gives an answer to questions of Handel-Mosher and Bridson-Vogtmann regarding the geometry of the axis bundle and a solution of a new flavor to the fully irreducible conjugacy problem in $\mathrm{Out}(F_r)$.
\end{abstract}

\thanks{Both authors are grateful to Lee Mosher for inspiring conversations and interest in their work. The first author is supported by an NSERC discovery grant and is appreciative of the hospitality of the IAS (and Bob Moses fund for funding her membership) and CIRGET (under the generous direction of Steven Boyer and Vestislav Apostolov). 
This project was completed while the second author was a CRM-ISM postdoctoral fellow based at CIRGET. He would like to thank the center for its support.}  

\date{}
\maketitle

\section{Introduction}

Let $F_r$ denote the free group of rank $r \geq 2$ and let $\out$ be its outer automorphism group. 
The \textbf{outer space} associated to $F_r$, defined in \cite{cv86} and denoted here by $CV_r$, is the projectivized space of marked metric graphs with fundamental group isomorphic to $F_r$ (\S \ref{s:os}). 
Equivalently, $CV_r$ is the projectivized space of free, properly discontinuous $F_r$-actions on $\mathbb{R}$-trees. 
One can compactify $CV_r$ by taking $\overline{CV_r}$ to be the projectivized space of very small $F_r$-actions on $\mathbb{R}$-trees \cite{bf94,cl95}.

An element $\vphi\in\out$ is \textbf{fully irreducible} if no power preserves the conjugacy class of a nontrivial proper free factor, and \textbf{geometric} if induced by a surface homeomorphism. We focus on nongeometric fully irreducible $\vphi$. Each fully irreducible $\vphi$ acts on $\overline{CV_r}$ with north-south dynamics \cite{ll03}, i.e. there is an attracting tree $T_+ \in \partial CV_r$ and a repelling tree $T_- \in \partial CV_r$ such that $\phi^k(\overline{CV_r} \backslash T_-)$ converges to $T_+$, uniformly on compact sets, as $k \to \infty$. A \textbf{fold line} for $\phi$ is a bi-infinite path in $CV_r$ from $T_-$ to $T_+$ which, morally speaking, is an axis for the $\phi$-action. Fold lines can be defined purely in terms of graph combinatorics, without reference to the $CV_r$ metric.
The \textbf{axis bundle} of $\phi$, introduced in \cite{hm11} and denoted here by $\mathcal{A}_\phi$, is the union of its fold lines. 

Handel and Mosher, via questions in \cite{hm11}, and Bridson and Vogtmann \cite[Question 3]{bv06} more directly, ask: Describe the geometry of $\Av$ for a fully irreducible $\vphi$.
The coarse topology of $\mathcal{A}_\phi$ was classified in \cite{hm11}, where they show the inclusion of each fold line is a proper homotopy equivalence. The finer combinatorics of $\mathcal{A}_\phi$ has remained mysterious, including the impacts of the `tripod folds' and `singularity merging' introduced in \cite{PffAutomata} and \cite{stablestrata}. We initiate a new framework to study an $\mathcal{A}_\phi$ by showing it admits a specific combinatorial structure (details are in \S \ref{sec:axisbundlebranchedcubes}-\ref{sec:cubistcomplex}, with specific examples given in \Cref{sec:ex}):

\begin{thm} \label{thm:introcubistcomplex}
Let $\phi$ be a nongeometric fully irreducible outer automorphism. Then $\mathcal{A}_\phi$ admits a canonical structure of a cubist complex.
From the cubist complex structure, there is a canonically defined directed graph $\mathfrak{c}_{\mathcal{A}_\phi}$ embedded in $\mathcal{A}_\phi$, we call the \textbf{cardiovascular system}, satisfying the following properties:
\begin{enumerate}[label=(\roman*)]
    \item There is a finite set of bi-infinite directed lines (called \textbf{arteries}) on which $\phi$ acts periodically.
    \item Each vertex of $\mathfrak{c}_{\mathcal{A}_\phi}$ has a unique outgoing edge, thus has a well-defined forward trajectory. Each forward trajectory eventually enters an artery.
    \item Any two arteries are related by sweeping across finitely many 2-dimensional branched cubes.
\end{enumerate}
\end{thm}

Here, the notion of a cubist complex is based on that of a cube complex, with three key differences being:\\
1.) A cube complex decomposes into cubes, whereas a cubist complex decomposes into \textbf{branched cubes}, i.e. unions of cubes glued along certain affine slices. The figure below shows a 2-dimensional branched cube formed by gluing three 2-dimensional cubes.\\
    \parpic[r]{\selectfont\fontsize{8pt}{8pt} 
\begingroup%
  \makeatletter%
  \providecommand\color[2][]{%
    \errmessage{(Inkscape) Color is used for the text in Inkscape, but the package 'color.sty' is not loaded}%
    \renewcommand\color[2][]{}%
  }%
  \providecommand\transparent[1]{%
    \errmessage{(Inkscape) Transparency is used (non-zero) for the text in Inkscape, but the package 'transparent.sty' is not loaded}%
    \renewcommand\transparent[1]{}%
  }%
  \providecommand\rotatebox[2]{#2}%
  \newcommand*\fsize{\dimexpr\f@size pt\relax}%
  \newcommand*\lineheight[1]{\fontsize{\fsize}{#1\fsize}\selectfont}%
  \ifx\svgwidth\undefined%
    \setlength{\unitlength}{224.99102495bp}%
    \ifx\svgscale\undefined%
      \relax%
    \else%
      \setlength{\unitlength}{\unitlength * \real{\svgscale}}%
    \fi%
  \else%
    \setlength{\unitlength}{\svgwidth}%
  \fi%
  \global\let\svgwidth\undefined%
  \global\let\svgscale\undefined%
  \makeatother%
  \begin{picture}(1,0.34096408)%
    \lineheight{1}%
    \setlength\tabcolsep{0pt}%
    \put(0,0){\includegraphics[width=\unitlength,page=1]{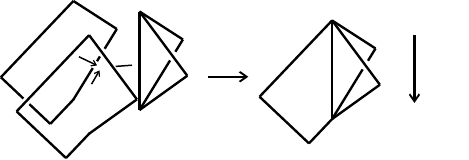}}%
    \put(0.60506706,0.31857067){\color[rgb]{0,0,0}\makebox(0,0)[lt]{\lineheight{1.25}\smash{\begin{tabular}[t]{l}splitting vertex\end{tabular}}}}%
    \put(0.55212479,0.0055081){\color[rgb]{0,0,0}\makebox(0,0)[lt]{\lineheight{1.25}\smash{\begin{tabular}[t]{l}folding vertex\end{tabular}}}}%
    \put(0.78664496,0.28077341){\color[rgb]{0,0,0}\makebox(0,0)[lt]{\lineheight{1.25}\smash{\begin{tabular}[t]{l}folding direction\end{tabular}}}}%
    \put(0,0){\includegraphics[width=\unitlength,page=2]{introbranchedcube.pdf}}%
  \end{picture}%
\endgroup%
}
\noindent 2.) A cubist complex is inherently directed: Each branched cube is equipped with a splitting vertex and `antipodal' folding vertex. Intuitively, the \textbf{folding direction} runs from the splitting vertex to the folding vertex, whereas the \textbf{splitting direction} is oppositely directed. The branched cubes must intersect in a manner preserving this directionality.\\
3.) When two cubes of a cube complex intersect, their intersection is a (complete) face of each of the two cubes, whereas when two branched cubes intersect in a cubist complex, their intersection is a (complete) face in the splitting side of one branched cube and a \emph{subset} of a face in the folding side of the other branched cube. Morally, this causes the branched cubes to get `finer', resembling a `Zenotic' division, in the folding direction. The black squares of \Cref{fig:introcubisteg} provide an example.

Property (2) follows directly from how we defined the branched cubes. 
For property (1), the branching of
 the branched cubes arises when the turns share directions. For example, if 3 directions lie in a single gate, then there are three edges in the branched cube, coming from the three ways to choose 2 of the 3 directions to fold. Further, folding two of the pairs of directions will force the remaining pair to be folded as well. The result is a branched cube, as depicted in the image accompanying (2) above.

\parpic[r]{\selectfont\fontsize{8pt}{8pt} 
\begingroup%
  \makeatletter%
  \providecommand\color[2][]{%
    \errmessage{(Inkscape) Color is used for the text in Inkscape, but the package 'color.sty' is not loaded}%
    \renewcommand\color[2][]{}%
  }%
  \providecommand\transparent[1]{%
    \errmessage{(Inkscape) Transparency is used (non-zero) for the text in Inkscape, but the package 'transparent.sty' is not loaded}%
    \renewcommand\transparent[1]{}%
  }%
  \providecommand\rotatebox[2]{#2}%
  \newcommand*\fsize{\dimexpr\f@size pt\relax}%
  \newcommand*\lineheight[1]{\fontsize{\fsize}{#1\fsize}\selectfont}%
  \ifx\svgwidth\undefined%
    \setlength{\unitlength}{149.94327695bp}%
    \ifx\svgscale\undefined%
      \relax%
    \else%
      \setlength{\unitlength}{\unitlength * \real{\svgscale}}%
    \fi%
  \else%
    \setlength{\unitlength}{\svgwidth}%
  \fi%
  \global\let\svgwidth\undefined%
  \global\let\svgscale\undefined%
  \makeatother%
  \begin{picture}(1,0.75451621)%
    \lineheight{1}%
    \setlength\tabcolsep{0pt}%
    \put(0,0){\includegraphics[width=\unitlength,page=1]{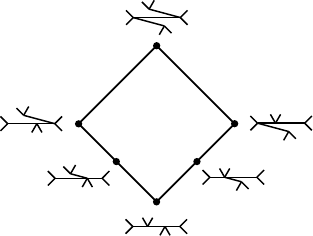}}%
  \end{picture}%
\endgroup%
}
For property (3), the Zenotic division in the folding direction arises if folding some of the pairs of directions causes vertices to appear in the path of the remaining folds (some folds are longer than others), so that in the process of carrying out these remaining folds, one passes through multiple fully preprincipal train tracks. To the right is an example.

We outline how the cubist complex structure arises for an $\mathcal{A}_\phi$:

For each $T \in \mathcal{A}_\phi$, there is a map $T \to T_+$ that respects the $F_r$-actions and restricts to an isometry on each leaf of the \textbf{attracting lamination} $\Lambda$ \cite{bfh97} determined by $T_-$. 
In general, there is an interval's worth of such maps, but we associate a canonical \textbf{$\Lambda$-isometry} to each $T$. 
A pair of directions at a vertex of $T$ is an \textbf{illegal turn} if it is identified in $T_+$ by the $\Lambda$-isometry.
A \textbf{gate} is an equivalence class of directions, under the equivalence relation generated by illegal turns.

\begin{figure}[ht!]
    \centering
    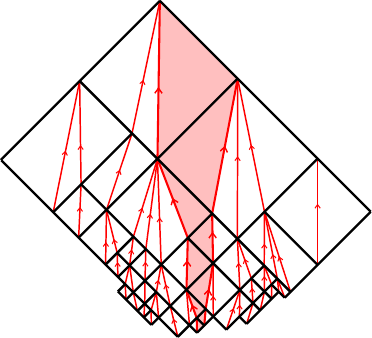
    \caption{A cubist complex (in black) and its cardiovascular system (in red).}
    \label{fig:introcubisteg}
\end{figure}

The branched cubes in the cubist complex structure are as follows.
The nodes, i.e. 0-dimensional cubes, in $\mathcal{A}_\phi$ are the \textbf{fully preprincipal weak train tracks} introduced in \cite{PffAutomata}. These are the graphs in $CV_r$ with at least three gates at each vertex. The edges, i.e. 1-dimensional cubes, are obtained by folding illegal turns. The higher dimensional branched cubes are obtained by folding illegal turns by various amounts.

From this cubist structure, the cardiovascular system $\mathfrak{c}_{\mathcal{A}_\phi}$ is defined by connecting, via directed segments, the nodes of $\mathcal{A}_\phi$ in the splitting direction.
See the \Cref{fig:introcubisteg} red graph, for an example, and see \S \ref{subsec:cardiovascularsystem} for the precise definition.
The refinement of $X$ in the folding direction implies that each node in $\mathfrak{c}_{\mathcal{A}_\phi}$ has precisely one outgoing edge but possibly many incoming edges. Thus each node has a unique forward trajectory, and intuitively these trajectories (which go in the splitting direction) tend to converge together.
The portion of a forward trajectory from the point it converges into a translate on then lies on an artery.
This shows \Cref{thm:introcubistcomplex}(i)-(ii).
Item (iii) follows from a more intricate analysis of the combinatorics of cubist complexes.

The nodes and edges of each artery determine a periodic folding sequence for $\phi$.
Varying over the arteries determines a canonical finite collection $\mathfrak{s}_\phi$ of such folding sequences.
If there are multiple arteries, then \Cref{thm:introcubistcomplex}(iii) states that the different sequences in this collection differ in a highly controlled way. 
However, we have reason to believe that having a unique artery is a generic property (being nongeometric fully irreducible is proven generic in \cite{r08} or \cite{randomout}, for example).

\begin{conj} \label{conj:uniquearterygeneric}
For a reasonable random walk on $\out$, having that $\mathcal{A}_\phi$ has a unique artery is generic.
\end{conj}

Regardless of whether \Cref{conj:uniquearterygeneric} is true, we can use the arteries to provide a new solution to the conjugacy problem for nongeometric fully irreducible outer automorphisms:
\begin{thm} \label{thm:introconjproblem}
Two nongeometric fully irreducible  $\phi, \phi' \in \out$ are conjugate if and only if the collections $\mathfrak{s}_\phi$ and $\mathfrak{s}_{\phi'}$ share a common periodic folding sequence (up to a translation on the indexing). 
\end{thm}

While the $Out(F_r)$ conjugacy problem has other solutions, see \cite{sela95isomorphism}, \cite{los96}, \cite{dahmani16}, our solution is of a different nature: The crux of these previous solutions is an algorithm that searches for possible conjugating elements. Even with a linear bound on the word length of such elements, the runtime of such an algorithm would be (at least) exponential. Meanwhile, our solution involves computing a combinatorial object for each of the given outer automorphisms and then comparing them, generalizing the approach in \cite{loneaxes} for lone axis fully irreducibles. In this direction we ask the following question.

\begin{qst}
Is there an algorithm to compute the axis bundle of a given nongeometric fully irreducible $\phi$ as a cubist complex, which runs in polynomial time in the word length and/or the stretch factor of $\phi$?
\end{qst}

In \Cref{sec:ex}, we compute examples of axis bundles and arteries using a rather straightforward algorithm, we have not tried to bound the run time of our process.
The focus, by showcasing a variety of phenomena regarding the structure of an $\mathcal{A}_\phi$ and its arteries, is on the questions of Handel-Mosher and Bridson-Vogtmann. In particular, \Cref{ex:bottleneck} is an example where a `tripod fold', as defined in \cite{PffAutomata}, causes the local dimension of the axis bundle to be nonconstant. This appears to be the first time this phenomenon is observed.



\subsection{Geometry of the axis bundle}\label{S:ABGeometry}
We compare the $\os$ situation to that in a hyperbolic space $\mathbb{H}^n$.

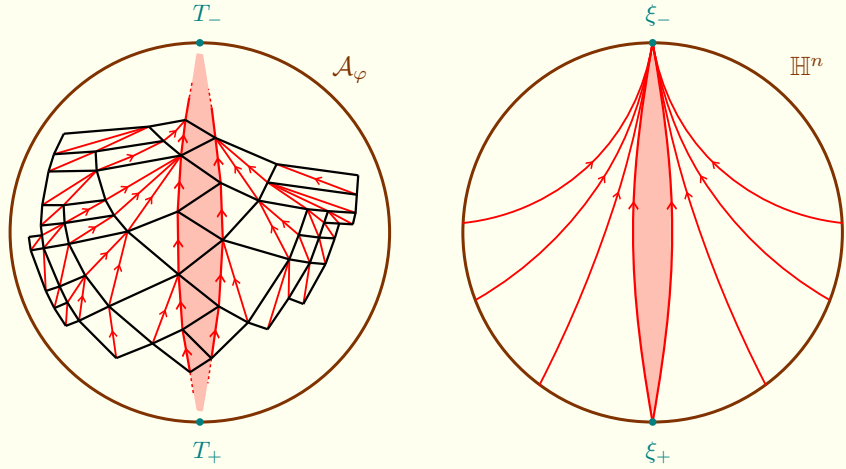
\begin{figure}[H]
    \centering
    \selectfont\fontsize{10pt}{10pt}
\begingroup%
  \makeatletter%
  \providecommand\color[2][]{%
    \errmessage{(Inkscape) Color is used for the text in Inkscape, but the package 'color.sty' is not loaded}%
    \renewcommand\color[2][]{}%
  }%
  \providecommand\transparent[1]{%
    \errmessage{(Inkscape) Transparency is used (non-zero) for the text in Inkscape, but the package 'transparent.sty' is not loaded}%
    \renewcommand\transparent[1]{}%
  }%
  \providecommand\rotatebox[2]{#2}%
  \newcommand*\fsize{\dimexpr\f@size pt\relax}%
  \newcommand*\lineheight[1]{\fontsize{\fsize}{#1\fsize}\selectfont}%
  \ifx\svgwidth\undefined%
    \setlength{\unitlength}{322.14044286bp}%
    \ifx\svgscale\undefined%
      \relax%
    \else%
      \setlength{\unitlength}{\unitlength * \real{\svgscale}}%
    \fi%
  \else%
    \setlength{\unitlength}{\svgwidth}%
  \fi%
  \global\let\svgwidth\undefined%
  \global\let\svgscale\undefined%
  \makeatother%
  \begin{picture}(1,0.52903277)%
    \lineheight{1}%
    \setlength\tabcolsep{0pt}%
    \put(0,0){\includegraphics[width=\unitlength,page=1]{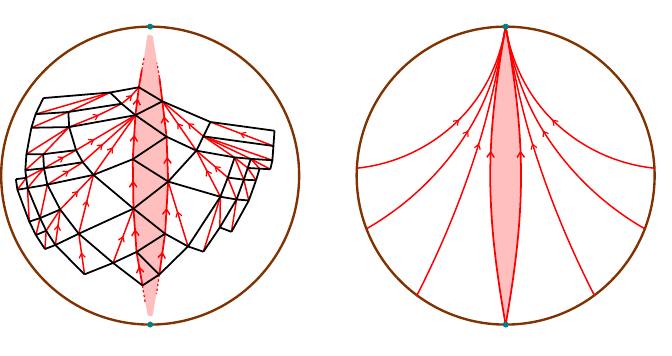}}%
    \put(0.21543403,0.51469995){\color[rgb]{0,0.50196078,0.50196078}\makebox(0,0)[lt]{\lineheight{1.25}\smash{\begin{tabular}[t]{l}\small $T_-$\end{tabular}}}}%
    \put(0.21545175,0.0032857){\color[rgb]{0,0.50196078,0.50196078}\makebox(0,0)[lt]{\lineheight{1.25}\smash{\begin{tabular}[t]{l}\small $T_+$\end{tabular}}}}%
    \put(0.74509918,0.51469754){\color[rgb]{0,0.50196078,0.50196078}\makebox(0,0)[lt]{\lineheight{1.25}\smash{\begin{tabular}[t]{l}\small $\xi_-$\end{tabular}}}}%
    \put(0.7451169,0.00328315){\color[rgb]{0,0.50196078,0.50196078}\makebox(0,0)[lt]{\lineheight{1.25}\smash{\begin{tabular}[t]{l}\small $\xi_+$\end{tabular}}}}%
    \put(0.37840726,0.48997075){\color[rgb]{0.50196078,0.2,0}\makebox(0,0)[lt]{\lineheight{1.25}\smash{\begin{tabular}[t]{l}\huge\\$\Av$\end{tabular}}}}%
    \put(0.91413012,0.48997363){\color[rgb]{0.50196078,0.2,0}\makebox(0,0)[lt]{\lineheight{1.25}\smash{\begin{tabular}[t]{l}\huge\\$\HH^n$\end{tabular}}}}%
  \end{picture}%
\endgroup%

    \caption{The dynamics of the cardiovascular system resembles that of geodesics in hyperbolic space that limit onto a fixed boundary point $\xi_-$, possibly with the invariant geodesic being blown-up into a flat Euclidean strip (region in red).}
    \label{fig:introhypanalogue}
\end{figure}

Each node of the cardiovascular system $\mathfrak{c}_{\mathcal{A}_\phi}$ determines a directed ray.
These directed rays converge into a finite collection of periodic bi-infinite lines. We compare this situation with the dynamics of geodesics in $\mathbb{H}^n$: For a fixed $\xi \in \partial \mathbb{H}^n$, the collection of geodesics limiting onto $\xi$ converge together in forward time. If $\phi$ is a loxodromic isometry of $\mathbb{H}^n$ fixing $\xi$, one can extract from this collection a geodesic invariant under $\phi$.

This is, however, not a perfect analogy since the axis bundle can have multiple arteries. 
In the case of multiple arteries, \Cref{thm:introcubistcomplex}(iii) states that any two arteries are related by sweeping across finitely many 2-dimensional branched cubes. One can morally regard the union of the 2-dimensional branched cubes as a flat strip, see the shaded region of \Cref{fig:introhypanalogue} left, and we can revise the analogy by `blowing up' the invariant geodesic in $\mathbb{H}^n$ by inserting a Euclidean strip, see \Cref{fig:introhypanalogue} right.

\subsection{Motivation from surfaces}

Recall that a \textbf{train track} on a surface $S$ is an embedded graph $\tau$ where the half-edges at each vertex are tangent to a common tangent line.
In \cite{Ham09}, Hamenst\"adt showed that the set of trivalent train tracks on $S$ can be realized as the vertex set of a cube complex $\mathcal{TT}_S$ by taking the edges to be splitting moves, and the higher-dimensional cubes to be spanned by commuting splitting moves.
We regard $\mathcal{TT}_S$ as an analogue of outer space in the surface setting.

Let $f$ be a pseudo-Anosov map on $S$, and let $\ell_+$ be the attracting lamination of $f$.
The results of \cite{Ham09} imply that the set of trivalent train tracks that carry $\ell_+$ determines a CAT(0) subcomplex of $\mathcal{TT}_S$.
We denote this subcomplex as $\mathcal{A}_f$ and regard it as an analogue of the axis bundle. This analogy is justified by the fact that the axis bundle $\mathcal{A}_\phi$ admits a similar description as the set of graphs $\Gamma$ that carry the expanding lamination $\Lambda$, in the sense that there is a $\Lambda$-isometry from $\widetilde{\Gamma}$ to the attracting tree $T_+$. 

In this perspective, the first statement in \Cref{thm:introcubistcomplex}, that $\mathcal{A}_\phi$ is a cubist complex, is an analogue of the fact that $\mathcal{A}_f$ is a cube complex.
However, the vertices of $\mathcal{A}_f$ are trivalent train tracks, which are the `most generic' type of train track, in the sense that they form an open subset of the set of all train tracks carrying $\ell_+$.
On the other hand, the vertices of $\mathcal{A}_\phi$ are fully preprincipal train tracks, which are the `least generic' type of train track since they are the cubes with the highest possible codimension.
Hence it might be more accurate to say that our cubist complex structure on $\mathcal{A}_\phi$ is an analogue of a `dual complex' of $\mathcal{A}_f$. 

We also mention an unfinished monograph of Mosher \cite{Mos03} that explores similar ideas of constructing complexes out of train tracks and splitting moves (albeit of a more general type than those in \cite{Ham09}).

In \cite{Ago11}, Agol showed there is a canonical axis of the action of $f$ on $\mathcal{A}_f$ and the corresponding periodic splitting sequence gives a complete conjugacy invariant of the pseudo-Anosov map $f$.
\Cref{thm:introcubistcomplex}(i) and \Cref{thm:introconjproblem} can be considered as analogues of these facts. On the other hand, our methods differ from those of Agol in \cite{Ago11}: He uses the \textbf{maximal splitting} operation, which means splitting the branches of a train track that have maximal weight with respect to the transverse measure on $\ell_+$, while the operation underlying a cardiovascular system is closer to performing all possible splits. 

\begin{qst} \label{qst:maxsplit}
Is there a meaningful version of Agol's maximal splitting operation in the free group setting?
\end{qst}

There are some inherent difficulties in a na\"ive generalization. For example, while a natural choice of weights can be assigned to branches via studying the incidence matrices, there is no guarantee that a branch with maximal weight can be split or that branches of maximal weight are disjoint. The latter fact is in part due to our train tracks not being trivalent -- some arguments in \cite{Ago11} break down because of this. Nevertheless, a positive answer to \Cref{qst:maxsplit} would suggest a positive answer to the following question:

\begin{qst} \label{qst:singleaxis}
Is it possible to upgrade \Cref{thm:introcubistcomplex}(i) from a finite canonical collection of fold lines to a single canonical fold line?
\end{qst}

One motivation for \Cref{qst:singleaxis} is that pseudo-Anosov maps have a unique axis in Teichm\"uller space. By contrast, fully irreducible outer automorphisms have, in general, uncountably many fold lines. \Cref{thm:introcubistcomplex}(i) extracts a finite canonical collection of fold lines out of these, but it would be more satisfying if one could take things a step further and extract a single canonical fold line. 






\subsection{Generalization to cut decomposition axis bundles}

A \textbf{(weak) periodic Nielsen path (PNP)} in a tree $T \in \mathcal{A}_\phi$ is a homotopically nontrivial path in $T$ whose endpoints are principal points with the same image in $T_+$.
Each PNP can be written as a concatenation of indivisible PNPs, which are paths of the form $\alpha_1^{-1} * \alpha_2$, with $\alpha_1$ and $\alpha_2$ mapped to the same interval in $T_+$.
We refer to \cite{bh92,hm11} for details. 

Intuitively, a PNP represents some redundancy of $T$: Given an indivisible PNP $\alpha_1^{-1} * \alpha_2$ in $T$, one can fold $T$ by identifying $\alpha_1$ and $\alpha_2$ and get a `reduced' element $\overline{T}$ of $\mathcal{A}_\phi$ with a naturally induced $\Lambda$-isometry $\overline{T} \to T_+$. For this reason, it is sometimes convenient to consider instead the \textbf{stable axis bundle} $\mathcal{SA}_\phi$, which is the subset of $\mathcal{A}_\phi$ that consists of elements without PNPs.
See, for example, \cite{hm11,loneaxes}. 
The machinery developed in this paper fully carries over to this stable category. 

To be more general yet, we define the \textbf{cut decomposition axis bundles}: Handel and Mosher define a finite graph, $IW(\phi)$, that records how lamination leaves pass through singularities. The possible PNPs in trees in $\mathcal{A}_\phi$ correspond to the cut vertices of $IW(\phi)$, see \cite[\S 4]{hm11}. 
Fixing a subset of possible PNPs is equivalent to choosing a cut decomposition $\mathcal{G}$ of $IW(\phi)$.
For each cut decomposition $\mathcal{G}$, we define the \textbf{$\mathcal{G}$-axis bundle} $\mG\Av$ as the subset of $\mathcal{A}_\phi$ that consists of elements whose PNPs determine the cut decomposition $\mG$.

\begin{thm} \label{thm:intropartialstableaxis}
Let $\phi$ be a nongeometric fully irreducible outer automorphism. 
Let $\mG$ be a cut decomposition of $IW(\phi)$. The $\mG$-axis bundle $\mG\Av$ of $\phi$ admits a canonical structure of a cubist complex which makes it a subcubist complex of $\mathcal{A}_\phi$.
From the cubist complex structure, there is a canonically defined directed graph $\mathfrak{c}_{\mG\Av}$ embedded in $\mG\Av$, which we call the \textbf{cardiovascular system}, satisfying the following properties:
\begin{enumerate}[label=(\roman*)]
    \item There is a finite set of bi-infinite directed lines (called \textbf{arteries}) on which $\phi$ acts periodically.
    \item Each vertex of $\mathfrak{c}_{\mG\Av}$ has a unique outgoing edge, thus has a well-defined forward trajectory. Each forward trajectory eventually enters an artery.
    \item Any two arteries are related by sweeping across finitely many 2-dimensional branched cubes.
\end{enumerate}
\end{thm}

We expect that the cubist complex machinery can be used to study the following problem.

\begin{qst}
How do the cut decomposition axis bundles $\mG\Av$ sit inside the full axis bundle $\mathcal{A}_\phi$?
\end{qst}

As a final remark, we note that even though $\mG\Av$ is a subcubist complex of $\mathcal{A}_\phi$, the arteries in $\mG\Av$ and $\mathcal{A}_\phi$ can be completely different. We demonstrate an example of this in \Cref{ex:2dimbundle1dimstable}.

\section{Preliminary definitions and notation}{\label{S:Prelims}}

Throughout this paper, $F_r$ is the rank $r$ free group and $\out$ its outer automorphism group. 
We further assume throughout that $r \geq 3$.
Finally, $\vphi \in \out$ will denote a nongeometric fully irreducible outer automorphism. Recall from the introduction that this means no power of $\phi$ preserves the conjugacy class of a nontrivial proper free factor and $\phi$ is not induced by a surface homeomorphism. 

Expository overlap may occur between this section and previous papers of the authors.

\subsection{Culler-Vogtmann Outer Space}\label{s:os} 

Outer space was introduced by Culler and Vogtmann in \cite{cv86} as an $\out$ analogue to \teich space.

Let $R_r$ be the $r$-petaled rose, i.e. the graph with precisely $r$ edges and one vertex. Recall from \cite{bh92} that a \textbf{marked graph} of rank $r$ is a connected finite graph $\Gamma$, with no valence $1$ or $2$ vertices, together with an isomorphism $\pi_1(\Gamma) \cong F_r$ defined via a homotopy equivalence (called the \textbf{marking}) $\rho \colon \Gamma \to R_r$. Marked graphs $\rho \colon \Gamma \to R_r$ and $\rho' \colon \Gamma' \to R_r$ are considered equivalent when there exists a homeomorphism $h \colon \Gamma \to \Gamma'$ such that $\rho' \circ h$ is homotopic to $\rho$. 
We denote the edge set by $E\G$ and the vertex set by $V\G$.

A \textbf{metric} on $\Gamma$ is the path metric determined by choosing for each edge $e$ of $\Gamma$ a length $\ell(e)$ and a characteristic map $j_e \colon [0,\ell(e)] \to e$, in the sense of CW complexes. 
A metric is determined, up to homeomorphism isotopic to the identity, by the assignment of lengths $\ell(e)$.
The \textbf{volume} of $\Gamma$ is defined as $\vol(\Gamma) := \sum\limits_{e \in E(\Gamma)} \ell(e)$. 

The \textbf{unprojectivized outer space} $\widehat{CV_r}$ is the space of all metric marked graphs of rank $r$ modulo marking-preserving isometry.
The \textbf{outer space} $CV_r$ itself is the projectivization of $\widehat{CV_r}$, i.e. the quotient of $\widehat{CV_r}$ by homothety. By instead viewing the points in $CV_r$ as those $\G$ with $\vol(\Gamma)=1$, one obtains its decomposition into disjoint open simplices, one for each marked graph.

Lifting to universal covers yields an alternative definition of $CV_r$: 
Given a marked graph $(\G,m)$ in $\os$,  the universal cover is a simplicial tree with a free $\pi_1(\G) \cong F_r$-action by deck transformations. 
The \textbf{compactified outer space} $\overline{CV_r} = CV_r \cup \partial CV_r$ is the space of minimal, very small $F_r$-actions on $\mathbb{R}$-trees, known as \textbf{$F_r$-trees}, modulo $F_r$-equivariant homothety. We may pass between viewing $CV_r$ as a space of marked graphs or a space of $F_r$-trees without comment.

A \textbf{direction} at a point $p$ in an $\mathbb{R}$-tree $T$ is a connected component of $T\backslash \{p\}$, and if there are $\geq 3$ directions at $p$ it is a \textbf{branch point}. A \textbf{turn} at $p$ is an unordered pair of directions at $p$. 
Let $\mD_pT$ denote the set of directions at $p$, or $\mD(p)$ if $T$ is clear. Define $\mD T:=\cup_{p\in T}\mD_pT$.
Given a locally injective continuous map $f\from T\to T'$ of $\mathbb{R}$-trees, define a \textbf{direction map} $Df\from\mD_pT\to\mD_{f(p)}T'$ by sending a direction to its $f$-image. 

\subsection{Train track representatives} \label{subsec:ttrep}

A homotopy equivalence $g \colon \Gamma \to \Gamma$ of a marked graph $\Gamma$ is a \textbf{train track representative} for $\varphi \in Out(F_r)$ if it maps vertices to vertices, $\varphi = g_{\ast} \colon \pi_1(\Gamma) \to \pi_1(\Gamma)$, and $g^k \mid _{int(e)}$ is locally injective for each $e\in E\G$ and $k>0$. Many of the definitions and notation for discussing train track representatives were established in \cite{bh92} and \cite{bfh00}. We recall some here. 

A \textbf{direction} at a vertex $v\in V\G$ is a germ of initial segments of directed edges emanating from $v$. The set of directions at $v$ is denoted $\mD(v)$ and $Dg$ denotes the direction map induced by $g$. A point $v$ \textbf{periodic} if there exists a $j \geq 1$ such that $g^j (v) = v$ and a direction $d$ at a periodic point $v$ \textbf{periodic} if $Dg^k(d)=d$ for some $k>0$. We call an unordered pair of directions $\{d_i, d_j\}$, based at the same point, a \textbf{turn}.

We call a locally injective path \textbf{tight}.
Recall from \cite{bh92} that a nontrivial tight path $\rho$ in $\Gamma$ is a \textbf{periodic Nielsen path (PNP)} for $g$ if $g^k(\rho)\simeq\rho$ rel endpoints for some $k\in\ZZ_{>0}$, a \textbf{Nielsen path (NP)} if the period is one, and an \textbf{indivisible Nielsen path (iNP)} if it further cannot be written as a concatenation $\rho = \rho_1\rho_2$, where $\rho_1$ and $\rho_2$ are also NPs for $g$. \cite{bh92} gives an algorithm for finding a representative with the minimal number of Nielsen paths, such a representative is called a \textbf{stable} representative.

As in \cite{hm11}, we call a periodic point $v \in \Gamma$ \textbf{principal} that either has at least three periodic directions or is an endpoint of a periodic Nielsen path. A train track representative is called \textbf{rotationless} if every principal point is fixed and every periodic direction at each principal point is fixed. 
We use from \cite[Corollary 4.43]{fh11} that rotationless powers exist, depend only on the rank $r$, and fix all PNPs.

\subsection{Attracting tree $T_+$ and lamination $\Lambda$} \label{subsec:attractingtreeandlam}

Each fully irreducible $\vphi\in\out$ acts on $\overline{CV_r}$ with an attracting tree $\Tv \in \partial CV_r$ and a repelling tree $T_{\vphi}^{-} \in \partial CV_r$ \cite{ll03}. Dual to $\Tv$ and $T_{\vphi}^{-}$, we have the repelling and attracting laminations respectively. 
In this paper we only concern ourselves with the attracting tree and lamination, which we thus write succinctly as $T_+$ and $\Lambda$ when $\vphi$ is clear. In this subsection, we provide a description of these objects in terms of train track representatives.

\begin{constr}[Attracting tree $T_+$] \label{constr:attractingtree}
Let $g\colon \Gamma \to \Gamma$ be a train track representative of $\varphi$ and $\widetilde{\Gamma}$ the universal cover of $\Gamma$ equipped with a distance function $\widetilde{d}$ lifted from $\Gamma$. Then $\pi_1(\G)\cong F_r$ acts by deck transformations, hence isometries, on $\widetilde{\Gamma}$. A lift $\widetilde{g}$ of $g$ corresponds to a unique automorphism $\Phi$ representing $\varphi$. For each $w \in F_r$ and $x \in \widetilde{\Gamma}$, we have $\Phi(w)\widetilde{g}(x)=\widetilde{g}(wx)$. Define the pseudo-distance $d_{\infty}$, 
for each $x,y \in \widetilde{\Gamma}$, by
$d_{\infty}(x,y)=\lim_{k \to \infty}\frac{1}{\lambda^k}d(\widetilde{g}^k(x),\widetilde{g}^k(y))$. Then $T_+$ is the quotient of $\widetilde{\Gamma}$ under $x\sim y$ when $d_{\infty}(x,y)=0$.
\end{constr}

To describe the attracting lamination we need the following:
Let $\Gamma$ be a marked graph with universal cover $\widetilde{\Gamma}$ and projection map $p \colon \widetilde{\Gamma} \to \Gamma$. 
By a \textbf{line} in $\widetilde{\Gamma}$ we mean a proper embedding of the real line $\widetilde{\lambda} \colon \mathbb{R} \to \widetilde{\Gamma}$, modulo reparametrization. 
We denote by $\widetilde{\mathcal{B}}(\Gamma)$ the space of lines in $\widetilde{\Gamma}$ with the compact-open topology (generated by the open sets $\widetilde{\mathcal{U}}(\widetilde{\gamma}):=\{L \in \widetilde{\mathcal{B}}(\Gamma) \mid \widetilde{\gamma}$ is a finite subpath of $L\}$). 
A \textbf{line} in $\Gamma$ is then the projection $p \circ \widetilde{\lambda} \colon \mathbb{R} \to \Gamma$ of a line $\widetilde{\lambda}$ in $\widetilde{\Gamma}$.
We denote by $\mathcal{B}(\Gamma)$ the space of lines in $\Gamma$ with the quotient topology induced by the natural projection map from $\widetilde{\mathcal{B}}(\Gamma)$ to $\mathcal{B}(\Gamma)$. One can then define $\mathcal{U}(\gamma):=\{L \in \mathcal{B}(\Gamma) \mid \gamma$ is a finite subpath of $L\}$, these sets generate the topology on $\mathcal{B}$. 
For a marked graph $\Gamma$, a line $\gamma$ in $\Gamma$ is \textbf{birecurrent} if each finite subpath of $\gamma$ occurs infinitely often as an unoriented subpath in each end of $\gamma$.

\begin{constr}[Attracting lamination $\Lambda$] \label{constr:attractinglam}
Fix a fully irreducible $\phi \in Out(F_r)$ and any train track representative $g \colon \Gamma \to \Gamma$ for $\varphi$. Given any edge $e$ in $\Gamma$, there is a $k>0$ such that the following is a sequence of nested open sets:
$\mathcal{U}(e) \supset \mathcal{U}(g^k(e)) \supset \mathcal{U}(g^{2k}(e)) \dots$
The \textbf{attracting lamination} $\Lambda$ is the set of birecurrent lines in the intersection. We may use the same notation for the total lift $\widetilde{\Lambda}$ of $\Lambda$ to the universal cover.
\end{constr}

\begin{rmk}[Viewing $\Lambda$ in trees $T\in\os$]\label{r:ViewingLambda}
The definition of $\Lambda$ is well-defined, independent of the choice of train track representative; see \cite[Lemma 1.12]{bfh97} for proof. Once a basepoint lift is chosen in $\widetilde{\Gamma}$, one can identify $\partial \widetilde{\Gamma}$ with the hyperbolic boundary $\partial F_r$ of $F_r$. 
This allows identification of $\widetilde{\Lambda}$ with a set of unordered pairs of points in $\partial F_r$, so $\widetilde{\Lambda}$ is also well-defined. Then define the realization of $\Lambda$ in a general point of $CV_r$ represented by a marked graph $\Gamma'$ with universal cover $\widetilde{\Gamma'}$ and a chosen basepoint in $\widetilde{\Gamma'}$:
Use the identifications $\partial \widetilde{\Gamma} \cong \partial F_r \cong \partial \widetilde{\Gamma'}$, to obtain $\widetilde{\mathcal{B}}(\Gamma) \cong \widetilde{\mathcal{B}}(\Gamma')$, identifying $\widetilde{\Lambda} \subset \widetilde{\mathcal{B}}(\Gamma)$ with a subset of $\widetilde{\mathcal{B}}(\Gamma')$ called the \textbf{realization} of $\widetilde \Lambda$ in $\widetilde \Gamma'$. Via the projection $\widetilde{\mathcal{B}}(\Gamma') \to \mathcal{B}(\Gamma')$, we obtain the \textbf{realization} of $\widetilde{\Lambda}$ in $\Gamma'$.
\end{rmk}

\subsection{Folds and splits} \label{subsec:foldsandsplits}

Throughout this subsection, $I$ will denote a (possibly infinite) interval.

A \textbf{fold path} in $CV_r$ is a continuous, injective, proper function $\mF\from I \to CV_r$ defined by\\
1. a continuous 1-parameter family of marked graphs $t \to \Gamma_t$ and \\
2. a family of homotopy equivalences $h_{ts} \colon \Gamma_s \to \Gamma_t$ defined for $s \leq t \in I$, each marking-preserving,
satisfying:\\ 
\indent $\bullet$ \textbf{train track property}: $h_{ts}$ is a local isometry on each edge for all $s \leq t \in I$ and\\
\indent $\bullet$ \textbf{semiflow property}: $h_{ut} \circ h_{ts} = h_{us}$ for all $s \leq t \leq u \in I$ and $h_{ss} \colon \Gamma_s \to \Gamma_s$ is the identity for each $s$.

We call $\mF$ a \textbf{fold line} when $I=\RR$. When $I=[a,b]$ for some $a<b$ and  $h_{sa}(\G_a)$ is homeomorphic to $h_{ta}(\G_a)$ for each $a<s<t<b$, we call $\mF$ a \textbf{fold}.
In an abuse of notation, one sometimes uses the same terminology for the quotient map $h_{ba}$ as for $\mF$. A fold with only 2 edges in the support is \textbf{simple}.

For a free, simplicial $F_r$-tree $T$, a \textbf{$\Lambda$-isometry} on $T$ is an $F_r$-equivariant map $\ft \colon T \to T_+$ such that, for each leaf $L$ of $\Lambda$ realized in $T$, the restriction of $\ft$ to $L$ is an isometry onto a bi-infinite geodesic in $T_+$. Since $\ft$ is continuous, there is a well-defined map of directions $D\ft$, with a restriction $D_p\ft$ to the set of directions at $p$ for each $p\in T$.
A fold $\mF$ is \textbf{$\Lambda$-legal} if $h_{ts}$ is a $\Lambda$-isometry for each $t,s$.
Note that a $\Lambda$-legal fold cannot identify the directions in a $\Lambda$-legal turn.

Viewed as a quotient map, a fold induces a map on directions and a gate structure in which \textbf{gates} are defined as equivalence classes of directions identified by the fold.
We call this structure \textbf{weighted}, as each turn $\{d_1,d_2\}$ in a gate has an associated \textbf{length} $\ell(\{d_1,d_2\})$ that is the length of the initial segments of the edges $e_1$ and $e_2$, in the directions $d_1$ and $d_2$, identified by the fold. 
In particular $\ell(\{d_1,d_2\})\leq \ell(e_1),\ell(e_2)$.

\begin{lem} \label{lem:foldlengthconstraint}
Suppose that $\ft$ is a $\Lambda$-isometry.
Let $d_0,...,d_m$ be directions in a common gate. Then
\begin{equation} \label{eq:foldlengthconstraint}
\min_{i=1,...,m} \{\ell(\{d_{i-1},d_i\}) \} \leq \ell(\{d_0,d_m\}).
\end{equation}
\end{lem}

\begin{proof}
By induction, it suffices to show the lemma when $m=2$. In this case, the initial segments of $d_0$ and $d_1$ of length $\ell(\{d_0,d_1\}) \leq \min\{ \ell(\{d_0,d_1\}), \ell(\{d_1,d_2\}) \}$ have the same $\ft$-image, and the initial segments of $d_1$ and $d_2$ of length $\ell(\{d_1,d_2\}) \leq \min\{ \ell(\{d_0,d_1\}), \ell(\{d_1,d_2\}) \}$ have the same $\ft$-image. Hence, the initial segments of $d_0$ and $d_2$ of at least length $\min\{ \ell(\{d_0,d_1\}), \ell(\{d_1,d_2\}) \}$ have the same $\ft$-image.
\end{proof}

We call a continuous, injective, proper function $\mF'\from I \to CV_r$ defined by a continuous 1-parameter family of marked graphs $t \to \Gamma_t$ a \textbf{split path} if $\mF(t)=\mF'(a+b-t)$ is a fold path. We call $\mF'$ a \textbf{split} if $\mF$ is a fold, \textbf{$\Lambda$-legal} if $\mF$ is, and \textbf{simple} if $\mF$ is. Note that, since a $\Lambda$-legal fold must be a $\Lambda$-isometry, a $\Lambda$-legal split must also restrict to an isometry on each leaf $L$ of $\Lambda$.

\subsection{Weak train tracks} \label{subsec:weaktraintrack}

A \textbf{normalized weak train track} for $\varphi$ is a $T \in \widehat{CV_r}$ on which a $\Lambda$-isometry exists. A \textbf{weak train track} is an element of $CV_r$ represented by a normalized weak train track.

As explained in \cite[Theorem 5.8]{hm11}, the choices of a $\Lambda$-isometry on a fixed normalized weak train track $T$ can be nonunique, and are parametrized by a closed interval (that is possibly a single point).
In \cite[\S 6.2]{hm11}, Handel and Mosher define a \textbf{right-most isometry} $k_T^+$ and \cite[Lemma 6.3]{hm11} provides that this assignment of $k_T^+$ is continuous (further explanation is provided in the remark following the lemma in \cite{hm11}).
In this paper, we always equip a normalized weak train track with its right-most $\Lambda$-isometry.

Having a canonical choice of a $\Lambda$-isometry ensures having a unique fold between elements of the axis bundle where one exists.

\begin{lem} \label{lem:axisbundleinteriorrightmostfold}
Let $T, T'$ be normalized weak train tracks. If there is a fold $f:T \to T'$ such that $k_{T'}^+ f = k_T^+$, then $f$ is unique.
\end{lem}

\begin{proof}
Suppose $f':T \to T'$ is another fold such that $k_{T'}^+ f' = k_T^+$. 
Since $k_T^+$ and $k_{T'}^+$ are $\Lambda$-isometries, $f$ and $f'$ must be $\Lambda$-legal.
This implies that for each leaf $L$ of $\Lambda$, we have $f(L) = f'(L) = L$, as realized in $T'$.

Now let $x$ be a point in $T$ and $L$ a leaf of $\Lambda$ passing through $x$. Then $f(x)$ and $f'(x)$ lie on the realization of $L$ in $T'$. But $f(x)$ and $f'(x)$ map to the same point $k_T^+(x)$ under $k_{T'}^+$, so we must have $f(x) = f'(x)$ otherwise $k_{T'}^+$ would not be a $\Lambda$-isometry.
\end{proof}

As with folds, there is a \textbf{weighted induced gate structure} on $T$ for each $\Lambda$-isometry $\ft \colon T \to T_+$. 

\begin{df}[Fully preprincipal]
A (normalized) weak train track is \textbf{fully preprincipal} if, in the gate structure induced by its rightmost $\Lambda$-isometry, each vertex has $\geq 3$ gates. 
This generalizes the \cite{PffAutomata} notion, leaving out PNPs restrictions.
\end{df}

\subsection{Axis bundle} \label{subsec:axisbundle}

The axis bundle $\mathcal{A}_{\varphi}$ for a nongeometric fully irreducible $\varphi \in Out(F_r)$ was first introduced in \cite{hm11} via 3 equivalent definitions  (proof of their equivalence is in \cite[Theorem 1.1]{hm11}). Further description is given in \cite{loneaxes}. We give here the one definition we use: Fix a normalization of $T_{+}$. Then:
$$\widehat{\mathcal{A}_{\varphi}} = \{\text{free simplicial } F_r \text{-trees } T \in \widehat{CV_r} \mid \exists~\Lambda\text{-isometry } f_T \colon T \to T_{+}\}.$$
In other words, $\widehat{\mathcal{A}_{\varphi}}$ is the set of normalized weak train tracks in $\widehat{CV_r}$. The \textbf{axis bundle} $\mathcal{A}_{\varphi}$ is the set of weak train tracks in $CV_r$ for $\varphi$, i.e. $\mathcal{A}_{\varphi}$ is the image of $\widehat{\mathcal{A}_{\varphi}}$ under the projectivization of $\widehat{CV_r}$.

By \cite[Lemma 5.1]{hm11}, each weak train track in $CV_r$ is represented by a unique 	normalized weak train track in $\widehat{CV_r}$; equivalently, the projection $\widehat{CV_r} \to CV_r$ restricts to a bijection $\widehat{\Av} \to \Av$. As such, we may occasionally blur the distinction between weak train tracks and normalized weak train tracks.

We note that $\mathcal{A}_{\varphi}$ is also the union of the images of all fold lines $\mathcal{F} \colon \mathbb{R} \to CV_r$ such that $\mathcal{F}$(t) converges in $\overline{CV_r}$ to $T_{-}^{\varphi}$ as $t \to -\infty$ and to $T_{+}^{\varphi}$ as $t \to +\infty$. An important example of such a fold line is a periodic fold line for a \textbf{Stallings fold decomposition} of a train track representative $g\from \Gamma  \to \Gamma$: At an illegal turn for $g$, fold two maximal initial segments with the same image to obtain a map $\mathfrak{g}_1 \from \Gamma_1 \to \G$ of the quotient graph $\Gamma_1$. Repeat the process for $\mathfrak{g}_1$ and recursively. If some $\mathfrak{g}_k \from \Gamma_{k-1} \to \Gamma$ has no illegal turn, then $\mathfrak{g}_k$ is a homeomorphism and the fold sequence is complete. Taking a rotationless power avoids the homeomorphism.

Several crucial properties of the axis bundle are recorded in \cite[Theorem 6.1, Lemma 6.2]{hm11}. We summarize a few here as \Cref{P:6.2}. 

\begin{prop}[\cite{hm11}] \label{P:6.2} 
Let $\varphi \in \out$ be nongeometric fully irreducible. 
Then the map $\vol \colon \widehat{\mathcal{A}_{\varphi}} \to (0, \infty)$ is a surjective, $\varphi$-equivariant homotopy equivalence, where $\varphi$ acts on $(0, +\infty)$ by multiplication by $\frac{1}{\lambda}$.
\end{prop}

Further, the $\vphi$-action gives a means to decompose $\Av$ into compact fundamental domains:

\begin{lem} \label{lem:fpplocallyfinite}
Suppose that $\vphi\in\out$ is nongeometric fully irreducible. Then each fundamental domain of the $\vphi$-action on $\Av$ contains only finitely many fully preprincipal train tracks.
\end{lem}

\begin{proof}
First note that, since the fundamental domain is compact, it can intersect only finitely many simplices. So it suffices to show that each simplex contains only finitely many preprincipal train tracks of $\vphi$.

Since each vertex of each preprincipal train track $T$ contains $\geq 3$ gates, each vertex must be mapped by the $\Lambda$-isometry $\ft$ to a branch point. Further, $\ft$, as a $\Lambda$-isometry, is an isometry on each edge. Since $T_+$ has only finitely many orbits of branch points, this gives a finite list of possible edge lengths in $T$. Together with the domain containing only finitely many simplices, this gives only finitely many possibilities for $T$.
\end{proof}

We use the following construction, allowing us to connect the axis bundle to train track representatives. 

\begin{constr}[Train tracks] \label{constr:tttoweaktt}
Weak train tracks can be constructed from train track representatives: Let $g \colon \Gamma \to \Gamma$ be a train track representative of a nongeometric fully irreducible $\varphi \in Out(F_r)$. 
Recall \Cref{constr:attractingtree} and let $T_k$ denote the simplicial $F_r$-tree obtained from $\widetilde{\Gamma}$ by assigning the metric $d_{k}(x,y)=\frac{1}{\lambda^k}d(\widetilde{g}^k(x),\widetilde{g}^k(y))$, identifying each $x,y \in \widetilde{\Gamma}$ with $d_k(x,y)=0$, and then equipping the quotient graph with the metric induced by $d_k$. Then, for each $i$, a basepoint-preserving lift of $g$ induces a basepoint-preserving $F_r$-equivariant map $\widetilde{g}_{i+1,i} \colon T_i \to T_{i+1}$ restricting to an isometry on each edge. Define a direct system $\widetilde{g}_{j,i} \colon T_i \to T_j$ inductively by $\widetilde{g}_{j,i}=\widetilde{g}_{j,j-1} \circ \widetilde{g}_{j-1,i}$. Then $\widetilde{\Gamma}$ is a weak train track where the $\Lambda$-isometry $g_{\infty} \colon \widetilde{\Gamma} \to T_+$ is the direct limit map. $\widetilde{\Gamma}$ is called a \textbf{train track}. $TT(\vphi)$ denotes the set of train tracks for $\vphi$.

\begin{figure}[ht!]
\[
\xymatrix{\widetilde{\Gamma}=T_0 \ar[r]_{\widetilde{g}_{1,0}} \ar@/^4pc/[rrrr]_{g_{\infty}} & T_1 \ar[r]_{\widetilde{g}_{2,1}} \ar@/^3pc/[rrr]_{g_{1,\infty}} & T_2 \ar[r]_{\widetilde{g}_{3,2}} \ar@/^2pc/[rr]_{g_{2,\infty}} & \dots   & T_+ \\}
\]
\end{figure}
\end{constr}

\subsection{Ideal Whitehead graphs} \label{subsec:IWG}

The ideal Whitehead graph of a nongeometric fully irreducible $\vphi\in\out$ is first defined in \cite{hm11}. One can reference \cite{Thesis} and \cite{hm11} for alternative definitions and its outer automorphism invariance. Further explanation yet can be found in \cite[\S 2.8, \S 2.10]{loneaxes}.

\begin{df}[Ideal Whitehead graph $IW(\varphi)$]\label{d:IWG}
Let $\varphi \in Out(F_r)$ be nongeometric fully irreducible with lifted attracting lamination $\widetilde\Lambda$ realized in $T_+$.
$\widetilde{IW(\varphi)}$ is the union of the components with at least three vertices of the graph that has a vertex for each distinct leaf endpoint and an edge connecting the vertices for the endpoints of each leaf. $F_r$ acts freely, properly discontinuously, and cocompactly in such a way that the restriction to each component of $\widetilde{{IW}(\varphi)}$ has trivial stabilizer. The \textbf{ideal Whitehead graph} ${IW}(\varphi)$ is the quotient under this action. It has only finitely many components.

Using \Cref{r:ViewingLambda}, one can view $\widetilde{IW(\varphi)}$ in any $F_r$-tree $T \in TT(\varphi^k)$.
\end{df}

\begin{df}[Principal points]\label{d:Principal}
Given a branch point $b$ of $T_+$, the lifted ideal Whitehead graph $\widetilde{{IW}}(\varphi)$ has one component, denoted $\widetilde{IW_b}(\varphi)$, whose edges, realized as lines in $T_+$, all contain $b$. This relationship gives a one-to-one correspondence between components of $\widetilde{{IW}}(\varphi)$ and branch points of $T_+$. Given a branch point $b$ of $T_+$, let $\widetilde{{IW}_b}(\varphi;T)$ denote the realization of $\widetilde{{IW}_b}(\varphi)$ in $T$. This makes sense by viewing the ideal Whitehead graph in terms of the lamination leaves, as in \Cref{d:IWG}. We call a point $v$ in $T$ \textbf{principal} for $f$ if there exists a branch point $b$ of $T_+$ such that $f(v)=b$ and $v$ lies in some leaf of $\widetilde{{IW}_b}(\varphi;T)$.

It is shown in \cite{hm11}, and can be ascertained from the alternative $IW(\vphi)$ definitions given there (and in \cite{Thesis}) that a principal point downstairs either has 3 periodic directions or is the endpoint of a PNP.
\end{df}

\subsection{Cut Decompositions}{\label{subsec:cutdecomp}}

Suppose $\vphi\in\out$ is nongeometric fully irreducible. We describe here a methodology of Handel and Mosher \cite[\S 4]{hm11} for using cut vertices in $IW(\vphi)$ to obtain train track representatives with varied numbers of PNPs. The methodology is used in \Cref{p:CutDecompositions} to construct fully preprincipal train track representatives of $\vphi$ that realize all possible ``cut decompositions'' of $IW(\vphi)$.

\parpic[r]{\selectfont\fontsize{12pt}{12pt} \resizebox{!}{3.75cm}{
\begingroup%
  \makeatletter%
  \providecommand\color[2][]{%
    \errmessage{(Inkscape) Color is used for the text in Inkscape, but the package 'color.sty' is not loaded}%
    \renewcommand\color[2][]{}%
  }%
  \providecommand\transparent[1]{%
    \errmessage{(Inkscape) Transparency is used (non-zero) for the text in Inkscape, but the package 'transparent.sty' is not loaded}%
    \renewcommand\transparent[1]{}%
  }%
  \providecommand\rotatebox[2]{#2}%
  \newcommand*\fsize{\dimexpr\f@size pt\relax}%
  \newcommand*\lineheight[1]{\fontsize{\fsize}{#1\fsize}\selectfont}%
  \ifx\svgwidth\undefined%
    \setlength{\unitlength}{168.46702636bp}%
    \ifx\svgscale\undefined%
      \relax%
    \else%
      \setlength{\unitlength}{\unitlength * \real{\svgscale}}%
    \fi%
  \else%
    \setlength{\unitlength}{\svgwidth}%
  \fi%
  \global\let\svgwidth\undefined%
  \global\let\svgscale\undefined%
  \makeatother%
  \begin{picture}(1,0.88737131)%
    \lineheight{1}%
    \setlength\tabcolsep{0pt}%
    \put(0,0){\includegraphics[width=\unitlength,page=1]{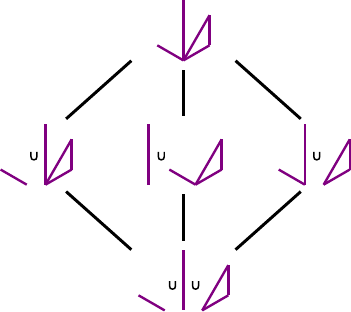}}%
  \end{picture}%
\endgroup%
}}
Refer to the right-hand image. Suppose $G$ is a graph that is a union of two nontrivial subgraphs of itself intersecting in a single vertex $v\in VG$, then we call $v$ a \textbf{cut vertex} of $G$. By a \textbf{cut decomposition} of $G$, we mean a collection of nontrivial connected subgraphs $\{G_1,\dots, G_k\}$ of $G$ satisfying:
\begin{enumerate}
\item $G=\cup ~G_j$ and
\item $G_i\cap G_j$ is either empty or a vertex for each $i\neq j$.
\end{enumerate}
As an example, we included the figure to the right, displaying the cut decompositions of the graph located at the top. The cut decompositions are arranged in the partial order of fineness.


The ideal Whitehead graph has another interpretation in terms of singular leaves of $\widetilde{\Lambda}$.
Here, a leaf of $\widetilde{\Lambda}$ is \textbf{singular} if it shares a ray with another leaf. 

\begin{df}[${LW}(\widetilde{v},T)$, ${SW}(\widetilde{v},T)$]{\label{d:WhiteheadGraphsWTT}}
Let $T$ be a weak train track and $\widetilde{v}$ a principal point of $T$. 
The \textbf{local Whitehead graph} ${LW}(\widetilde{v};T)$ has a vertex for each direction at $\widetilde{v}$ and an edge connecting the vertices representing a pair of directions $\{d_1,d_2\}$ precisely when the turn $\{d_1,d_2\}$ is taken by the realization in $T$ of a leaf of $\widetilde{\Lambda}$. The \textbf{stable Whitehead graph} ${SW}(\widetilde{v};T)$ at $\widetilde{v}$ is the subgraph of ${LW}(\widetilde{v};T)$ obtained by restricting to the \textbf{principal directions}, i.e. those containing singular rays emanating from $\widetilde{v}$. 
\end{df}

Each ${SW}(\widetilde{v};T)$ sits inside $\widetilde{{IW}}(\varphi)$ as follows: A vertex of ${SW}(\widetilde{v};T)$ corresponds to a singular leaf ray $\widetilde{R}$ emanating from $\widetilde{v}$. The endpoint of this ray corresponds to a vertex of $\widetilde{{IW}}(\varphi)$. 
An edge of ${SW}(\widetilde{v};T)$ corresponds to a singular leaf based at $\widetilde{v}$ (as in \Cref{d:Principal}). This leaf also gives an edge of $\widetilde{{IW}}(\varphi)$. 

The following is a restatement of \cite[Lemma 5.2]{hm11}, focused for our purposes.

\begin{lem}[\cite{hm11}] \label{lemma:SWGinIWG}
Suppose that $T$ is a weak train track and $F_T \colon T \to T_+$ a $\Lambda$-isometry. Suppose that $b$ is a branch point of $T_+$ and $\{\wtilde{w_i}\} \subset T$ is the set of principal vertices mapped by $F_T$ to $b$. Then
\begin{enumerate}
\item $\widetilde{{IW}_b}(\varphi;T) = \cup {SW}(\wtilde{w_i};T)$.
\item For each $i \neq j$, there is at most one vertex in ${SW}(\wtilde{w_i};T) \cap {SW}(\wtilde{w_j};T)$. If there is a vertex $P$ in the intersection, then $P$ is a cut point of $\widetilde{{IW}}(\varphi)$, separating ${SW}(\wtilde{w_i};T)$ from ${SW}(\wtilde{w_j};T)$ in $\widetilde{IW(\varphi)}$.
\end{enumerate}
\end{lem}

\begin{df}[Local decomposition]
The cut decomposition in \Cref{lemma:SWGinIWG} is the \textbf{local decomposition} of $T$. Let $T,T'$ be weak train tracks. As in \cite{hm11}, one says $T$ \textbf{is split at least as much as} $T'$ if the local decomposition $\widetilde{{IW}}(\varphi) = \bigcup {SW}(v_j;T)$ is at least as fine as the local decomposition $\widetilde{{IW}}(\varphi) = \bigcup {SW}(w_i;T')$. That is, for each principal vertex $v_j$ of $T$, there exists a principal vertex $w_i$ of $T'$ such that ${SW}(v_j;T) \subset {SW}(w_i;T)$, where the inclusion takes place in $\widetilde{{IW}}(\varphi)$, realized as a decomposition, as above.
\end{df}

Suppose $\vphi\in\out$ is a rotationless nongeometric fully irreducible outer automorphism and $\mG$ is a cut decomposition of $IW(\vphi)$. 
\cite[Lemma 4.3]{hm11} provides a method to obtain a train track representative $g$ of $\vphi$ so that $\widetilde{\Gamma}$ has local decomposition $\mG$. 
More specifically, $g$ is obtained from a stable train track representative $h$ of $\vphi$ via iteratively ``splitting open'' at cut vertices of the stable Whitehead graphs:

\begin{constr}[Splitting open a cut vertex] \label{constr:splittingcutvertices}
Suppose $w$ is a cut vertex of a stable Whitehead graph $SW(f,u)$ of a train track representative $f\from\G\to\G$ of $\vphi$ and $G_1$, $G_2$ are nontrivial subgraphs of $SW(f,u)$ meeting only at $w$ and with $SW(f,u) = G_1 \cup G_2$.
Then $w$ is represented by a fixed direction $d_0$ at $u$, as well as a collection $d_1, \dots, d_N$ of directions mapped to $d_0$ by powers of $Df$. Let $E_0$ be the edge in the direction of $d_0$. 
See \Cref{fig:splittingcutvertex} left. In the top row of the figure, we have drawn the graph $\Gamma$. In the bottom row, we draw a `blown-up' view where we insert the local Whitehead graph ${LW}(v,T)$ at each vertex $v$, with the purple edges lying in ${SW}(v,T)$ and the red edges lying in ${LW}(v,T)$ but not ${SW}(v,T)$, as done in \cite{PffAutomata}.

We explain how to form an NP via splitting open at $E_0$. The same procedure should simultaneously be applied to the edge in the direction of each of $d_1, \dots, d_N$. 

$G_1$ and $G_2$ correspond to a bipartition of $\mD(u)\backslash\{d_0\}$ satisfying that the directions of each $f$-taken turn are in the same partition element.
Create from $\G$ a new graph $\G'$ where
\begin{itemize}
    \item $V\G' = (V\G\backslash\{u\}) \cup \{u_1, u_2\}$, and
    \item each edge $e_j= [v_j',u]$ is replaced with $[v_j',u_1]$, and
    \item each edge $e_j'= [v_j'',u]$ is replaced with $[v_j',u_2]$, and
    \item $E_0= [u,v]$ is replaced by the 2 edges $[u_1,v]$ and $[u_2,v]$, and
    \item all other edges remain the same.
\end{itemize}
See the middle image of \Cref{fig:splittingcutvertex}.

The map $f'$ is the same as $f$ but that $E_0$ is replaced by $E_1$ in the image of any edge when $u$ was passed through via $G_1$ and by $E_2$ when $u$ was passed through via $G_2$, and analogous alterations occur for the edges in the directions of $d_1, \dots, d_N$. The images of $E_1$ and $E_2$ are that of $E$, but that the image of $E_1$ now starts with $E_1$ and of $E_2$ with $E_2$. Note that $E_0$ was a fixed direction, as it was represented by a vertex in $SW(f,u)$, so this map of the $E_j$ makes sense. In the cases of $d_1, \dots, d_N$, instead of the image of $E_1, E_2$ starting with respectively $E_1, E_2$, the $f$-image of $e$ would start with $Df(e)$.

\begin{figure}
    \centering
    \selectfont\fontsize{6pt}{6pt}
    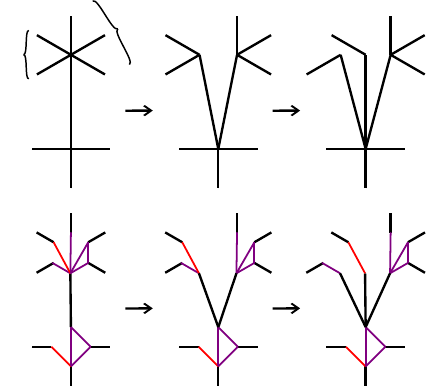
    \caption{Splitting open a cut vertex. In the top row, we have drawn the graph $\Gamma$. In the bottom row, we draw a `blown-up' view where we insert the local Whitehead graph ${LW}(v,T)$ at each vertex $v$, with the purple edges lying in the stable Whitehead graph ${SW}(v,T)$ and the red edges lying in ${LW}(v,T)$ but not ${SW}(v,T)$.}
    \label{fig:splittingcutvertex}
\end{figure}

We call this procedure \textbf{splitting $u$ open along $E$}.
\end{constr}

\begin{lem}\label{l:SplittingCutVertices} Suppose $g$ is a fully preprincipal train track representative of a rotationless nongeometric fully irreducible $\vphi\in\out$. Then splitting open at a cut vertex yields a train track representative of $\vphi$ that can be $\Lambda$-legal split to obtain a fully preprincipal train track representative of $\vphi$.
\end{lem}

\begin{proof}
Suppose $g$ is fully preprincipal and $u$ was split open along $E=[u,v]$. Note that $v$ still has $\geq 3$ gates. Our concern is that either $u_1$ or $u_2$ has only 2 gates. Without loss of generality we assume $u_1$ has only 2 gates and that the directed edges at $u$ are $E_1, e_1, \dots, e_n$. Then the only possible $\Lambda$-taken turns at $u$ are of the form $\{E_1, e_j\}$. So it is possible to split open $u_1$ to replace $\{E_1, e_1, \dots, e_n\}$ with $\{\bar{E_1}e_1, \dots, \bar{E_1}e_n\}$. 
See \Cref{fig:splittingcutvertex} right.
The process can be repeated at $u_2$, if necessary.
\end{proof}

\begin{prop}\label{p:CutDecompositions}
Suppose $\vphi\in\out$ is rotationless nongeometric fully irreducible. Then each cut decomposition of $IW(\vphi)$ is realized by a fully preprincipal train track representative of $\vphi$. 
\end{prop}

\begin{proof}
By \cite[Proposition 7.1]{PffAutomata}, $\vphi$ has a fully preprincipal train track representative with the minimal number of PNPs (this is explicitly stated in the ageometric case, but follows the same argumentation in the parageometric case). Repeated application of \Cref{l:SplittingCutVertices} yields a train track representative whose stable Whitehead graphs give the desired cut decomposition.
\end{proof}

The stable axis bundle was introduced in \cite[\S 6.5]{hm11} as an object of interest and was used extensively in \cite{loneaxes}. We expand upon the notion of the stable axis bundle to define an axis bundle for each cut decomposition of the ideal Whitehead graph.

\begin{df}[$\mG$-axis bundle $\mG\Av$]
Suppose $\vphi\in\out$ is nongeometric fully irreducible and $\mG$ is a cut decomposition of $IW(\vphi)$.
The \textbf{$\mG$-axis bundle} $\mG\Av$ is the set of all weak train tracks whose local decomposition is at most as coarse as $\mG$.
Under this terminology, the \textbf{stable axis bundle} is the $\mG$-axis bundle where $\mG$ is the coarsest possible local decomposition $\{IW(\vphi)\}$.
\end{df}

The partial ordering of fineness plays an important role in \cite[Proposition 5.4]{hm11}, which we record here as \Cref{P:5.4}. We use this proposition to ensure each weak train track is contained in the branched cube determined by a fully preprincipal train track with local decomposition as split as its own (\Cref{prop:branchedcubescoveraxisbundle}).

\begin{prop}[\cite{hm11}]\label{P:5.4} 
Let $\varphi \in Out(F_r)$ be nongeometric fully irreducible. Then for any train track representative $g \colon \Gamma \to \Gamma$ for $\varphi$ with associated $\Lambda$-isometry $g_{\infty} \colon \widetilde{\Gamma} \to T_+$, there exists an $\varepsilon > 0$ satisfying that, if $f \colon T \to T_+$ is any $\Lambda$-isometry, if $g_{\infty}$ splits at least as much as $f$, and if $Len(T) \leq \varepsilon$, then there exists a unique equivariant edge-isometry $h \colon \widetilde{\Gamma} \to T$ such that $g_{\infty} = f \circ h$. Moreover, $h$ is a $\Lambda$-isometry. 
\end{prop}

\begin{prop} \label{prop:stableaxisbundleconnected}
Let $\vphi\in\out$ be nongeometric fully irreducible and let $\mG$ be a cut decomposition of $IW(\vphi)$.
Then the $\mG$-axis bundle $\mG\Av$ is connected.
\end{prop}

\begin{proof}
Suppose $S_1, S_2 \in \mG\Av$.
By \Cref{p:CutDecompositions}, there is a fully preprincipal $T \in \Av$ with local decomposition $\mG$.
Using \Cref{P:6.2}, we shift $S_1$ and $S_2$ so that the $\veps$ requirement in \Cref{P:5.4} is satisfied and then use \Cref{P:5.4} to know that both $S_1$ and $S_2$ can be obtained from $T$ by folding. In other words, both $S_1$ and $S_2$ can be connected to $T$ by a fold path, thus by a path.
\end{proof}

\section{Branched cubes in the axis bundle} \label{sec:axisbundlebranchedcubes}

In this section we start building  a cubist complex structure on the axis bundle $\mathcal{A}_\phi$ by describing the branched cubes. 
We then show some combinatorial properties of the interaction between these branched cubes.
Unless otherwise indicated, we assume throughout this section that $\vphi\in\out$ is nongeometric fully irreducible, $T \in \Av$ is a fully preprincipal weak train track endowed with the weighted gate structure induced by its rightmost $\Lambda$-isometry $\ft \from T \to T_+$, and $\mathcal{T}_0$ is the set of illegal turns in $T$.

\subsection{Description of the branched cubes} \label{subsec:axisbundlebranchedcubesdefn}

For each $(\ell_\tau) \in \prod_{\tau \in \mathcal{T}_0} [0,\ell(\tau)]$, let $T_{(\ell_\tau)} \in \widehat{\os}$ denote the metric graph obtained from $T$ by folding each turn $\tau \in \mathcal{T}_0$ along initial segments of length $\ell_\tau$.

\begin{df} \label{df:axisbundlebranchedcube}
Let $T \in \Av$ be fully preprincipal. The \textbf{branched cube} at $T$ is the set
$$B_T = \left\{T_{(\ell_\tau)} \mid (\ell_\tau) \in \prod_{\tau \in \mathcal{T}_0} [0,\ell(\tau)] \right\}\subset\widehat{\os}.$$
For future reference, note that for each $T_{(\ell_\tau)} \in B_T$ there is a canonical fold $h:T \to T_{(\ell_\tau)}$.
\end{df}

\begin{lem} \label{lem:branchedcubesubsetaxisbundle}
Each branched cube $B_T$ is a subset of $\Av$. Further, if $T \in \mG\Av$, then $B_T$ is a subset of $\mG\Av$.
\end{lem}
\begin{proof}
Suppose $T' \in B_T$. For the first statement, it suffices to show there is a $\La$-isometry $F_{T'}\from T'\to T_+$. Consider the natural fold $f\from T\to T'$. Since $f$ folds along illegal turns, we know both that $f$ restricts to an isometry on each leaf of $\Lambda$ and that, if $f(x)=f(y)$, then $F_T(x)=F_T(y)$. Thus, we can define $F_{T'}$ as the quotient of $F_T$ induced by the quotient map $f$.

Now suppose in addition that $T \in \mG\Av$. The second statement follows from the fact that folding cannot increase the number of PNPs and a cut decomposition is determined by PNPs.
\end{proof}

One drawback to the notation $T_{(\ell_\tau)}$ is that we can have $T_{(\ell_\tau)} = T_{(\ell'_\tau)}$ for $(\ell_\tau) \neq (\ell'_\tau)$.
To overcome this, we define functions capturing how much folding we actually do, as opposed to how much folding we were instructed to do:
Let $\{d_1,d_2\}$ be a turn in $T$. Let $d_i$ denoted the initial direction of the oriented edge $e_i$. 
For each $T' \in B_T$, we define $x_{\{d_1,d_2\}}(T')$ to be the length of the largest initial segments in $e_1$ and $e_2$ identified by the fold $T \to T'$.
For example, if $d_1$ and $d_2$ do not lie in the same gate, then $x_{\{d_1,d_2\}}(T') = 0$.

\begin{lem}
Suppose $T\in \Av$ is fully preprincipal and $d_0,...,d_m$ is a sequence of directions in $T$. Then 
\begin{equation} \label{eq:coordinateconstraint}
\min_{i=1,...,m} \{x_{\{d_{i-1},d_i\}} \} \leq x_{\{d_0,d_m\}}.
\end{equation}
\end{lem}
\begin{proof}
The proof of this is exactly the same as that of \Cref{lem:foldlengthconstraint}.
\end{proof}

Together with a decomposition of $B_T$ into cubes, as will be described in \Cref{subsec:axisbundlecubesdefn}, the functions $x_\tau$ will serve as a kind of coordinate system.

\subsection{Examples of branched cubes} \label{subsec:branchedcubeseg}

Before analyzing the structure of the branched cubes in general, we take a moment to go through some examples. 
For simplicity, let us assume that only one gate, $G$, of $T$ has more than one direction.

\begin{ex} \label{eg:1dimbranchedcube}
Suppose $G$ has exactly 2 directions, which we denote as $d_1,d_2$.
The set of train tracks $B$ that can be obtained from $T$ by folding the unique turn $\{d_1,d_2\}$ in $G$ is homeomorphic to an interval. 
The coordinate $x_{\{d_1,d_2\}}$ on $B$ is the length folded.
\end{ex}

\begin{ex} \label{eg:2dimbranchedcube}
Suppose $G$ has exactly 3 directions, which we denote as $d_1,d_2,d_3$ (top left in the figure).
Then the set of illegal turns is $\mathcal{T}_0=\{\{d_1,d_2\},\{d_1,d_3\},\{d_2,d_3\}\}$.

\parpic[r]{\selectfont \fontsize{6pt}{6pt} 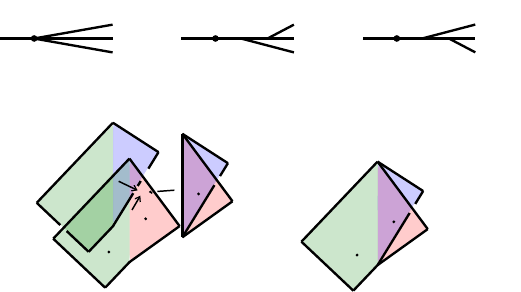}
For each subset $\mathcal{T} \subset \mathcal{T}_0$ containing at most two of the three turns, let $C_\mathcal{T}$ denote the set of train tracks obtained by folding $T$ along the turns in $\mathcal{T}$.
Then $C_{\mathcal{T}}$ is homeomorphic to a $|\mathcal{T}|$-dimensional cube for each $\mathcal{T}$.
The coordinates of the cube are the lengths $x_\tau$ we folded the turns $\tau \in \mathcal{T}$.
If $\mathcal{T}' \subset \mathcal{T}$, then $C_{\mathcal{T}'}$ is a \emph{splitting face} of $C_{\mathcal{T}}$ and consists of all points where the coordinates $x_\tau$ for $\tau \in \mathcal{T} ~\backslash~ \mathcal{T}'$ are zero.

The interaction between $C_{\mathcal{T}}$ and $C_{\mathcal{T}'}$ is more interesting when neither $\mathcal{T}$ nor $\mathcal{T}'$ is a subset of the other.
Let $T'$ be a train track obtained from $T$ by folding $\{d_1,d_2\}$ by $x_{12}$ and $\{d_2,d_3\}$ by $x_{23}$, where $x_{12} \geq x_{23}$.
Then, by definition, $T' \in C_{\{\{d_1,d_2\},\{d_2,d_3\}\}}$. 
However, $T'$ can also be obtained from $T$ by folding $\{d_1,d_2\}$ by $x_{12}$ and $\{d_1,d_3\}$ by $x_{13}$. This is because along the interval in which we are identifying $d_1$ and $d_3$, we have already identified $d_1$ and $d_2$.
Thus $T' \in C_{\{\{d_1,d_2\},\{d_1,d_3\}\}}$ as well.
Now let $T''$ be a train track obtained from $T$ by folding $\{d_1,d_2\}$ by $x_{12}$ and $\{d_2,d_3\}$ by $x_{23}$, where $x_{12} < x_{23}$.
Then $T'' \in C_{\{\{d_1,d_2\},\{d_2,d_3\}\}}$ but $T'' \not\in C_{\{\{d_1,d_2\},\{d_1,d_3\}\}}$. 
This argument reveals that $C_{\{\{d_1,d_2\},\{d_2,d_3\}\}}$ meets $C_{\{\{d_1,d_2\},\{d_1,d_3\}\}}$ along the \emph{slice} $\{(x_{12},x_{23}) \mid x_{12} \geq x_{23}\}$. 

A similar argument holds for any pair $C_{\mathcal{T}}$ and $C_{\mathcal{T}'}$ where $\mathcal{T}$ and $\mathcal{T}'$ are 2-element sets.
Thus the set of train tracks $B$ that can be obtained by folding $T$, being $C_{\{\{d_1,d_2\},\{d_1,d_3\}\}} \cup C_{\{\{d_1,d_2\},\{d_2,d_3\}\}} \cup C_{\{\{d_1,d_3\},\{d_2,d_3\}\}}$, is a branched cube as shown in the bottom right of the image accompanying this example.

\end{ex}

\begin{ex} \label{eg:higherdimbranchedcube}
In general, a subset $\mathcal{T}$ of $\mathcal{T}_0$ specifies a cube $C_\mathcal{T}$ of train tracks obtained by folding $T$ along the turns in $\mathcal{T}$ if and only if the turns in $\mathcal{T}$ can be folded independently of one another. 
For example, if $\mathcal{T}$ contains a triplet of the form $\{\{d_1,d_2\},\{d_1,d_3\},\{d_2,d_3\}\}$, then folding $\{d_1,d_2\}$ and $\{d_1,d_3\}$ will also force folding $\{d_2,d_3\}$. 
On the other hand, if $\mathcal{T}$ is of the form $\{\{d_1,d_2\},\{d_2,d_3\},\{d_3,d_4\}\}$, then independent folding is possible.
We formalize this idea later by considering the elements of $\mathcal{T}_0$ as edges of a graph on the elements of $G$.
The set of train tracks $B$ that can be obtained by folding $T$ is the union of the $C_\mathcal{T}$ as $\mathcal{T}$ ranges over such \emph{independent} subsets of $\mathcal{T}_0$.

The possible intersections between the cubes $C_{\mathcal{T}}$ also get more complicated as the number of elements in $G$ grows.
For example, suppose $\mathcal{T} = \{\{d_1,d_2\},\{d_2,d_3\},\{d_3,d_4\},\{d_4,d_5\}\}$ and $\mathcal{T}' = \{\{d_1,d_2\},\{d_3,d_4\},\{d_1,d_4\}\}$. 
Then, using the coordinates $(x_{12},x_{23},x_{34},x_{45})$ on $C_\mathcal{T}$ given by the fold lengths, we claim that $C_{\mathcal{T}}$ meets $C_{\mathcal{T}'}$ along the slice $S_{\mathcal{T},\mathcal{T}'} = \left\{(x_{12},x_{23},x_{34},x_{45}) \mid x_{23} \leq \min \{x_{12},x_{34}\} ~\&~ x_{45}=0 \right\}$.

For the $S_{\mathcal{T},\mathcal{T}'}\subseteq C_{\mathcal{T}}\cap C_{\mathcal{T}'}$ direction, consider a train track $T\in C_{\mathcal{T}}$ satisfying $x_{23} \leq \min \{x_{12},x_{34}\}$ and $x_{45}=0$. We can also reach $T$ by first folding $\{d_1,d_2\}$ for $x_{12}$ and $\{d_3,d_4\}$ for $x_{34}$, then folding $\{d_1,d_4\}$ for $x_{23}$. See \Cref{fig:eghigherdimbranchedcube}, top middle. The point is that since $x_{23} \leq \min \{x_{12},x_{34}\}$, along the folded segments of $d_1$ and $d_4$, we have $d_1$ is already folded with $d_2$, as is $d_3$ with $d_4$. So folding $\{d_1,d_4\}$ or $\{d_2,d_3\}$ is the same.

For $\supseteq$, consider a train track $T\in C_{\mathcal{T}}\cap C_{\mathcal{T}'}$. 
Then necessarily $x_{45}=0$ since the direction $d_5$ is not folded at all for train tracks in $C_{\mathcal{T}'}$.
Meanwhile, for all train tracks in $C_{\mathcal{T}'}$, the turn $\{d_2,d_3\}$ must be folded at most as much as $\{d_1,d_4\}$. Under the coordinates $(x_{12},x_{23},x_{34},x_{45})$, the amount of folding for the former is $x_{23}$ while that of the latter is $\min\{x_{12},x_{23},x_{34}\}$. 
Thus $x_{23} \leq \min\{x_{12},x_{23},x_{34}\}$, or equivalently, $x_{23} \leq \min\{x_{12},x_{34}\}$.

One can similarly verify that using the coordinates $(x_{12},x_{34},x_{14})$ on $C_{\mathcal{T}'}$ given by lengths of folding, $C_{\mathcal{T}'}$ meets $C_{\mathcal{T}}$ along the slice $S_{\mathcal{T}',\mathcal{T}} = \left\{(x_{12},x_{34},x_{14}) \mid x_{14} \leq \min \{x_{12},x_{34}\} \right\}$.

In general, the slices of intersection will be given by inequalities determined from a graph on the elements of $G$, with edges  $\mathcal{T}_0$, as described above.

\begin{figure}[H]
    \centering
    \selectfont\fontsize{7pt}{7pt}
    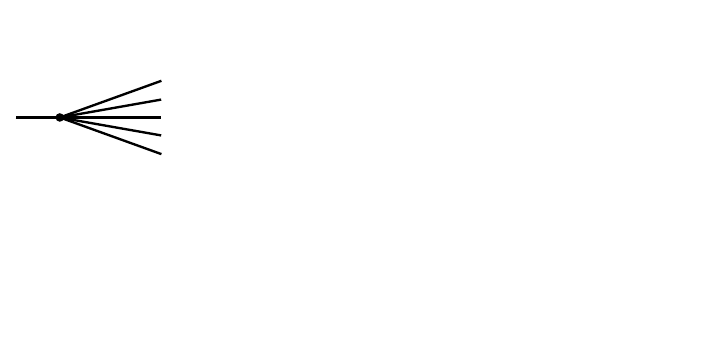
    \caption{An example where $T$ has exactly one gate $G$ of more than one direction and $G$ has exactly 5 directions. The folded train track $T'$ is in $C_{\mathcal{T}}\cap C_{\mathcal{T}'}$. The folded train track $T''$ is in $C_{\mathcal{T}}$ but not $C_{\mathcal{T}'}$. Bottom right is the graph for determining the relevant inequalities.}
    \label{fig:eghigherdimbranchedcube}
\end{figure}
\end{ex}

\subsection{The decomposition into cubes} \label{subsec:axisbundlecubesdefn} 
We view $\mathcal{T}_0$ as a set of possible edges in a (simple, undirected) graph whose vertex set is $\mD(T)$. 
In particular, each subset $\mathcal{T} \subset \mathcal{T}_0$ determines such a graph $\mathfrak{G}_\mT$ that is the graph with an edge for each turn in $\mT$.
We say $\mathcal{T}$ is \textbf{independent} if $\mathfrak{G}_\mT$ has no cycles.

\begin{df}[Cubes]
Let $\mathcal{T} \subset \mathcal{T}_0$ be an independent subset. 
The \textbf{cube} $C_{T,\mathcal{T}}$ is defined as
$$C_{T,\mathcal{T}} = \left\{ T_{(x_\tau,0)} \mid (x_\tau) \in \prod_{\tau \in \mathcal{T}} [0,\ell(\tau)] \right\},$$ 
where $(x_\tau,0)$ is the element of $\prod_{\tau \in \mathcal{T}_0} [0,\ell(\tau)]$ with coordinate $x_\tau$ for $\tau \in \mathcal{T}$ and $0$ for $\tau \in \mathcal{T}_0 \backslash \mathcal{T}$.
\end{df}

By definition, we have $C_{T,\mT} \subset B_T$.
Further, for each $\tau \in \mT$, we have $x_\tau(T_{(x_\tau,0)}) = x_\tau$.
The following lemma states that the functions $x_\tau$ parametrize $C_{T,\mT}$ as a cube.

\begin{lem} \label{lem:cubesarecubes}
The map $\prod_{\tau \in \mathcal{T}} [0,\ell(\tau)]\to C_{T,\mathcal{T}}$ defined by $(x_\tau) \mapsto T_{(x_\tau,0)}$ is a homeomorphism.
\end{lem}
\begin{proof}
The map is continuous and surjective by definition. Since $\prod_{\tau \in \mathcal{T}} [0,\ell(\tau)]$ is compact, it suffices to show injectivity.
From the proof of \Cref{lem:branchedcubesubsetaxisbundle}, the natural fold $f\from T\to T'$ satisfies $F_{T'} f = F_T$. By \Cref{lem:axisbundleinteriorrightmostfold}, $f$ is the unique fold satisfying this property.
Injectivity thus follows since each $x_\tau$ is defined as the length of the largest initial segments with the same $f$-image.
\end{proof}

The functions $x_{\{d_1,d_2\}}$, for $\{d_1,d_2\} \not\in \mT$, are PL functions on $C_{T,\mT}$ in the coordinates $x_\tau$, for the $\tau \in \mT$:

\begin{lem}
Suppose $x \in C_{T,\mT}$ and $\{d_1,d_2\} \not\in \mT$.
Recall the graph $\mathfrak{G}_{\mT}$ corresponding to $\mT$. 
\begin{enumerate}[label=(\alph*)]
    \item If there is a path in $\mathfrak{G}_{\mT}$ connecting $d_1$ and $d_2$, let $p$ be the shortest such path. Then
    \begin{equation} \label{eq:xnotcoord}
        x_{\{d_1,d_2\}} = \min_{\tau \in p} x_\tau.
    \end{equation}
    \item If there is no path in $\mathfrak{G}_{\mT}$ connecting $d_1$ and $d_2$, then $x_{\{d_1,d_2\}}$ is identically zero.
\end{enumerate}
\end{lem}
\begin{proof}
We first show (b). If there is no path in $\mathfrak{G}_{\mT}$ connecting $d_1$ and $d_2$, then either $d_1$ and $d_2$ are not directions at the same gate, or they lie in the same gate but are not folded. In both cases, $x_{\{d_1,d_2\}}=0$.

For (a), we have $x_{\{d_1,d_2\}} \geq \min_{\tau \in p} x_\tau$ since the initial segments of length $\min_{\tau \in p} x_\tau$ of each direction in $p$, including $d_1$ and $d_2$, are all identified together. Conversely, let $\tau_0 \in p \subset \mT$ be such that $\min_{\tau \in p} x_\tau = x_{\tau_0}$ and let $T'$ be the element of $C_{T,\mT}$ with the same coordinates $x_\tau$, for $\tau \in \mT$, as $x$ except $x_{\tau_0} = 0$. Then $T' \in C_{T,\mT'}$ for $\mT' = \mT \backslash \{\tau_0\}$. Since $\mT$ was independent, there is no path in $\mathfrak{G}_{\mT'}$ connecting $d_1$ and $d_2$. Hence, by (2), $x_{\{d_1,d_2\}}(T') = 0$. But $x$ is obtained from $T'$ by folding $\tau_0$ for $x_{\tau_0}$, so $x_{\{d_1,d_2\}} \leq x_{\tau_0} = \min_{\tau \in p} x_\tau$.
\end{proof}

The branched cube $B_T$ is the union of the cubes $C_{T,\mT}$ for the independent sets of turns $\mT\subset\mT_0$:

\begin{prop}
Let $T$ be a fully preprincipal train track. Then $$B_T = \bigcup_{\text{independent $\mathcal{T} \subset \mathcal{T}_0$}} C_{T,\mathcal{T}}.$$
\end{prop}

\begin{proof} 
Again, each $C_{T,\mT} \subset B_T$ by the definitions. 
So we consider $T_{(\ell_t)} \in B_T$ and show $T_{(\ell_t)}\in C_{T,\mT}$ for some $\mT$.
Define an independent $\mathcal{T}\subset\mT_0$ inductively as follows: 
Start with $\mathcal{T} = \varnothing$.
At each stage, consider the turns $\tau \in \mathcal{T}_0$ such that $\mathcal{T} \cup \{\tau\}$ is independent. Among such $\tau$, pick one such that $\ell_\tau$ is maximal, and add that turn to $\mathcal{T}$.
Since $\mathcal{T}_0$ is finite, this process terminates eventually and we have an independent subset $\mathcal{T}$.

We claim $T_{(\ell_\tau)} = T_{(\ell_\tau,0)}$.
First observe that we can obtain $T_{(\ell_\tau)}$ from $T$ by first folding the $\tau\in\mT$ by $\ell_\tau$, to get $T_{(\ell_\tau,0)}$, then folding the remaining $\tau \not\in \mT$ by $\ell_\tau$. 
But for each remaining turn $\{d_1,d_2\}$, we have $x_{\{d_1,d_2\}}(T_{(\ell_\tau,0)})$ is the minimum of the $\ell_\tau$ as $\tau$ ranges over edges in $\mT$ connecting $d_1$ to $d_2$. The value $\ell_{\{d_1,d_2\}}$ cannot exceed this value, or we would have chosen $\{d_1,d_2\}$ to be in $\mT$.
That is, the amount of folding we have to do on $\{d_1,d_2\}$ was already done in $T_{(\ell_\tau,0)}$, so there is no extra folding to be done, and $T_{(\ell_\tau)} = T_{(\ell_\tau,0)}$.
\end{proof}

Our next task is to compute the intersections between the cubes $C_{T,\mT}$ in $B_T$.
Let $\mT$ and $\mT'$ be two independent subsets of $\mT_0$.
Let the components of the graph $\mathfrak{G}_{\mathcal{T} \cap \mathcal{T}'}$ be $C_1,...,C_k$. Note that a single disconnected vertex is also a component.
We define the \textbf{slice} $S_{\mathcal{T},\mathcal{T}'}$ to be the subset of $C_\mathcal{T}$ consisting of all points satisfying the following inequalities for each $i,j=1,...,k$:
\begin{equation}\label{e:slice}
\begin{cases}  
x_{\{d_i,d_j\}} \leq \min\{x_{\{d_i,d'_i\}},x_{\{d_j,d'_j\}}\} & \text{if }  \exists~\{d_i,d_j\} \in \mathcal{T} \text{ and } \{d'_i,d'_j\} \in \mathcal{T}' \text{ connecting } C_i \text{ \& }C_j \\
x_{\{d_i,d_j\}} = 0 & \text{if} ~ \exists~\{d_i,d_j\} \in \mathcal{T} \text{ connecting } C_i  \text{ \& }C_j \text{, but no such }  \tau' \in \mathcal{T}'.
\end{cases}   
\end{equation}
One can verify that this gives the correct set of inequalities for an example in \Cref{eg:higherdimbranchedcube}, using the graph in \Cref{fig:eghigherdimbranchedcube} bottom right.

\begin{lem} \label{lem:branchedcubesintersectinslices}
Let $\mT$ and $\mT'$ be two independent subsets of $\mT_0$. The subset of points in $C_{T,\mT}$ that also lie in $C_{T,\mT'}$ is the slice $S_{\mathcal{T},\mathcal{T}'}$.
\end{lem}

\begin{proof}
Suppose $T' = (x_\tau) \in C_{T,\mT}$ lies in the slice $S_{\mathcal{T},\mathcal{T}'}$, i.e. the inequalities of \Cref{e:slice} are satisfied for each $i,j=1,...,k$.
Think of $T'$ as obtained from $T$ by (1) folding the $\tau\in\mathcal{T} \cap \mathcal{T}'$ for $x_\tau$, then (2) folding the $\{d_i,d_j\}\in\mathcal{T} \backslash \mathcal{T}'$ for $x_{\{d_i,d_j\}}$.
For the folds done in step (2), since $x_{\{d_i,d_j\}} = 0$ if there is no $\tau' \in \mathcal{T}'$ connecting $C_i$ to $C_j$, we only fold $\{d_i,d_j\}$ for values of $i,j$ for which there exists a $\{d'_i,d'_j\} \in \mathcal{T}'$ connecting $C_i$ to $C_j$. 

Since $x_{\{d_i,d_j\}} \leq \min\{x_{\{d_i,d'_i\}},x_{\{d_j,d'_j\}}\}$, by the time we do the step (2) folds, we are folding initial segments of $d_i$ and $d_j$ identified with $d'_i$ and $d'_j$, so we could have equivalently folded $\{d'_i,d'_j\}$ by $x_{\{d_i,d_j\}}$. That is, $T' \in C_{T,\mathcal{T}} \cap C_{T,\mathcal{T}'}$.

Conversely, suppose $T' \in C_{T,\mathcal{T}} \cap C_{T,\mathcal{T}'}$. Then for each $\{d_i,d_j\} \in \mathcal{T} \backslash \mathcal{T}'$, considering $T'$ as a point in $C_{T,\mathcal{T}'}$,
\begin{align*}
x_{\{d_i,d_j\}} &= 
\begin{cases}
\min \{x_{\{d_i,d'_i\}}, x_{\{d'_i,d'_j\}}, x_{\{d'_j,d_j\}} \} &  \text{if } \exists~ \{d'_i,d'_j\} \in \mathcal{T}' \text{ connecting } C_i  ~\&~ C_j \\
0 & \text{otherwise}
\end{cases} \\
&\leq  
\begin{cases}
\min \{x_{\{d_i,d'_i\}}, x_{\{d'_j,d_j\}} \} & \quad\quad\quad \text{if } \exists~ \{d'_i,d'_j\} \in \mathcal{T}' \text{ connecting } C_i  ~\&~ C_j \\
0 &\quad\quad\quad \text{otherwise}
\end{cases}.
\end{align*}
\end{proof}

\subsection{Splitting and folding faces} \label{subsec:splittingfoldingfaces}

In this subsection we define the splitting and folding faces of a branched cube. We use this terminology when analyzing the combinatorics of how branched cubes intersect each other.

Let $T$ be a fully preprincipal train track and $\mT \subset \mT_0$ an independent subset.
The \textbf{splitting face} of $C_{T,\mathcal{T}}$ associated to $\mathcal{T}' \subset \mathcal{T}$ is the subset of $C_{T,\mathcal{T}}$ defined by $x_\tau = 0$ for all $\tau \in \mathcal{T}'$. 
Note that this is the same set as $C_{T,\mathcal{T}'}$.
The \textbf{splitting vertex} of $C_{T,\mathcal{T}}$ is the splitting face associated to $\mathcal{T}$ itself, i.e. the point $\{T\}$.
The \textbf{folding face} of $C_{T,\mathcal{T}}$ associated to $\mathcal{T}' \subset \mathcal{T}$ is the subset of $C_{T,\mathcal{T}}$ defined by $x_\tau = \ell(\tau)$ for all $\tau \in \mathcal{T}'$.

The splitting/folding faces of $B_T$ will be unions of the splitting/folding faces of the $C_{T,\mathcal{T}}$:
The \textbf{splitting face} of $B_T$ associated to a choice of partition $G = \sqcup G_i$ for each gate $G$ of $T$ is the subspace of $B_T$ defined by $x_\tau = 0$ whenever $\tau \not\subset G_i$ for all $i$.
The \textbf{splitting vertex} of $B_T$ is the splitting face associated to the partition of each gate into one-element sets, so is defined by $x_\tau = 0$ for all $\tau$, and is just the point $\{T\}$.
A splitting face of $B_T$ is a proper subset of $B_T$ if and only if at least one of the gate partitions is a nontrivial partition.
The \textbf{folding face} of $B_T$ associated to a subset $\mathcal{T} \subset \mathcal{T}_0$ is the subspace of $B_T$ defined by $x_\tau = \ell(\tau)$ for each $\tau \in \mathcal{T}$.
Such a folding face is a proper subset of $B_T$ if and only if $\mathcal{T}$ is a nonempty subset of $\mathcal{T}_0$.

\begin{lem} \label{lemma:branchedcubeboundary}
Each branched cube $B_T$ is the union of its proper folding faces and the interiors of its splitting faces.
\end{lem}

\begin{proof}
Suppose $x\in B_T$ is not in the interior of $B_T$, nor of any proper folding face of $B_T$. Then $x_\tau = 0$ for some $\tau\in\mT_0$.
Declare directions $d_1,d_2$ equivalent if $x_{\{d_1,d_2\}} > 0$.
This relation is transitive by \Cref{eq:coordinateconstraint}, and clearly reflexive and symmetric.
Further, equivalent directions must lie in the same gate.
Thus, the equivalence classes partition the gates. The partition is nontrivial since $x_\tau = 0$ for some $\tau$. By definition, $x_\tau = 0$ if and only if $\tau$ does not lie in an element of this partition. 
And $x_\tau < \ell(\tau)$ for each $\tau$, or $x$ would lie on a proper folding face of $B_T$.
Thus $x$ lies in the interior of the splitting face associated to this partition.  
\end{proof}

\subsection{Union of the branched cubes} \label{subsec:branchedcubescoveraxisbundle}

The goal now is to show the branched cubes $B_T$ cover $\Av$.
We first show the following lemma, providing that, as a cube, the vertices of each $C_{T,\mT}$ are fully preprincipal points.

\begin{lem} \label{lem:cubeverticesfpp}
The vertices of $C_{T,\mathcal{T}}$, i.e. the points where each $x_\tau$ is either $0$ or $\ell(\tau)$, are fully preprincipal.
\end{lem}
\begin{proof}
Let $T'$ be a vertex of $C_{T,\mathcal{T}}$. Then $T'$ is obtained from $T$ by folding some collection of illegal turns in $T$ fully. 
Since each turn folded is folded fully (and by a $\La$-sometry), any new vertex created by the fold has at least 3 gates. 
Further, since each turn folded was illegal, the number of gates at a vertex could not have decreased. Since $T$ was fully preprincipal, each vertex of $T$ has at least 3 gates, and so each vertex of $T'$ has at least 3 gates, and thus $T'$ is also fully preprincipal.
\end{proof}

\begin{prop} \label{prop:branchedcubescoveraxisbundle}
Let $S \in \Av$. Then there is a fully preprincipal $T \in \Av$ with $S \in B_T$. 
In fact, we can choose $T$ to have the same local decomposition as $S$, and so that $S$ is not on a proper folding face of $B_T$.

In particular, by \Cref{lemma:branchedcubeboundary}, for each $\mG$, we have that $\mG\Av$ is the union of the interiors of the splitting faces of $B_T$, as $T$ ranges over all fully preprincipal weak train tracks whose local decomposition is $\mG$.
\end{prop}
\begin{proof}
Let $R$ be the rotationless power of $\vphi$ and $\mG$ the local decomposition of $S$. By \Cref{p:CutDecompositions}, there is a fully preprincipal train track representative $g:\Gamma \to \Gamma$ for $\phi$ such that $T = \widetilde{\Gamma}$ has local decomposition $\mG$.

Using \Cref{P:6.2}, we can shift $T$ so that the $\veps$ requirement in \Cref{P:5.4} is satisfied and then we can use \Cref{P:5.4} to know that $S$ can be obtained from $T$ by folding. However, we are not done yet because the fold path $\alpha$ from $T$ to $S$ may pass through multiple branched cubes.

We modify the fold path inductively as follows:
If $\alpha$ does not meet a proper folding face of $B_T$, then $\alpha$ stays inside $B_T$.
Otherwise, if $\alpha$ meets a proper folding face of $B_T$, then $\alpha$ must meet a folding face of some cube $C_{T,\mT}$ at some point $S'$. 
We can choose a vertex $T' \neq T$ of $C_{T,\mT}$ so that there is a fold path from $T'$ to $S'$. We then modify $\alpha$ by replacing its initial segment from $T$ to $S'$ by the fold path from $T'$ to $S'$.

Since there is a fold path from $T$ to $T'$ and from $T'$ to $S$, the local decomposition of $T'$ is also $\mG$. For the same reason, $vol(T')<vol(T)$ and $vol(T')\geq vol(S)$. By \Cref{lem:fpplocallyfinite}, this shows the process terminates.
\end{proof}

\subsection{Intersections of the branched cubes} \label{subsec:axisbundlebranchedcubesintersect}

The goal of this subsection is to analyze the intersection of the branched cubes $B_T$.
The following construction will play a large role.

\begin{constr}[Peel-off] \label{constr:peeloff}
Let $T$ be an element of $\mathcal{A}_\phi$ and $v\in VT$.
Suppose there is a $d_0\in \mD(v)$ and disjoint nonempty subsets $D_1, D_2\subset \mD(v)$ with (see \Cref{fig:possiblesplit}, where the purple lines are $\Lambda$-leaves):
\begin{itemize}
    \item[a.] $\mD(v)=\{d_0\} \cup D_1 \cup D_2$.
    \item[b.] All $\Lambda$-leaves that pass through $v$ by entering at $d_0$ exit through $D_1$ or $D_2$.
    \item[c.] No $\Lambda$-leaves pass through $v$ by entering at $D_1$ and exiting through $D_2$.
    \item[d.] All $\Lambda$-leaves that pass through $v$ by entering at $D_1$ and exiting through $d_0$ travel along the same segment $I$ before meeting a vertex $w$ with $\geq 3$ gates.
\end{itemize}

In this case, we say $(D_1,D_2,I)$ is a \textbf{possible peel}. Note that, by irreducibility, (c)-(d) imply there are $\La$-leaves entering at each of $D_1$ and $D_2$ before passing through $I$.

We can define a fully preprincipal train track $T_I$ by detaching the directions in $D_1$ at $v$, attaching them to an endpoint of a copy of $I$, and attaching the other endpoint of the copy of $I$ to $w$.
We say that $T_I$ is obtained from $T$ by \textbf{peeling off} $D_1$ from $D_2$ at $v$ along $I$.
Note that there is a folding map $h_I:T_I \to T$ and the preimage of $I$ is the union of two segments $I_1$ and $I_2$ that meet at $w$. 

The special case when $v$ is a two-gate vertex will be particularly important in the following. 
In this case, note that we decrease the valence at a two-gate vertex when splitting from $T$ to $T_I$. 

We say that a fully preprincipal train track $T'$ is obtained by \textbf{completely peeling} $T$ if there is a sequence of peel-offs $T=T_0 \to T_1 \to ... \to T_m = T'$ at two-gate vertices.

\begin{figure}[H]
    \centering
    \selectfont\fontsize{6pt}{6pt}
    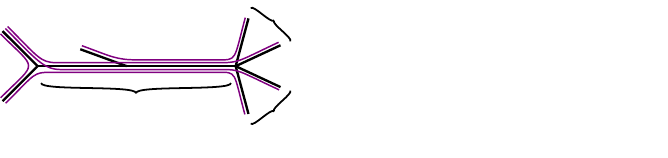
    \caption{Left: A possible peel $(D_1,D_2,I)$. Right: Peeling off $D_1$ from $D_2$ at $v$ along $I$.}
    \label{fig:possiblesplit}
\end{figure}
\end{constr}

\begin{ex} \label{ex:simplefoldpathsgivesplits}
Let $T$ be a point in a branched cube $B_{T'}$. Then $T$ lies in a cube $C_{T',\mathcal{T}}$.
Let $\tau = \{d_1,d_2\} \in \mathcal{T}$ be a turn such that $x_\tau(T) > 0$. Let $T_1$ be the point with the same coordinates as $T$ except that $x_\tau(T_1) = 0$.

Then $T$ can be obtained by first folding $T'$ to $T_1$ along the turns in $\mathcal{T} \backslash \{\tau\}$ then folding initial segments $I_1 \subset d_1$ and $I_2 \subset d_2$ of length $x_\tau$. Let $h:T' \to T$ be that folding map and $v_1$, $v_2$ the terminal points of $I_1$, $I_2$ respectively.
Let $I \subset T$ be the common image of $I_1$ and $I_2$. 
Let $v = h(v_1) = h(v_2)$ be the endpoint of $I$ other than $w$, and let $d_0$ be the direction at $v$ determined by $I$.
Let $D_1$ be the $h$-image of the directions at $v_1$ other than $I_1$, and let $D_2=\mD(v) \backslash (\{d_0\} \cup D_1)$.
Then $(D_1,D_2,I)$ is a possible peel and $T_I = T_1$.
\end{ex}

We have a criterion for determining whether we are in the situation of the example.

\begin{lem} \label{lemma:splitinsidebranchedcubecriterion}
Suppose $T \in B_{T'}$.
Recall from \Cref{df:axisbundlebranchedcube} that there is a canonical fold $h: T' \to T$.
A split train track $T_I$ lies in $B_{T'}$ if and only if there exists no possible peel $(D'_1,D'_2,I')$ where $h(I') \supset I$, and $h(D'_1) \subset D_1$, and $h(D'_2) \subset D_2$.
If $T_I\in B_{T'}$, then $T_I$ lies in a splitting face of $B_{T'}$.
\end{lem}

\begin{proof}
Suppose $T_I$ lies in $B_{T'}$ and suppose there is a possible peel $(D'_1,D'_2,I')$ as in the lemma. Let $h'$ be the fold $h':T' \to T_I$ and $h_I$ the fold $h_I:T_I \to T$. Then the $h'$-image of the segment $I'$ has to be mapped via $h_I$ over $I$. But since there are $\Lambda$-leaves passing through $D'_1$ and $I'$, and $h_I$ is a $\Lambda$-isometry, $h'(I')$ must pass through $I_1$. Similarly, $h'(I')$ must pass through $I_2$. This contradicts $h:T' \to T$ being an isometry on $I'$. 
This argument shows that if a possible peel $(D'_1,D'_2,I')$ exists, then $T_I$ does not lie in $B_{T'}$.

Conversely, suppose $T_I\notin B_{T'}$. Then the preimage $h^{-1}(I)$ must be a union of edge segments, or since $T'$ is fully preprincipal, there would be a $\geq 3$-gate vertex in the interior of $h^{-1}(I)$, mapping to a $\geq 3$-gate vertex in the interior of $I$.
Consider the edge segments $J$ in $h^{-1}(I)$, having endpoints $v_J$, $w_J$ mapped to $v$, $w$ respectively. 
For each $J$, consider the leaves that pass through $v_J$ by entering through $J$. 
There are three cases (indicated in the upper left image of the figure):

\parpic[l]{\selectfont \fontsize{6pt}{6pt} 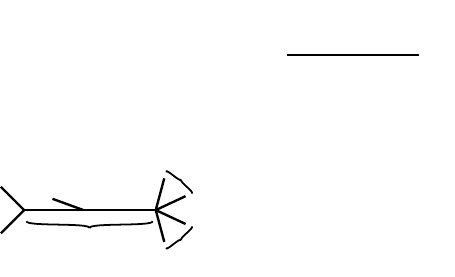}
(1) All leaves exit $v_J$ through a direction mapped by $h$ to $D_1$.

(2) All leaves exit $v_J$ through a direction direction mapped by $h$ to $D_2$.

(3) Some leaves exit $v_J$ through a direction mapped by $h$ to $D_1$, while some leaves exit $v_J$ through a direction mapped by $h$ to $D_2$.

If $J$ is of type (3), one can construct a possible peel by extending $J$ until it meets a vertex.
Thus we can assume each $J$ to be of type (1) or (2) as in the upper right image of the figure.

Let $w'_1,...,w'_m$ be the vertices of $T'$ mapped to $w$ by $h$.
Up to reindexing, suppose $w'_1,...,w'_p$ are the vertices that meet \emph{both} segments $J$ of type (1) \emph{and} of type (2). For each $i=1,...,p$, we choose a turn $\tau_i$ with one direction being a segment of type (1) and the other direction being a segment of type (2).

The goal of the rest of the proof is to locate a cube of the form $C_{T', \mT_1 \cup \{\tau_1,...,\tau_p\} \cup \mT_2}$ containing both $T$ and $T_I$, hence contradicting the assumption that $T_I \not\in B_{T'}$.
Toward this end, choose a cube $C_{T',\mT}$ containing $T$ and let $(x_\tau)$ be the coordinates of $T$ in that cube.
Temporarily suppressing the subscript $i$ for notational simplicity, suppose $\tau=\{d_1,d_2\}$, with $d_1$ determined by the segment of type (1) and $d_2$ determined by the segment of type (2).
Let $G$ be the gate containing $d_1$ and $d_2$. 
Let $G_1 = \{d \in G \mid x_{\{d,d_1\}} > x_{\{d_1,d_2\}}\}$ and $G_2 = \{d \in G \mid x_{\{d,d_1\}} \leq x_{\{d_1,d_2\}}\}$.
Thus $G = G_1 \sqcup G_2$, with $d_1 \in G_1$ and $d_2 \in G_2$.

There must be a path in $\mathfrak{G}_{\mT}$ from $d_1$ to $d_2$, or we would have $x_{\{d_1,d_2\}} = 0$ which is not the case since $\tau$ is folded. 
Recall from \Cref{eq:xnotcoord} that $x_{\{d_1,d_2\}}$ is the minimum of $x_\tau$ as $\tau$ ranges over edges of this path. Let $\{d'_1,d'_2\}$ be the edge such that $x_{\{d'_1,d'_2\}} = x_{\{d_1,d_2\}}$. Then 
\begin{align*}
x_{\{d'_1,d'_2\}} &= x_{\{d_1,d_2\}} \\
&\leq \min \{x_{\{d_1,d'_1\}}, x_{\{d'_1,d'_2\}}, x_{\{d'_2,d_2\}}\} & \text{by \Cref{eq:xnotcoord}}\\
&\leq \min \{x_{\{d_1,d'_1\}}, x_{\{d_2,d'_2\}}\}
\end{align*}
so by \Cref{lem:branchedcubesintersectinslices}, $T$ lies in $S_{\mathcal{T},\mathcal{T}'}$ where $\mathcal{T}' = (\mathcal{T} ~\backslash~ \{\{d'_1,d'_2\}\})\cup\{\{d_1,d_2\}\}$. 
(Here $\mathcal{T} \cap \mathcal{T}' = \mathcal{T} ~\backslash~ \{\{d'_1,d'_2\}\}$, so in the notation of \Cref{e:slice}, up to relabeling, we can take $d_1,d'_1 \in C_1$ and $d_2,d'_2 \in C_2$ with $[d'_1,d'_2] \in \mathfrak{G}_{\mT}$ and $[d_1,d_2] \in \mathfrak{G}_{\mT'}$.)
Thus, up to replacing $\mathcal{T}$ by $\mathcal{T}'$, we can assume $\tau = \{d_1,d_2\} \in \mathcal{T}$.

If $\mathfrak{G}_{\mT}$ contains an edge $\{d,d_1\}$ where $d \in G_2$ then, since $x_{\{d,d_1\}} \leq x_{\{d_1,d_2\}}$, we have $T\in S_{\mathcal{T},\mathcal{T}'}$ for $\mathcal{T}' = (\mathcal{T} ~\backslash~ \{\{d,d_1\}\})\cup \{\{d,d_2\}\}$ by \Cref{lem:branchedcubesintersectinslices}. 
(Here $\mathcal{T} \cap \mathcal{T}' = \mathcal{T} ~\backslash~ \{\{d,d_1\}\}$, so in the notation of \Cref{e:slice}, up to relabeling, we can take $d \in C_1$ and $d_1,d_2 \in C_2$ with $(d,d_1) \in \mT$ and $(d,d_2) \in \mT'$.)
Thus, up to replacing $\mathcal{T}$ by $\mathcal{T}'$, we can assume no edges connect $d_1$ and $G_2$. 

Similarly, if $\mathfrak{G}_{\mT}$ contains an edge $\{d,d_2\}$ where $d \in G_1$, then since $x_{\{d,d_1\}} > x_{\{d_1,d_2\}}$, we have $x_{\{d,d_2\}} = x_{\{d_1,d_2\}}$ by \Cref{eq:coordinateconstraint}. Thus, by the symmetric argument, we can assume no edges connect $d_2$ and $G_1$.

The conclusion is that $T \in C_{T', \mT_1 \cup \{\tau_1,...,\tau_p\} \cup \mT_2}$ where $\mT_j$ is an independent subset whose elements lie in $G_j$ (for some $i$), for $j=1,2$. Within this cube, the split train track $T_I$ is the point with the same coordinates as $T$ except $x_{\tau_i} = 0$ for each $i=1,...,p$. In particular $T_I$ lies in a splitting face of $B_{T'}$. Contradiction.
\end{proof}

This criterion allows us to show that each $T'$ has a unique complete peeling.

\begin{lem} \label{lem:completesplitinterior}
Each $T' \in \mathcal{A}_\phi$ has a unique complete peel. 
More specifically, if $T'$ lies in the interior of a splitting face of  $B_T$ then the complete peel of $T'$ is $T$.
\end{lem}
\begin{proof} 
By \Cref{prop:branchedcubescoveraxisbundle}, we know that $T'$ lies in the interior of a splitting face of some $B_{T}$.
Then $T'$ lies in the interior of some cube $C_{T,\mathcal{T}}$.
If $x_\tau(T') > 0$, then we can run \Cref{ex:simplefoldpathsgivesplits} to peel to the point $T'_1$ with the same coordinates as $T'$, except that $x_\tau(T'_1) = 0$.
Repeating this argument inductively, we get to a complete peel, which must then be the splitting vertex of $B_{T}$, namely $T$.

Conversely, by \Cref{lemma:splitinsidebranchedcubecriterion}, each peel of $T'$ at a 2-gate vertex lies in a splitting face of $B_{T}$, for otherwise there is a possible peel at a 2-gate vertex of $T$, but since $T'$ is fully preprincipal these cannot exist.
\end{proof}

\begin{lem} \label{lem:completesplittingboundary}
Suppose $T'$ lies in the branched cube $B_T$. Then the complete peel of $T'$ lies in $B_T$ as well. 
\end{lem}
\begin{proof}
We apply \Cref{lemma:splitinsidebranchedcubecriterion} repeatedly. At every peeling at a 2-gate vertex, we stay in $B_{T'}$, for otherwise there is a possible peel at a 2-gate vertex of $T'$, but since $T'$ is fully preprincipal these cannot exist.
\end{proof}

\begin{prop} \label{prop:branchedcubedisjointint}
Distinct branched cubes $B_{T_1}$ and $B_{T_2}$ cannot intersect away from their folding faces.
\end{prop}
\begin{proof}
Assume otherwise, then there exists some $T$ lying in the interior of a splitting face of $B_{T_1}$ and that of $B_{T_2}$. Taking their complete peels, we get $T_1 = T_2$ by \Cref{lem:completesplitinterior}.
\end{proof}

\begin{prop} \label{prop:foldfacevertex}
Let $T'$ be a fully preprincipal weak train track on a folding face $F$ of $B_T$. Then $B_T \cap B_{T'}$ is a splitting face of $B_{T'}$ contained in $F$. See \Cref{fig:branchedcubefoldface}.
\end{prop}
\begin{proof}
For each point $T'' \in B_T \cap B_{T'}$, there is a fold path from $T' \in F$ to $T''$, hence $B_T \cap B_{T'}\subset F$. It
\parpic[r]{\selectfont\fontsize{16pt}{16pt} \resizebox{!}{1.6cm}{
\begingroup%
  \makeatletter%
  \providecommand\color[2][]{%
    \errmessage{(Inkscape) Color is used for the text in Inkscape, but the package 'color.sty' is not loaded}%
    \renewcommand\color[2][]{}%
  }%
  \providecommand\transparent[1]{%
    \errmessage{(Inkscape) Transparency is used (non-zero) for the text in Inkscape, but the package 'transparent.sty' is not loaded}%
    \renewcommand\transparent[1]{}%
  }%
  \providecommand\rotatebox[2]{#2}%
  \newcommand*\fsize{\dimexpr\f@size pt\relax}%
  \newcommand*\lineheight[1]{\fontsize{\fsize}{#1\fsize}\selectfont}%
  \ifx\svgwidth\undefined%
    \setlength{\unitlength}{297.635784bp}%
    \ifx\svgscale\undefined%
      \relax%
    \else%
      \setlength{\unitlength}{\unitlength * \real{\svgscale}}%
    \fi%
  \else%
    \setlength{\unitlength}{\svgwidth}%
  \fi%
  \global\let\svgwidth\undefined%
  \global\let\svgscale\undefined%
  \makeatother%
  \begin{picture}(1,0.24398338)%
    \lineheight{1}%
    \setlength\tabcolsep{0pt}%
    \put(0,0){\includegraphics[width=\unitlength,page=1]{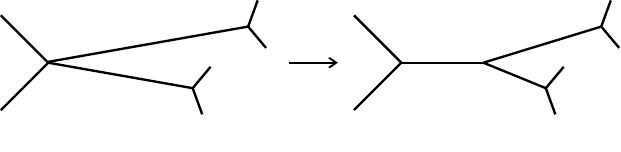}}%
    \put(0.23463357,0.19758848){\color[rgb]{0,0,0}\makebox(0,0)[lt]{\lineheight{1.25}\smash{\begin{tabular}[t]{l}$e_1$\end{tabular}}}}%
    \put(0.19286337,0.07404203){\color[rgb]{0,0,0}\makebox(0,0)[lt]{\lineheight{1.25}\smash{\begin{tabular}[t]{l}$e_2$\end{tabular}}}}%
    \put(0.8372841,0.19548334){\color[rgb]{0,0,0}\makebox(0,0)[lt]{\lineheight{1.25}\smash{\begin{tabular}[t]{l}$e'_1$\end{tabular}}}}%
    \put(0.79238845,0.07387345){\color[rgb]{0,0,0}\makebox(0,0)[lt]{\lineheight{1.25}\smash{\begin{tabular}[t]{l}$e'_2$\end{tabular}}}}%
    \put(0.18731677,0.00306035){\color[rgb]{0,0,0}\makebox(0,0)[lt]{\lineheight{1.25}\smash{\begin{tabular}[t]{l}$T$\end{tabular}}}}%
    \put(0.72163551,0.00276927){\color[rgb]{0,0,0}\makebox(0,0)[lt]{\lineheight{1.25}\smash{\begin{tabular}[t]{l}$T'$\end{tabular}}}}%
  \end{picture}%
\endgroup%
}}
\noindent remains to show that this intersection is a splitting face of $B_{T'}$.

Let $(x'_\tau)$ be the coordinates of $T'$ in $B_T$ and $\mathcal{T}$ the subset of $\mathcal{T}_0$ consisting of elements $\tau$ for which $x'_\tau < \ell(\tau)$. 
Each $\tau \in \mathcal{T}$ specifies an element $\tau^\circ$ of $\mathcal{T}'_0$, namely the unfolded portion of the directions in $\tau$.
More precisely, suppose $\tau = \{d_1,d_2\}$, and suppose $d_i$ is the germ of the edge $e_i$. Then since $x_{\tau} < \ell(\tau)$, there are some terminal segments $e'_1 \subset e_1$ and $e'_2 \subset e_2$ so that the images of $e'_1$ and $e'_2$ in $T'$ determine two directions $d'_1$ and $d'_2$ at a vertex. Moreover, $d'_1$ and $d'_2$ lie in the same gate, since $e'_1$ and $e'_2$ would be folded if one folds $\tau$ completely. The turn $\tau^\circ$ determined by $\tau$ is $\{d'_1,d'_2\}$.
Furthermore, $\ell(\tau^\circ)\leq\ell(\tau) - x'_\tau$. Let $\mathcal{T}^\circ \subset \mathcal{T}'_0$ denote the subset consisting of the $\tau^\circ$ arising as such.

We claim $\mathcal{T}^\circ$, when considered as a graph with vertex set $G'$, is a disjoint union of complete subgraphs of $\mathcal{T}'_0$. For this, it suffices to show that if $\{d'_0,d'_1\} \in \mathcal{T}^\circ$ and $\{d'_1,d'_2\} \in \mathcal{T}^\circ$, then $\{d'_0,d'_2\} \in \mathcal{T}^\circ$.
Let $\{d_0,d_1\} \in \mathcal{T}_0$ be a turn determining $\{d'_0,d'_1\}$ and $\{d_2,d_3\} \in \mathcal{T}_0$ a turn determining $\{d'_1,d'_2\}$. 
The directions $d_0,d_1,d_2,d_3$ must lie in the same gate of $T$, since otherwise $d_1$ and $d_2$ cannot pass through the same direction $d'_1$.
Thus $\{d_0,d_3\} \in \mathcal{T}_0$ and determines $\{d'_0,d'_2\} \in \mathcal{T}^\circ$, as desired. So the claim is proved.

By the claim, $\mathcal{T}^\circ$ determines a partition of $G'$, namely where two directions $d'_1,d'_2$ lie in the same subset if and only if $\{d'_1,d'_2\} \in \mathcal{T}^\circ$.
The intersection $B_T \cap B_{T'}$ is the splitting face associated to this partition.
\end{proof}

\begin{figure}[H]
    \centering
    \selectfont\fontsize{8pt}{8pt}
\begingroup%
  \makeatletter%
  \providecommand\color[2][]{%
    \errmessage{(Inkscape) Color is used for the text in Inkscape, but the package 'color.sty' is not loaded}%
    \renewcommand\color[2][]{}%
  }%
  \providecommand\transparent[1]{%
    \errmessage{(Inkscape) Transparency is used (non-zero) for the text in Inkscape, but the package 'transparent.sty' is not loaded}%
    \renewcommand\transparent[1]{}%
  }%
  \providecommand\rotatebox[2]{#2}%
  \newcommand*\fsize{\dimexpr\f@size pt\relax}%
  \newcommand*\lineheight[1]{\fontsize{\fsize}{#1\fsize}\selectfont}%
  \ifx\svgwidth\undefined%
    \setlength{\unitlength}{235.95743495bp}%
    \ifx\svgscale\undefined%
      \relax%
    \else%
      \setlength{\unitlength}{\unitlength * \real{\svgscale}}%
    \fi%
  \else%
    \setlength{\unitlength}{\svgwidth}%
  \fi%
  \global\let\svgwidth\undefined%
  \global\let\svgscale\undefined%
  \makeatother%
  \begin{picture}(1,0.38879148)%
    \lineheight{1}%
    \setlength\tabcolsep{0pt}%
    \put(0,0){\includegraphics[width=\unitlength,page=1]{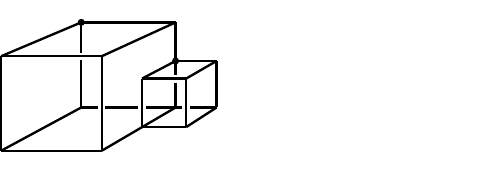}}%
    \put(0.16280154,0.36270099){\color[rgb]{0,0,0}\makebox(0,0)[lt]{\lineheight{1.25}\smash{\begin{tabular}[t]{l}$T$\end{tabular}}}}%
    \put(0.36904788,0.27724069){\color[rgb]{0,0,0}\makebox(0,0)[lt]{\lineheight{1.25}\smash{\begin{tabular}[t]{l}$T'$\end{tabular}}}}%
    \put(0,0){\includegraphics[width=\unitlength,page=2]{branchedcubefoldface.pdf}}%
    \put(0.71984881,0.36270099){\color[rgb]{0,0,0}\makebox(0,0)[lt]{\lineheight{1.25}\smash{\begin{tabular}[t]{l}$T$\end{tabular}}}}%
    \put(0.92609519,0.27724069){\color[rgb]{0,0,0}\makebox(0,0)[lt]{\lineheight{1.25}\smash{\begin{tabular}[t]{l}$T'$\end{tabular}}}}%
    \put(0.8679416,0.11201386){\color[rgb]{0,0,0}\makebox(0,0)[lt]{\lineheight{1.25}\smash{\begin{tabular}[t]{l}$x$\end{tabular}}}}%
    \put(0,0){\includegraphics[width=\unitlength,page=3]{branchedcubefoldface.pdf}}%
  \end{picture}%
\endgroup%

    \caption{If $T'$ be a fully preprincipal train track on a folding face of $B_T$, then $B_T \cap B_{T'}$ is a splitting face of $B_{T'}$, as in the left, and \emph{not} on the right.}
    \label{fig:branchedcubefoldface}
\end{figure}

\begin{prop} \label{prop:foldfacebranchedcubes}
Each folding face of a branched cube $B_T$ is a union of splitting faces of branched cubes. 
\end{prop}

\begin{proof}
Fix a folding face $F$.
By \Cref{prop:branchedcubescoveraxisbundle}, each point $x \in F$ lies in the interior of a splitting face $S$ of a branched cube $B_{T'}$. 
\Cref{lem:completesplittingboundary} implies that $T'$, being the complete peel of $x$, must lie in $B_T$.
Suppose $T'$ lies in a folding face $F'$ of $B_T$, and suppose $F'$ is minimal with respect to this property.
Since there is a fold path from $T'$ to $x$, we have that $F$ is a subfolding face of $F'$.

By \Cref{prop:foldfacevertex}, $B_T \cap B_{T'}$ is a splitting face $S'$ of $B_{T'}$ contained in $F'$.
Since $x \in B_T \cap B_{T'} = S'$, we have that $S$ is a subsplitting face of $S'$.

Suppose $F$ is a proper subset of $F'$. Then by the minimality of $F'$, we have $T' \in F' \backslash F$, which implies that the interiors of the subsplitting faces of $S' = B_T \cap B_{T'}$ cannot meet $F$, contradicting the choice of $S$ as a splitting face that contains $x$ in its interior.
Thus $F=F'$. In particular $T'$ lies in $F$, which implies that $S' = B_T \cap B_{T'}$ lies in $F$ as well.
This argument shows that $F$ is the union of such splitting faces $S'$.
\end{proof}

\subsection{Unique successor property} \label{subsec:axisbundleuniquesuccessor}

This subsection has a final combinatorial property of branched cubes.

\begin{prop} \label{prop:uniquesuccessor}
For each point $T$ in the axis bundle, the set of branched cubes $B_{T'}$ for which $T$ lies in the interior of a folding face of $B_{T'}$, once partially ordered by inclusion, has a unique maximal element.
\end{prop}

\begin{proof}
Suppose otherwise that there are two maximal branched cubes $B_{T'_1}, B_{T'_2}$ for which $T$ lies in the interior of a folding face of $B_{T'_i}$. 
Suppose $T$ lies in a cube $C_{T'_2,\mathcal{T}} \subset B_{T'_2}$. Without loss of generality suppose $\mathcal{T}$ is a maximal independent set. As explained in \Cref{ex:simplefoldpathsgivesplits}, each coordinate $x_\tau$ gives a possible peel, and the act of peeling determines a split path. 
If all such split paths lie in $B_{T'_1}$, then the interior of $B_{T'_2}$ would meet $B_{T'_1}$.
But then by \Cref{prop:branchedcubedisjointint}, $B_{T'_2}$ lies on a folding face of $B_{T'_1}$, contradicting maximality of $B_{T'_2}$.
Thus some split path, determined by some possible peel $(D_1,D_2,I)$, does not lie in $B_{T'_1}$.

We apply \Cref{lemma:splitinsidebranchedcubecriterion} to $B_{T'_1}$ and the possible peel $(D_1,D_2,I)$ to obtain a possible peel $(D'_1,D'_2,I')$ of $T'_1$ which maps to $(D_1,D_2,I)$ in the sense of the lemma.
Let $T''$ be the train track obtained by splitting this possible peel.
Since $T'_1$ lies in a folding face of $B_{T''}$, \Cref{prop:foldfacevertex} states that $B_{T'_1} \cap B_{T''}$ is a splitting face of $B_{T'_1}$, so $T$ cannot lie in the interior of a folding face of $B_{T'_1}$ unless $B_{T'_1} \cap B_{T''} = B_{T'_1}$, see \Cref{fig:uniquesuccessor}.
This is equivalent to $B_{T'_1} \subset B_{T''}$, but this contradicts the maximality of $B_{T'_1}$.
\end{proof}

\begin{figure}
    \centering
    \selectfont\fontsize{7pt}{7pt}
\begingroup%
  \makeatletter%
  \providecommand\color[2][]{%
    \errmessage{(Inkscape) Color is used for the text in Inkscape, but the package 'color.sty' is not loaded}%
    \renewcommand\color[2][]{}%
  }%
  \providecommand\transparent[1]{%
    \errmessage{(Inkscape) Transparency is used (non-zero) for the text in Inkscape, but the package 'transparent.sty' is not loaded}%
    \renewcommand\transparent[1]{}%
  }%
  \providecommand\rotatebox[2]{#2}%
  \newcommand*\fsize{\dimexpr\f@size pt\relax}%
  \newcommand*\lineheight[1]{\fontsize{\fsize}{#1\fsize}\selectfont}%
  \ifx\svgwidth\undefined%
    \setlength{\unitlength}{134.26594892bp}%
    \ifx\svgscale\undefined%
      \relax%
    \else%
      \setlength{\unitlength}{\unitlength * \real{\svgscale}}%
    \fi%
  \else%
    \setlength{\unitlength}{\svgwidth}%
  \fi%
  \global\let\svgwidth\undefined%
  \global\let\svgscale\undefined%
  \makeatother%
  \begin{picture}(1,0.6643065)%
    \lineheight{1}%
    \setlength\tabcolsep{0pt}%
    \put(0,0){\includegraphics[width=\unitlength,page=1]{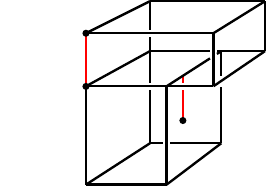}}%
    \put(0.20849971,0.35515874){\color[rgb]{0,0,0}\makebox(0,0)[lt]{\lineheight{1.25}\smash{\begin{tabular}[t]{l}$T'_1$\end{tabular}}}}%
    \put(0.20849971,0.53390863){\color[rgb]{0,0,0}\makebox(0,0)[lt]{\lineheight{1.25}\smash{\begin{tabular}[t]{l}$T''$\end{tabular}}}}%
    \put(0.67077407,0.19926989){\color[rgb]{0,0,0}\makebox(0,0)[lt]{\lineheight{1.25}\smash{\begin{tabular}[t]{l}$T$\end{tabular}}}}%
    \put(0.67186985,0.29428755){\color[rgb]{1,0,0}\makebox(0,0)[lt]{\lineheight{1.25}\smash{\begin{tabular}[t]{l}$(D_1,D_2,I)$\end{tabular}}}}%
    \put(-0.00338201,0.43974701){\color[rgb]{1,0,0}\makebox(0,0)[lt]{\lineheight{1.25}\smash{\begin{tabular}[t]{l}$(D'_1,D'_2,I')$\end{tabular}}}}%
  \end{picture}%
\endgroup%

    \caption{$B_{T'_1} \cap B_{T''}$ is a splitting face of $B_{T'_1}$, so $T$ cannot lie in the interior of a folding face of $B_{T'_1}$ unless $B_{T'_1} \cap B_{T''} = B_{T'_1}$.}
    \label{fig:uniquesuccessor}
\end{figure}

\section{Cubist complexes} \label{sec:cubistcomplex}

In this section we introduce the definition of a cubist complex. We choose to present the definition in a general abstract setting in anticipation of future applications. Because of this, we have to first introduce an abstract definition of a branched cube in \S \ref{subsec:branchedcubesaxiomdefn}. The definition of a cubist complex will then appear in \S \ref{subsec:cubistcomplexdefn}.

Using the properties established in the previous section, we then show the axis bundle is a cubist complex in \S \ref{subsec:axisbundleiscubistcomplex}. Finally, we define the cardiovascular system of a cubist complex and study its properties in \S \ref{subsec:cardiovascularsystem}.

\subsection{Abstract definition of branched cubes} \label{subsec:branchedcubesaxiomdefn}

We first present an abstract definition of branched cubes. 
The properties and terminology here are all motivated by the axis bundle setting of \Cref{sec:axisbundlebranchedcubes}.

Let $\mathcal{G}=\{G_1,...,G_m\}$ be a finite collection of finite sets.
We denote by $\Lambda^2 G_i$ the set of unordered pairs of distinct elements of $G_i$. (This notation is motivated from exterior products of vector spaces.) 
Let $\mathcal{T}_0 = \Lambda^2 G_1 \sqcup ... \sqcup \Lambda^2 G_m$. 
As in \Cref{sec:axisbundlebranchedcubes}, we can view each subset $\mathcal{T} \subset \mathcal{T}_0$ as a graph $\mathfrak{G}_{\mT}$ with vertex set $\sqcup_{i=1}^m G_i$. 
We say that $\mathcal{T}$ is \textbf{independent} if $\mathfrak{G}_{\mT}$ has no cycles.

Suppose we have a $\ell(\tau)\in\RR_{>0}$ associated to each $\tau \in \mathcal{T}_0$ satisfying \Cref{eq:foldlengthconstraint}. For each independent subset $\mathcal{T} \subset \mathcal{T}_0$, we define the \textbf{cube} $C_\mathcal{T}$ to be the metric space $\prod_{\tau \in \mathcal{T}} [0,\ell(\tau)]$. 
The coordinates of $C_\mathcal{T}$ are denoted $x^{(\mathcal{T})}_\tau$, for each $\tau \in \mathcal{T}$.

We extend the notation $x^{(\mathcal{T})}_{\{d_1,d_2\}}$ to make sense for each $\{d_1,d_2\} \in \mT_0$ as follows:
If there is a path in $\mathfrak{G}_{\mT}$ connecting $d_1$ to $d_2$, consider the shortest such path (which is unique since $\mathcal{T}$ has no cycles) and define $x^{(\mathcal{T})}_{\{d_1,d_2\}}$ as the minimum over $x^{(\mathcal{T})}_\tau$ as $\tau$ ranges over the edges of this path.
If $d_1=d_2$, then $x^{(\mathcal{T})}_{\{d_1,d_2\}} = \infty$.
Also, if $\tau \in \mathcal{T}$, then $x^{(\mathcal{T})}_\tau$ retains its original definition as a coordinate on $C_{\mathcal{T}}$.
If there is no path in $\mathfrak{\mathcal{T}}$ connecting $d_1$ to $d_2$, define $x^{(\mathcal{T})}_{\{d_1,d_2\}}=0$.

A \textbf{fold path} in $C_\mathcal{T}$ is an oriented path of the form $\alpha(s) = (\alpha_t(s))_{t \in \mathcal{T}}$ where each $\alpha_t$ is an nondecreasing function.
A \textbf{subcube} of $C_\mathcal{T}$ is a subset of the form $\prod_{\tau \in \mathcal{T}} [x'_\tau,x''_\tau]$ where $x'_\tau \leq x''_\tau$ for each $\tau$. 
Each subcube can be viewed as an isometrically embedded copy of some other cube such that the image of each folding path is a folding path.
As another subcube example, let $p_0\in C_\mathcal{T}(\ell_t)$ and then the set of $p\in C_\mathcal{T}$ for which there is a fold path from $p_0$ to $p$ is a subcube. We refer to this set as the \textbf{subcube determined by $p_0$}. 

Finally, we define the slice $S_{\mathcal{T},\mathcal{T}'} \subset C_{\mathcal{T}}$ for each ordered pair of independent subsets $\mathcal{T}, \mathcal{T}' \subset \mathcal{T}_0$:
Let $C_1,\dots ,C_k$ be the components of $\mathfrak{G}_{\mathcal{T} \cap \mathcal{T}'}$. 
Define $S_{\mathcal{T},\mathcal{T}'}$ to be the subset of $C_\mathcal{T}$ consisting of all points satisfying the following inequalities for each $i,j=1,\dots,k$:
$$\begin{cases}  
x^{(\mathcal{T})}_{\{d_i,d_j\}} \leq \min\{x^{(\mathcal{T})}_{\{d_i,d'_i\}},x^{(\mathcal{T})}_{\{d_j,d'_j\}}\} & \text{if } \exists ~\{d_i,d_j\} \in \mathcal{T}\text{ and } \{d'_i,d'_j\} \in \mathcal{T}' \text{ connecting } C_i ~\&~ C_j \\
x^{(\mathcal{T})}_{\{d_i,d_j\}} = 0 & \text{if } \exists ~\{d_i,d_j\} \in \mathcal{T}  \text{ but no } \tau' \in \mathcal{T}' \text{ connecting } C_i ~\&~ C_j .
\end{cases}$$

\begin{df} \label{defn:branchedcube}
A \textbf{branched cube} $B_\mathcal{G}$ associated to $\mathcal{G}$ is a metric space that is the union of cubes $C_\mathcal{T}$, as $\mathcal{T}$ ranges over all independent subsets of $\mathcal{T}_0$, for some choice of  $\ell(\tau)\in\RR_{>0}$ associated to the elements $\tau \in \mathcal{T}_0$ satisfying \Cref{eq:foldlengthconstraint}, so that for each pair of independent subsets $\mathcal{T}_1, \mathcal{T}_2 \subset \mathcal{T}_0$, we have 
$$C_{\mathcal{T}_1} \cap C_{\mathcal{T}_2} = S_{\mathcal{T}_1,\mathcal{T}_2} = S_{\mathcal{T}_2,\mathcal{T}_1},$$ with the function $x^{(\mathcal{T}_1)}_{\{d_i,d_j\}}$ identified with the function $x^{(\mathcal{T}_2)}_{\{d'_i,d'_j\}}$ for each $i,j$ for which there is a $\{d_i,d_j\} \in \mathcal{T}$ and a $\{d'_i,d'_j\} \in \mathcal{T}'$ connecting $C_i$ and $C_j$.
This ensures that the functions $x^{(\mathcal{T})}_\tau$, as $\mathcal{T}$ ranges over all independent subsets of $\mathcal{T}_0$, can be patched together into a function $x_\tau$ on $B_\mathcal{G}$.

The \textbf{splitting face} of $B_\mathcal{G}$ associated to partitions $G_i = G_{i,1} \sqcup ... \sqcup G_{i,m_i}$ is the subspace defined by $x_\tau = 0$ whenever $\tau \not\subset G_{i,m_i}$ for all $i$.
In particular, the \textbf{splitting vertex} of $B_G$ is the splitting face associated to the partition of each $G_i$ into one-element sets, defined by $x_\tau = 0$ for all $\tau$.
The \textbf{folding face} of $B_\mathcal{G}$ associated to an independent subset $\mathcal{T} \subset \mathcal{T}_0$ is the subspace defined by $x_\tau = \ell(\tau)$ for all $\tau \in \mathcal{T}$.

A \textbf{fold path} in $B_\mathcal{G}$ is an oriented path that is locally a fold path in each cube that it lies in.
Let $p_0$ be a point in $B_\mathcal{G}$. Consider the set of points $p$ in $B_\mathcal{G}$ for which there is a folding path from $p_0$ to $p$. This set is also the union of subcubes determined by $p_0$ in each cube that contains $p_0$. Note that this set is in general not a branched cube. We refer to it as the \textbf{generalized branched cube} determined by $p_0$. 
\end{df}

\begin{lem} \label{lem:branchcubesprop}
We have the following properties of branched cubes.
\begin{enumerate}
    \item A finite product of branched cubes is a branched cube. 
    \item Each splitting face of a branched cube is a branched cube.
    \item Each folding face of a branched cube is a branched cube.
\end{enumerate}
\end{lem}
\begin{proof}
For (1), we have $B_\mathcal{G} \times B_{\mathcal{G}'} \cong B_{\mathcal{G} \sqcup \mathcal{G}'}$.
Because of this, it suffices to show (2) and (3) in the case when $\mathcal{G}$ just contains one finite set $G$.

For (2), consider the splitting face $S$ associated to a partition $G = G_1 \sqcup ... \sqcup G_m$. Then for every independent subset $\mathcal{T} \subset \mathcal{T}_0$, we have $C_\mathcal{T} \cap S = C_{\mathcal{T}_1 \sqcup ... \sqcup \mathcal{T}_m} = C_{\mathcal{T}_1} \times ... \times C_{\mathcal{T}_m}$ where $\mathcal{T}_i$ is the subset of $\mathcal{T}$ consisting of elements lying within $G_i$, which is independent as a subset of $\Lambda^2 G_i$. 
Conversely, given independent subsets $\mathcal{T}_i \subset \Lambda^2 G_i$, we have that $\mathcal{T}_1 \sqcup ... \sqcup \mathcal{T}_m$ is an independent subset of $\mathcal{T}_0$.
Thus $S = \bigcup_{\mathcal{T}_1,...,\mathcal{T}_m} C_{\mathcal{T}_1} \times ... \times C_{\mathcal{T}_m} \cong B_{G_1} \times ... \times B_{G_m}$.

For (3), we first show this for a folding face $F$ associated to a single element $\{d_1,d_2\}$.
Let $G_1 = \{d \in G \mid \ell({\{d,d_1\})} > \ell({\{d_1,d_2\}})\}$ and $G_2 = \{d \in G \mid \ell({\{d,d_1\})} \leq \ell({\{d_1,d_2\}})\}$.
Thus $G = G_1 \sqcup G_2$, with $d_1 \in G_1$ and $d_2 \in G_2$.
The same argument as in \Cref{lemma:splitinsidebranchedcubecriterion} shows that $F \subset \bigcup_{\mathcal{T}_1,\mathcal{T}_2} C_{\mathcal{T}_1 \cup \{d_1,d_2\} \cup \mathcal{T}_2}$ where the union ranges over all independent subsets $\mathcal{T}_1 \subset \Lambda^2 G_1$ and $\mathcal{T}_2 \subset \Lambda^2 G_2$.
Now $C_{\mathcal{T}_1 \cup \{d_1,d_2\} \cup \mathcal{T}_2} \cap F \cong C_{\mathcal{T}_1} \times C_{\mathcal{T}_2}$. Thus $$F = \bigcup_{\mathcal{T}_1,\mathcal{T}_2} (C_{\mathcal{T}_1 \cup \{d_1,d_2\} \cup \mathcal{T}_2} \cap F) \cong \bigcup_{\mathcal{T}_1,\mathcal{T}_2} (C_{\mathcal{T}_1} \times C_{\mathcal{T}_2}) = B_{G_1} \times B_{G_2}.$$ 
Moreover, under this isomorphism, the maximum value of $x_\tau$ for $\tau \in \Lambda^2 G_i$ equals $\ell(\tau)$. Hence, for folding faces of $B_G$ associated to larger subsets, we can run this argument inductively.
\end{proof}

\subsection{The definition of a cubist complex} \label{subsec:cubistcomplexdefn}

\begin{df} \label{defn:cubistcomplex}
A \textbf{cubist complex} is an ordered pair $(X,\mathcal{B})$ where $X$ is a topological space and  $\mathcal{B}$ is a collection of subspaces, 
each homeomorphic to a branched cube, and satisfying each of (a)-(d):
\begin{enumerate}
    \item[a.] The space $X$ is the disjoint union of the interiors of the elements of $\mathcal{B}$.
    \item[b.] For each $B \in \mathcal{B}$, each splitting face of $B$ is an element of $\mathcal{B}$, while each folding face of $B$ is a union of elements of $\mathcal{B}$.
    \item[c.] For each $B_1, B_2 \in \mathcal{B}$, either $B_1 \cap B_2 = \varnothing$, or $B_1 \cap B_2 \in \mathcal{B}$ and is a sub-branched cube of $B_1$ and $B_2$.
    \item[d.] For each 0-dimensional branched cube $\{v\} \in \mathcal{B}$, the set of all $B\in\mathcal{B}$ for which $v$ lies in the interior of a folding face of $B$, once partially ordered by inclusion, has a unique maximal element $B(v)$. 
\end{enumerate}
\end{df}

\parpic[r]{\resizebox{!}{3cm}{
\begingroup%
  \makeatletter%
  \providecommand\color[2][]{%
    \errmessage{(Inkscape) Color is used for the text in Inkscape, but the package 'color.sty' is not loaded}%
    \renewcommand\color[2][]{}%
  }%
  \providecommand\transparent[1]{%
    \errmessage{(Inkscape) Transparency is used (non-zero) for the text in Inkscape, but the package 'transparent.sty' is not loaded}%
    \renewcommand\transparent[1]{}%
  }%
  \providecommand\rotatebox[2]{#2}%
  \newcommand*\fsize{\dimexpr\f@size pt\relax}%
  \newcommand*\lineheight[1]{\fontsize{\fsize}{#1\fsize}\selectfont}%
  \ifx\svgwidth\undefined%
    \setlength{\unitlength}{143.83106142bp}%
    \ifx\svgscale\undefined%
      \relax%
    \else%
      \setlength{\unitlength}{\unitlength * \real{\svgscale}}%
    \fi%
  \else%
    \setlength{\unitlength}{\svgwidth}%
  \fi%
  \global\let\svgwidth\undefined%
  \global\let\svgscale\undefined%
  \makeatother%
  \begin{picture}(1,1.0612978)%
    \lineheight{1}%
    \setlength\tabcolsep{0pt}%
    \put(0,0){\includegraphics[width=\unitlength,page=1]{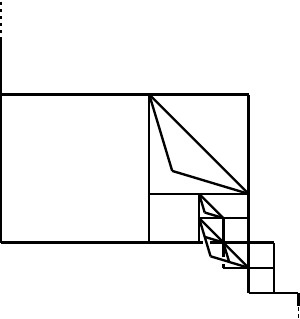}}%
  \end{picture}%
\endgroup%
}}
For a cubist complex $(X,\mathcal{B})$, $X$ carries a natural PL structure by declaring each branched cube in $\mathcal{B}$ to be PL-isomorphic to a branched cube.
We will always implicitly endow $X$ with this PL-structure.
In particular, an \textbf{isomorphism} of cubist complexes is a PL isomorphism sending branched cubes to branched cubes. Finally, whenever the collection $\mathcal{B}$ of branched cubes is clear, we simply refer to $X$ as a cubist complex.

An example of a cubist complex is shown to the right.
Note that the `local dimension' of a cubist complex is allowed to jump. In the image, the local dimension at some parts is $2$, while at other parts it is $1$.

Recall from the introduction the differences between cubist complexes and cube complexes. We make now item (3) from that comparison precise via the following lemma.

\begin{lem} \label{lem:zenotic}
Let $(X,\mathcal{B})$ be a cubist complex. Let $Y$ be a subspace of $X$ that is a union of elements in $\mathcal{B}$ and which is PL isomorphic to a generalized branched cube itself. Then for each 0-dimensional branched cube $\{v\} \in \mathcal{B}$, the generalized sub-branched cube of $Y$ determined by $v$ is a union of elements of $\mathcal{B}$.
\end{lem}

\begin{proof}
Let $N$ be the number of elements of $\mathcal{B}$ that intersect the interior of the sub-branched cube $Z$ of $Y$ 
determined by $v$.
We proceed by induction on $(\dim Z, N)$, lexicographically ordered.
If $\dim Z=0$, then $Z=\{v\}$ is a $0$-dimensional element of $\mathcal{B}$.

\parpic[r]{\selectfont\fontsize{8pt}{8pt}
\begingroup%
  \makeatletter%
  \providecommand\color[2][]{%
    \errmessage{(Inkscape) Color is used for the text in Inkscape, but the package 'color.sty' is not loaded}%
    \renewcommand\color[2][]{}%
  }%
  \providecommand\transparent[1]{%
    \errmessage{(Inkscape) Transparency is used (non-zero) for the text in Inkscape, but the package 'transparent.sty' is not loaded}%
    \renewcommand\transparent[1]{}%
  }%
  \providecommand\rotatebox[2]{#2}%
  \newcommand*\fsize{\dimexpr\f@size pt\relax}%
  \newcommand*\lineheight[1]{\fontsize{\fsize}{#1\fsize}\selectfont}%
  \ifx\svgwidth\undefined%
    \setlength{\unitlength}{141.24226728bp}%
    \ifx\svgscale\undefined%
      \relax%
    \else%
      \setlength{\unitlength}{\unitlength * \real{\svgscale}}%
    \fi%
  \else%
    \setlength{\unitlength}{\svgwidth}%
  \fi%
  \global\let\svgwidth\undefined%
  \global\let\svgscale\undefined%
  \makeatother%
  \begin{picture}(1,0.79467005)%
    \lineheight{1}%
    \setlength\tabcolsep{0pt}%
    \put(0,0){\includegraphics[width=\unitlength,page=1]{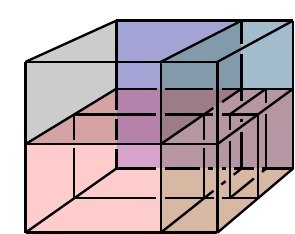}}%
    \put(0.02816206,0.59406878){\color[rgb]{0,0,0}\makebox(0,0)[lt]{\lineheight{1.25}\smash{\begin{tabular}[t]{l}$v$\end{tabular}}}}%
    \put(0.37422867,0.76198009){\color[rgb]{0,0,1}\makebox(0,0)[lt]{\lineheight{1.25}\smash{\begin{tabular}[t]{l}$v_3$\end{tabular}}}}%
    \put(-0.00321504,0.29411037){\color[rgb]{1,0,0}\makebox(0,0)[lt]{\lineheight{1.25}\smash{\begin{tabular}[t]{l}$v_1$\end{tabular}}}}%
    \put(-0.00321504,0.13480931){\color[rgb]{1,0,0}\makebox(0,0)[lt]{\lineheight{1.25}\smash{\begin{tabular}[t]{l}$Z_1$\end{tabular}}}}%
    \put(0.58662973,0.76198009){\color[rgb]{0,0,1}\makebox(0,0)[lt]{\lineheight{1.25}\smash{\begin{tabular}[t]{l}$Z_3$\end{tabular}}}}%
    \put(0.18013095,0.69711441){\color[rgb]{0,0,0}\makebox(0,0)[lt]{\lineheight{1.25}\smash{\begin{tabular}[t]{l}$B_0$\end{tabular}}}}%
    \put(0.67428064,0.60469543){\color[rgb]{0,0.50196078,0}\makebox(0,0)[lt]{\lineheight{1.25}\smash{\begin{tabular}[t]{l}$Z_2$\end{tabular}}}}%
    \put(0.49374036,0.60469548){\color[rgb]{0,0.50196078,0}\makebox(0,0)[lt]{\lineheight{1.25}\smash{\begin{tabular}[t]{l}$v_2$\end{tabular}}}}%
  \end{picture}%
\endgroup%
}
For the induction step, consider the elements of $\mathcal{B}$ contained in $Y$ that have splitting vertex at $v$, and let $B_0$ be the element among these of the highest dimension.
Here we use that $Y$ is PL isomorphic to a generalized branched cube to ensure that $B_0$ is uniquely defined.
Let $e_1,...,e_n$ be the collection of 1-dimensional splitting edges of $B_0$. Let $v_i$ be the endpoint of $e_i$ not $v$, and let $Z_i$ be the sub-branched cube of $Y$ determined by $v_i$. 
For each $i$, if $e_i$ is an entire splitting edge of $Z$, then $Z_i$ has lower dimension than $Z$, otherwise $Z_i$ intersects strictly less elements of $\mathcal{B}$ in their interior than $Z$ does, since $B_0$ intersects the interior of $Z$ but not that of $Z_i$. 
By our induction hypothesis, each $Z_i$ is a union of elements of $\mathcal{B}$. This implies $Z=B_0 \cup Z_1 \cup ... \cup Z_m$ is a union of elements of $\mathcal{B}$.
\end{proof}

We refer to the union of the $k$-dimensional branched cubes in a cubist complex $X$ as the \textbf{$k$-skeleton} of $X$.
We caution that the $k$-skeleton of $X$ as a cubist complex is different from the $k$-skeleton of $X$ when considered as a cell complex.
For example, if $B$ is a branched cube as in \Cref{eg:2dimbranchedcube}, then the 1-cell that the branching happens along does not lie in the 1-skeleton of $X$ as a cubist complex.

The $1$-skeleton of $X$ can be given the structure of a directed graph $\mathfrak{g}_X$ by orienting each 1-dimensional branched cube so that it is a folding path in each branched cube that contains it. 

The next two lemmas concern the behaviour of $\mathfrak{g}_X$ in each branched cube.

\begin{lem} \label{lemma:edgepathexists}
Let $(X,\mathcal{B})$ be a cubist complex. Let $Y$ be a subspace of $X$ that is a union of elements in $\mathcal{B}$ and which is PL-isomorphic to a generalized branched cube itself. Let $v$ be the splitting vertex of $Y$. Then for each vertex $w$ of $\mathfrak{g}_X$ lying in $Y$, there is a directed edge path in $\mathfrak{g}_X$ from $v$ to $w$. 
\end{lem}

\begin{proof}
Let $N$ be the number of elements of $\mathcal{B}$ intersecting the interior of $Y$. We proceed by induction on $(\dim Y, N)$, lexicographically ordered.
For $\dim Y = 0,1$ the lemma is clear. For the induction step, define the sub-branched cube $B_0$, the vertices $v_i$, and the generalized sub-branched cubes $Z_i$ as in \Cref{lem:zenotic}.
If $w=v$, the lemma holds trivially.
Otherwise, since $Y=B_0 \cup Z_1 \cup ... \cup Z_m$, $w$ lies in some $Z_i$. By induction, there is a directed edge path from the splitting vertex $v_i$ to $w$. Concatenate with the edge from $v$ to $v_i$.
\end{proof}

Furthermore, we claim that a directed edge path as in \Cref{lemma:edgepathexists} is unique up to sweeping the path across 2-dimensional branched cubes. More precisely, suppose $B$ is a 2-dimensional branched cube. Suppose $\beta$ and $\beta'$ are directed edge paths in $\mathfrak{g}_X$ with a common initial vertex $v_1$ and a common terminal vertex $v_2$, and each of which is the concatenation of one 1-dimensional splitting face of $B$ and one 1-dimensional folding face of $B$. 
Then for each edge path $\alpha_1$ with terminal vertex $v_1$ and each edge path $\alpha_2$ with initial vertex $v_2$, we say that the directed edge paths $\alpha_1 * \beta * \alpha_2$ and $\alpha_1 * \beta' * \alpha_2$ are related by \textbf{sweeping across $B$}. 

\begin{lem} \label{lemma:edgepathunique}
Let $(X,\mathcal{B})$ be a cubist complex. Let $Y$ be a subspace of $X$ that is a union of elements in $\mathcal{B}$ and which is PL isomorphic to a generalized branched cube itself. Let $v$ be the splitting vertex of $Y$, and let $w$ be some vertex lying in $Y$. 
Suppose $\gamma$ and $\gamma'$ are directed edge paths in $\mathfrak{g}_X$ from $v$ to $w$. Then $\gamma$ and $\gamma'$ are related by sweeping across finitely many 2-dimensional branched cubes. 
\end{lem}

\begin{proof}
Define $N$ as in \Cref{lemma:edgepathexists}. We proceed by induction on $(\dim Y, N)$, lexicographically ordered.
For $\dim Y = 0,1$ the lemma is clear. For the induction step, define the sub-branched cube $B_0$, the edges $e_i$, the vertices $v_i$, and the generalized sub-branched cubes $Z_i$ as in \Cref{lem:zenotic}.
If $w=v$, then the lemma holds trivially.
Otherwise the initial edges of $\gamma$ and $\gamma'$ must each be one of the $e_i$.

We first assume the initial edges of $\gamma$ and $\gamma'$ are the same $e_i$ (see \Cref{fig:edgepathunique} left).
In this case, $\gamma = e_i * \alpha$ and $\gamma' = e_i * \alpha'$ for some directed edge paths $\alpha$ and $\alpha'$ in $Z_i$ from $v_i$ to $w$.
By induction, $\alpha$ and $\alpha'$ are related by sweeping across finitely many 2-dimensional branched cubes, so the same holds for $\gamma$ and $\gamma'$.

It remains to prove the lemma for one choice of $\gamma$ whose initial edge is $e_i$ and for one choice of $\gamma'$ whose initial edge is $e_{i'}$, for each pair $(i,i')$.
Note that there is a vertex $u$ on $B_0$ such that the sub-branched cube $Q$ determined by $u$ equals $Z_i \cap Z_{i'}$. 
More concretely, there is a 2-dimensional splitting face $F$ of $B_0$ for which $e_i$ and $e_{i'}$ are among the 1-dimensional splitting faces of $F$. The vertex $u$ can be characterized as the intersection between the subcubes of $F$ determined by $v_i$ and $v_{i'}$.

By \Cref{lemma:edgepathexists}, there is a directed edge path $\alpha$ from $u$ to $w$. Meanwhile, there is a directed edge path $\beta$ from $e_i$ to $u$ and a directed edge path $\beta'$ from $e_{i'}$ to $u$. See \Cref{fig:edgepathunique} right.
The edge paths $e_i * \beta * \alpha$ and $e_{i'} * \beta' * \alpha$ are related by sweeping across $F$.
\end{proof}

\begin{figure}
    \centering
    \selectfont\fontsize{8pt}{8pt}
    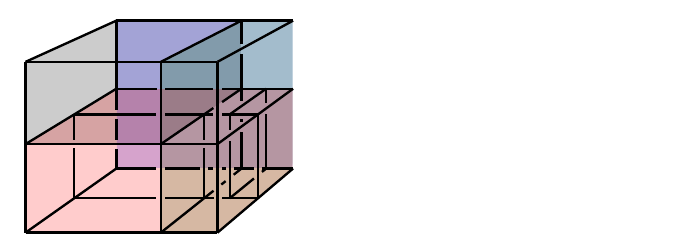
    \caption{For \Cref{lemma:edgepathunique}, we first argue the case when the initial edges of $\gamma$ and $\gamma'$ agree (left), we then argue one case when the initial edges of $\gamma$ and $\gamma'$ differ (right).}
    \label{fig:edgepathunique}
\end{figure}

Finally, we introduce the following condition, which will come into play in \Cref{subsec:cardiovascularsystem}.

\begin{df} \label{df:periodiccubistcomplex}
A \textbf{periodic cubist complex} is a connected cubist complex $(X,\mathcal{B})$ together with an isomorphism $\phi:X \to X$ such that the $\mathbb{Z}$-action generated by $\phi$ is free and cocompact. 
\end{df}

\subsection{The axis bundle is a cubist complex} \label{subsec:axisbundleiscubistcomplex}

Let $\vphi\in\out$ be nongeometric fully irreducible.
We briefly return to the axis bundle setting to show that it has a cubist complex structure.

By \Cref{sec:axisbundlebranchedcubes} there is a branched cube $B_T$ for each fully preprincipal $T \in \Av$.
\Cref{lem:branchcubesprop} implies each splitting face of $B_T$ is a branched cube.
Let $\mathcal{B}_\phi$ denote the collection of the $B_T$ and their splitting faces, as $T$ ranges over all fully preprincipal $T \in \Av$. More generally, given any local decomposition $\mG$ of $IW(\vphi)$, denote by $\mG \mathcal{B}_\phi$ the collection of the $B_T$ and their splitting faces, as $T$ ranges over all fully preprincipal elements at least as split as $\mG$. To distinguish a branched cube $B_T$ from a general element of $(\mG)\mathcal{B}_\phi$, which can be a proper splitting face of $B_T$, we call the former \textbf{primary branched cubes} in the following proof.

\begin{thm}\label{t:CubistDecomp}
Suppose that $r\geq 3$ and $\vphi\in\out$ is nongeometric fully irreducible.
Then $(\mathcal{A}_{\vphi}, \mathcal{B}_{\vphi})$ is a cubist complex.  
Furthermore, the action of $\phi$ on $\mathcal{A}_{\vphi}$ makes it into a periodic cubist complex.

More generally, for each local decomposition $\mG$, we have that $(\mG\mathcal{A}_{\vphi}, \mG\mathcal{B}_{\vphi})$ is a cubist complex, and the action of $\phi$ on $\mG\mathcal{A}_{\vphi}$ makes it into a periodic cubist complex.
\end{thm}
\begin{proof}
\Cref{defn:cubistcomplex}(a) follows from \Cref{prop:branchedcubescoveraxisbundle} and \Cref{prop:branchedcubedisjointint}. 
\Cref{defn:cubistcomplex}(b) follows from the definition of the branched cubes and \Cref{prop:foldfacebranchedcubes}.
\Cref{defn:cubistcomplex}(d) is a special case of \Cref{prop:uniquesuccessor}.

We now show \Cref{defn:cubistcomplex}(c).
Since branched cubes have disjoint interiors, two branched cubes $B_1$ and $B_2$ can only possibly intersect along a union of branched cubes, each of them being (1) a splitting face of $B_1$ and a branched cube in a folding face of $B_2$ or (2) a branched cube in the interior of a folding face of $B_1$ and of $B_2$.
If a branched cube of type (1) arises, \Cref{prop:foldfacevertex} implies $B_1 \cap B_2$ is a sub-branched cube. 
If a branched cube of type (2) arises, then by \Cref{prop:uniquesuccessor}, the splitting vertex of $B_1$ and $B_2$ coincide, so $B_1$ and $B_2$ are splitting faces of some primary branched cube $B_T$, in which case they intersect in a splitting face of $B_T$. 
(In fact, this argument shows there are never branched cubes of type (2).)

It remains to show that $(\mG \mathcal{A}_{\vphi},\phi)$ is a periodic cubist complex. Connectedness follows from \Cref{prop:stableaxisbundleconnected}. Freeness follows from \Cref{P:6.2}. Finally, cocompactness follows from \Cref{lem:fpplocallyfinite}.
\end{proof}

\subsection{Cardiovascular system} \label{subsec:cardiovascularsystem}

Let $(X,\mathcal{B})$ be a cubist complex. We define a directed graph $\mathfrak{c}_X$ as follows:
\begin{itemize}
    \item The vertex set of $\mathfrak{c}_X$ is the $0$-skeleton of $X$.
    \item For each vertex $v$, let $B(v)$ be the branched cube that is maximal with respect to the property that $v$ is contained in the interior of a folding face of $B(v)$, and let $S(v)$ be the splitting vertex of $B(v)$. By \Cref{defn:cubistcomplex}(4), $S(v)$ is well-defined. If $S(v) \neq v$, then we add a directed edge from $v$ to $S(v)$. (If $S(V) = v$ then we do not add any edges.)
\end{itemize}
\parpic[r]{\resizebox{!}{3.25cm}{
\begingroup%
  \makeatletter%
  \providecommand\color[2][]{%
    \errmessage{(Inkscape) Color is used for the text in Inkscape, but the package 'color.sty' is not loaded}%
    \renewcommand\color[2][]{}%
  }%
  \providecommand\transparent[1]{%
    \errmessage{(Inkscape) Transparency is used (non-zero) for the text in Inkscape, but the package 'transparent.sty' is not loaded}%
    \renewcommand\transparent[1]{}%
  }%
  \providecommand\rotatebox[2]{#2}%
  \newcommand*\fsize{\dimexpr\f@size pt\relax}%
  \newcommand*\lineheight[1]{\fontsize{\fsize}{#1\fsize}\selectfont}%
  \ifx\svgwidth\undefined%
    \setlength{\unitlength}{145.26045191bp}%
    \ifx\svgscale\undefined%
      \relax%
    \else%
      \setlength{\unitlength}{\unitlength * \real{\svgscale}}%
    \fi%
  \else%
    \setlength{\unitlength}{\svgwidth}%
  \fi%
  \global\let\svgwidth\undefined%
  \global\let\svgscale\undefined%
  \makeatother%
  \begin{picture}(1,1.05085458)%
    \lineheight{1}%
    \setlength\tabcolsep{0pt}%
    \put(0,0){\includegraphics[width=\unitlength,page=1]{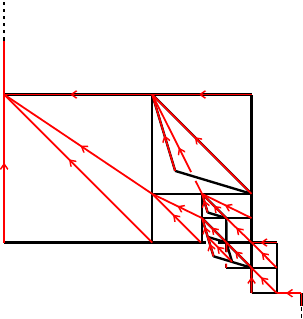}}%
  \end{picture}%
\endgroup%
}}
The graph $\mathfrak{c}_X$ is naturally seen as a subspace of $X$ by placing the directed edges along straight lines.
We refer to $\mathfrak{c}_X$ as the \textbf{cardiovascular system} of $X$.
See the right-hand image for an example of a cardiovascular system.

A crucial property of the cardiovascular system is each vertex having at most one outgoing edge, but possibly multiple incoming edges. 
In particular, each vertex $v$ has a unique `successor' $S(v)$.
Iteratively determining edges as such, we obtain a directed edge path $\gamma_v = (S^i(v))_{i \geq 0}$, which is a ray if each $S^{i+1}(v) \neq S^i(v)$, and is a finite path otherwise.
Conversely, each maximal directed edge path arises as such.

For the rest of this section, we restrict to the setting where $(X,\phi)$ is periodic (recall \Cref{df:periodiccubistcomplex}).
The aim is to deduce the cardiovascular system properties in this setting.

The cocompactness of the $\mathbb{Z}$-action implies each directed edge path $\gamma_v=(S^i(v))_{i \geq 0}$ is eventually periodic:

\begin{prop} \label{prop:rayseventuallyperiodic}
Let $(X, \phi)$ be a periodic cubist complex. For each vertex $v$, there exist $P,N \in \mathbb{Z}$, and an $i_0 \geq 0$, such that $S^{i+P}(v)=\phi^N(S^i(v))$ for all $i \geq i_0$. 
\end{prop}

\begin{proof}
Since $\langle \phi \rangle$ acts cocompactly on $X$, the set $V$ of $\langle \phi \rangle$-orbits of vertices is finite. We define a map $\mathbb{Z}_{\geq 0} \to V$ by sending $i$ to the orbit of $S^i(v)$. By the pigeonhole principle, there are integers $i_1 > i_2$ such that $S^{i_1}(v)$ and $S^{i_2}(v)$ lie in the same orbit, i.e. there is some $N\in\mathbb{Z}$ such that $S^{i_1}(v) = \phi^N(S^{i_2}(v))$. 

Meanwhile, since $\phi$ preserves the cubist complex structure of $X$, we have $\phi(S(v))=S(\phi(v))$ for any vertex $v$. Applying this fact repeatedly, we deduce that, for each $i \geq i_2$,
\begin{center}
$S^{i+(i_1-i_2)}(v) = S^{i-i_2}(\phi^N(S^{i_2}(v))) = \phi^N(S^i(v)).$   
\end{center}
\end{proof}

Next, we show that the connectedness of $X$ implies that each directed path $\gamma_v=(S^i(v))_{i \geq 0}$ is actually a ray.
To this end, we introduce a measure of distance between vertices of $\mathfrak{c}_X$.

\begin{df}
Let $\alpha$ be an edge path in $\Gamma$. Note that $\alpha$ may not be a directed edge path, i.e. it can traverse some edges of $\Gamma$ in the opposite direction as their prescribed orientations.

We define the \textbf{combinatorial length} of $\alpha$ to be 
$$\min\{n \mid \alpha = \alpha_1 * ... * \alpha_n, \text{
for some monotone edge paths $\alpha_j$ each lying in a single branched cube
}\}$$
where by a monotone edge path, we mean an edge path $\beta$ where either $\beta$ or $-\beta$ is a directed edge path.

The \textbf{combinatorial distance} between two vertices $v_0$ and $v_1$ is the minimum combinatorial length of paths between them. It is straightforward to verify that the combinatorial distance is a metric.
\end{df}

\begin{lem} \label{lem:combindistone}
Let $v_0$ and $v_1$ be two vertices of $X$.
Suppose there is a directed edge path $\alpha$ from $v_0$ to $v_1$ that lies in one branched cube $B$. Then either $S(v_0)=v_1$, or there is a directed edge path $\beta$ from $v_1$ to $S(v_0)$ that lies in one branched cube, so that the $\mathfrak{c}_X$-edge from $v_0$ to $S(v_0)$ is homotopic to $\alpha * \beta$.
\end{lem}
\begin{proof}
Without generality loss, we can assume $B$ is the maximal branched cube with respect to the property that $v_0$ is in the interior of a folding face of $B$, so $S(v_0)$ is the splitting vertex of $B$. 
Since $\alpha$ is in $B$, we have $v_1$ is in $B$. By \Cref{lemma:edgepathexists}, there is a directed edge path $\beta$ from $v_1$ to $S(v_0)$ lying in $B$.
Since $\alpha*\beta$ and the $\mathfrak{c}_X$-edge from $v_0$ to $S(v_0)$ are paths in $B$ with the same initial and terminal vertices, and $B$ is contractible, they are homotopic. 
\end{proof}

\begin{cor} \label{cor:combindistoneflowforward}
Suppose $v_0$, $v_1\in VX$. If the combinatorial distance between $v_0$ and $v_1$ is $1$, then either $v_0=S(v_1)$ or $v_1=S(v_0)$ or $S(v_0)=S(v_1)$, or the combinatorial distance between $S(v_0)$ and $S(v_1)$ is $\leq 1$.
\end{cor}
\begin{proof}
Up to switching $v_0$ and $v_1$, there is a directed edge path $\alpha_0$ from $v_0$ to $v_1$ lying in one branched cube. By \Cref{lem:combindistone},either $v_1=S(v_0)$, or there is a directed edge path $\beta_0$ from $v_1$ to $S(v_0)$ lying in one branched cube. By \Cref{lem:combindistone} again, either $S(v_0)=S(v_1)$, or there a directed edge path $\alpha_1$ from $S(v_0)$ to $S(v_1)$ lying in one branched cube.
\end{proof}

\begin{prop} \label{prop:raysarerays}
Let $(X, \phi)$ be a periodic cubist complex. For each vertex $v$, the directed edge path $\gamma_v=(S^i(v))_{i \geq 0}$ is a ray, i.e. $S^{i+1}(v) \neq S^i(v)$ for each $i \geq 0$.
\end{prop}

\begin{proof}
Suppose otherwise that there is a vertex $v_0$ for which $S^{i+1}(v_0) = S^i(v_0)$ for some $i \geq i_0$. Let $v_1$ be a vertex that is of combinatorial distance $1$ away from $v_0$. We claim that the directed edge path from $v_1$ stabilizes at the same point as that from $v_0$, i.e. $S^{i+1}(v_1) = S^i(v_1) = S^i(v_0)$ for $i \geq i_1$.

By \Cref{cor:combindistoneflowforward}, $S^{i_0}(v_0)$ and $S^{i_0}(v_1)$ have combinatorial distance at most $1$. If $S^{i_0}(v_0) = S^{i_0}(v_1)$ then our claim is clear. Otherwise there is a directed edge path $\alpha$ either from $S^{i_0}(v_0)$ to $S^{i_0}(v_1)$ or from $S^{i_0}(v_1)$ to $S^{i_0}(v_0)$ that lies in one branched cube. The former cannot be true since $S^{i_0}(v_0)=S^{i_0+1}(v_0)$.
But then, applying \Cref{lem:combindistone}, there is a directed edge path $\beta$ from $S^{i_0}(v_0)$ to $S^{i_0+1}(v_1)$, thus $S^{i_0+1}(v_0) = S^{i_0}(v_0) = S^{i_0+1}(v_1)$ since again $S^{i_0}(v_0)=S^{i_0+1}(v_0)$. 

Now let $\alpha$ be an edge path between $v_0$ and $\phi(v_0)$. Let $v_0$, $v_1$, ..., $v_m=\phi(v_0)$ be the sequence of vertices on $\alpha$. For each $i$, the combinatorial distance between $v_{i-1}$ and $v_i$ is $1$, hence by applying our claim in the first paragraph repeatedly, we have $S^i(\phi(v_0)) = S^i(v_0)$ for all large $i$. 
Thus $S^{i_0}(v_0) = S^i(v_0) = S^i(\phi(v_0)) = \phi(S^i(v_0)) = \phi(S^{i_0}(v_0))$, that is, $S^{i_0}(v_0)$ is a fixed point of $\phi$. This contradicts freeness of $\phi$.
\end{proof}

An \textbf{artery} of the cardiovascular system is a periodic directed edge path, i.e. directed edge path $A = (v_i)_{i \in \mathbb{Z}}$ for which there exist $P,N$ such that $v_{i+P}=\phi^N(v_i)$ for all $i$.
Then $P$ is the \textbf{period} of $A$ and $N$ the \textbf{order}.

\Cref{prop:raysarerays} implies that $N \neq 0$, since if $N=0$, we have that $S^P(v_i) = v_{i+P} = v_i$, and the directed edge path starting at $v_i$ is finite.
Thus we can define the \textbf{average period} of $A$ to be $\frac{P}{N}$.

\begin{prop} \label{prop:arteryexist}
Let $(X, \phi)$ be a periodic cubist complex. There is at least one, but finitely many arteries in the cardiovascular system of $X$.
\end{prop}
\begin{proof}
To show that there is at least one artery, take some vertex $v$ and apply \Cref{prop:rayseventuallyperiodic} to the directed edge ray $r=(S^i(v))_{i \geq 0}$ to get values of $P,N,i_0$ so that $S^{i+P}(v)=\phi^N(S^i(v))$ for all $i \geq i_0$. 
For each $i$, we define a vertex $v_i$ by picking $k$ large enough so that $i+kP \geq i_0$, and setting $v_i = \phi^{-kN}(S^{i+kP}(v))$.
Note that if $k_1<k_2$ are integers such that $i+k_1P, i+k_2P \geq i_0$, then
\begin{align*}
\phi^{-k_2N}(S^{i+k_2P}(v)) = \phi^{-k_2N}(S^{i+(k_2-1)P+P}(v)) &= \phi^{-k_2N} \phi^N(S^{i+(k_2-1)P}(v)) \\
&= \phi^{-(k_2-1)N}(S^{i+(k_2-1)P}(v)) = ... = \phi^{-k_1N}(S^{i+k_1P}(v))
\end{align*}
so $v_i$ is well-defined.
The following similar computation
shows that $(v_i)_{i \in \mathbb{Z}}$ is an artery:
$$v_{i+P} = \phi^{-kN}(S^{i+P+kP}(v)) = \phi^{-(k-1)N}(S^{i+kP}(v)) = \phi^N(v_i).$$ 

To show finiteness, we first claim that there is a uniform bound on the order of the arteries in $X$:
Suppose two arteries $A$ and $A'$ pass through the same $\langle \phi \rangle$-orbit of vertices, say $v \in A$ and $v' \in A'$ where $v'=\phi^q(v)$. Then $A$ and $A'$ share the same order. Indeed, if $S^{i+P}(v)=\phi^N(S^i(v))$ then 
$$S^{i+P}(v')=S^{i+P}(\phi^q(v))=\phi^q(S^{i+P}(v))=\phi^{N+q}(S^i(v))=\phi^N(S^i(\phi^q(v)))=\phi^N(S^i(v')).$$
Thus, the order of $A'$ divides that of $A$. Symmetrically, the order of $A$ divides that of $A'$, so they must coincide.
Our claim now follows from there being only finitely many $\langle \phi \rangle$-orbits of vertices.

Let $N_0$ be the lowest common multiple of the orders of all arteries. Let $V_0$ be the finite set of $\langle \phi^{N_0} \rangle$-orbits of vertices. If there are more than $|V_0|$ arteries, then two of them must pass through the same $\langle \phi^{N_0} \rangle$-orbit of vertices, but then they would actually share some vertex, which would imply the two arteries coincide.
\end{proof}

By \Cref{lemma:edgepathexists}, each artery $A$ can be homotoped into the reverse of a directed edge line $L$ of the 1-skeleton $\Gamma$. Indeed, one can replace each edge $v \to S(v)$ by the reverse of a homotopic directed edge path within the same cube. In this context we call $L$ a \textbf{simple factorization} of $A$.
\Cref{lemma:edgepathunique} implies that any two simple factorizations of a common artery $A$ are related by sweeping across 2-dimensional branched cubes.
The last goal of this section is to show that any two simple factorizations of any two arteries also relate in this way.

Recall the notion of the combinatorial distance between two vertices.
We define the \textbf{combinatorial distance} between two arteries to be the minimum combinatorial distance between their vertices.

\begin{lem} \label{lem:distoneartery}
Let $(X, \phi)$ be a periodic cubist complex. Suppose $A$ and $A'$ are arteries combinatorial distance $1$ apart. Then $A$ and $A'$ have a common average period and they admit a common simple factorization. 
\end{lem}

\begin{proof}
By definition, up to switching $A$ and $A'$, there are vertices $v \in A$ and $v' \in A'$, and a directed edge path $\alpha_0$ from $v$ to $v'$ that lies in one branched cube. By \Cref{lem:combindistone}, either $S(v)=v'$, or there is a directed edge path $\beta_0$ from $v'$ to $S(v)$ so that the $\mathfrak{c}_X$-edge from $v$ to $S(v)$ is homotopic to $\alpha_0 * \beta_0$. The former case cannot happen here or we would have $A=A'$. Now apply \Cref{lem:combindistone} to $\beta_0$ to obtain a directed edge path $\alpha_1$ from $S(v)$ to $S(v')$ so that the $\mathfrak{c}_X$-edge from $v'$ to $S(v')$ is homotopic to $\beta_0 * \alpha_1$. Repeating this argument, we have edge paths $\alpha_k$ from $S^k(v)$ to $S^k(v')$ and $\beta_k$ from $S^k(v')$ to $S^{k+1}(v)$ such that the $\mathfrak{c}_X$-edge from $S^k(v)$ to $S^{k+1}(v)$ is homotopic to $\alpha_k * \beta_k$ and the $\mathfrak{c}_X$-edge from $S^k(v')$ to $S^{k+1}(v')$ is homotopic to $\beta_k * \alpha_{k+1}$.

Let $P$ and $N$ be the period and order of $A$, and let $P'$ and $N'$ be the period and order of $A'$.
For each $q$, we have $\phi^{-qNP'}(\alpha_{qPP'})$ is an edge path from $\phi^{-qNP'}(S^{qPP'}(v)) = v$ to $\phi^{-qNP'}(S^{qPP'}(v')) = \phi^{q(PN'-NP')}(v')$ that lies in a branched cube. 
Since there are only finitely many such edge paths, for some $0 < q_1 < q_2$, we have $\phi^{-q_1NP'}(\alpha_{q_1PP'})=\phi^{-q_2NP'}(\alpha_{q_2PP'})$. In particular, $\phi^{q_1(PN'-NP')}(v') = \phi^{q_2(PN'-NP')}(v')$, thus $\phi^{(q_1-q_2)(PN'-NP')}(v')=v'$. 
Since $\phi$ cannot have fixed points, we have $PN'-NP'=0$, thus $\frac{P}{N}=\frac{P'}{N'}$.

Finally, the directed edge path $*_{j=-\infty}^{\infty} \phi^{j(q_2-q_1)NP'}(*_{k=q_1PP'}^{q_2PP'} (\alpha_k * \beta_k))$ is a common simple factorization for $A$ and $A'$.
\end{proof}

\begin{prop} \label{prop:arteryunique}
If $(X, \phi)$ is a periodic cubist complex and $A$, $A'$ are arteries, then $A$ and $A'$ have a common average period. Their simple factorizations relate by sweeping across 2-dimensional branched cubes.
\end{prop}
\begin{proof}
We first claim there are arteries $A_0=A, A_1,...,A_n=A'$ such that $A_{i-1}$ and $A_i$ are of combinatorial distance $1$ apart for each $1\leq i\leq n$. To see this, let $\alpha$ be an edge path between $A$ and $A'$. Let $v_0$,...,$v_m$ be the sequence of vertices in $\alpha$, where $v_0 \in A$ and $v_m \in A'$. For each $i=0,...,m$, let $r_i$ be the $\mathfrak{c}_X$-ray starting at $v_i$. By \Cref{prop:rayseventuallyperiodic}, each $r_i$ eventually converges into an artery $A_i$. For each $i$, the combinatorial distance between $v_{i-1}$ and $v_i$ is $1$, hence by \Cref{cor:combindistoneflowforward}, either $A_{i-1}=A_i$ or the combinatorial distance between $A_{i-1}$ and $A_i$ is $1$. Hence it suffices to discard any repeated arteries. The proposition now follows from \Cref{lemma:edgepathunique} and \Cref{lem:distoneartery}.
\end{proof}

\section{Examples} \label{sec:ex}

We show several examples of cubist decompositions of axis bundles, including examples with unique and nonunique arteries, branching, and varying local dimension. Each example is described by a train track map, with vertices of the graph carrying it ``blown up'' (see \cite{PffAutomata}), indicating how edge images cross them.

\begin{ex}[Lone axes] \label{ex:loneaxis}
By \cite{loneaxes}, the axis bundle of an ageometric, fully irreducible $\phi \in \out$ is a line if and only if 
 the index satisfies $i(\phi) = \frac{3}{2}-r$ and no component of  ${IW}(\phi)$ has a cut vertex.
See \cite{automaton} and \cite{PffAutomata} for concrete examples. In this case, the cubist decomposition of $\mathcal{A}_\phi$ is a union of 1-cubes and 0-cubes. The cardiovascular system coincides with $\mathcal{A}_\phi$ (with the additional data of being oriented toward the splitting direction). Thus, there is exactly one artery and it coincides with the axis bundle. 
\end{ex}

\begin{ex}[Multiple arteries] \label{ex:2dimbundlemultiartery}
Define $\phi \in \outt$ by 
$\phi(a)= cbca$, and
$\phi(b) = cbc$, and
$\phi(c) = ac$.
This 
\parpic[r]{\selectfont\fontsize{9pt}{9pt} 
\begingroup%
  \makeatletter%
  \providecommand\color[2][]{%
    \errmessage{(Inkscape) Color is used for the text in Inkscape, but the package 'color.sty' is not loaded}%
    \renewcommand\color[2][]{}%
  }%
  \providecommand\transparent[1]{%
    \errmessage{(Inkscape) Transparency is used (non-zero) for the text in Inkscape, but the package 'transparent.sty' is not loaded}%
    \renewcommand\transparent[1]{}%
  }%
  \providecommand\rotatebox[2]{#2}%
  \newcommand*\fsize{\dimexpr\f@size pt\relax}%
  \newcommand*\lineheight[1]{\fontsize{\fsize}{#1\fsize}\selectfont}%
  \ifx\svgwidth\undefined%
    \setlength{\unitlength}{191.73688032bp}%
    \ifx\svgscale\undefined%
      \relax%
    \else%
      \setlength{\unitlength}{\unitlength * \real{\svgscale}}%
    \fi%
  \else%
    \setlength{\unitlength}{\svgwidth}%
  \fi%
  \global\let\svgwidth\undefined%
  \global\let\svgscale\undefined%
  \makeatother%
  \begin{picture}(1,0.32339498)%
    \lineheight{1}%
    \setlength\tabcolsep{0pt}%
    \put(0,0){\includegraphics[width=\unitlength,page=1]{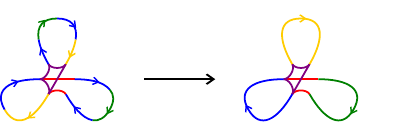}}%
    \put(0.73434892,0.29931404){\color[rgb]{1,0.8,0}\makebox(0,0)[lt]{\lineheight{1.25}\smash{\begin{tabular}[t]{l}$e_1$\end{tabular}}}}%
    \put(0.19683078,0.17421028){\color[rgb]{1,0.8,0}\makebox(0,0)[lt]{\lineheight{1.25}\smash{\begin{tabular}[t]{l}$e_1$\end{tabular}}}}%
    \put(0.08833064,0.00551598){\color[rgb]{1,0.8,0}\makebox(0,0)[lt]{\lineheight{1.25}\smash{\begin{tabular}[t]{l}$e_1$\end{tabular}}}}%
    \put(0.90395143,0.02259322){\color[rgb]{0,0.50196078,0}\makebox(0,0)[lt]{\lineheight{1.25}\smash{\begin{tabular}[t]{l}$e_2$\end{tabular}}}}%
    \put(0.28804816,0.01936858){\color[rgb]{0,0.50196078,0}\makebox(0,0)[lt]{\lineheight{1.25}\smash{\begin{tabular}[t]{l}$e_2$\end{tabular}}}}%
    \put(0.04135191,0.27152516){\color[rgb]{0,0.50196078,0}\makebox(0,0)[lt]{\lineheight{1.25}\smash{\begin{tabular}[t]{l}$e_2$\end{tabular}}}}%
    \put(0.5586778,0.02259322){\color[rgb]{0,0,1}\makebox(0,0)[lt]{\lineheight{1.25}\smash{\begin{tabular}[t]{l}$e_3$\end{tabular}}}}%
    \put(0.23596361,0.13753893){\color[rgb]{0,0,1}\makebox(0,0)[lt]{\lineheight{1.25}\smash{\begin{tabular}[t]{l}$e_3$\end{tabular}}}}%
    \put(0.19098019,0.27152516){\color[rgb]{0,0,1}\makebox(0,0)[lt]{\lineheight{1.25}\smash{\begin{tabular}[t]{l}$e_3$\end{tabular}}}}%
    \put(0.043705,0.17421671){\color[rgb]{0,0,1}\makebox(0,0)[lt]{\lineheight{1.25}\smash{\begin{tabular}[t]{l}$e_3$\end{tabular}}}}%
    \put(0.00517865,0.13509414){\color[rgb]{0,0,1}\makebox(0,0)[lt]{\lineheight{1.25}\smash{\begin{tabular}[t]{l}$e_3$\end{tabular}}}}%
    \put(0.13906946,0.01630838){\color[rgb]{0,0,1}\makebox(0,0)[lt]{\lineheight{1.25}\smash{\begin{tabular}[t]{l}$e_3$\end{tabular}}}}%
  \end{picture}%
\endgroup%
}
\noindent is Example 9.1 of \cite{PffAutomata}, where it is proved to be ageometric fully irreducible.
It is represented by the fully preprincipal train track map $g$ on the 3-petaled rose marked by $a = [e_1]$, $b = [e_2]$, and $c = [e_3]$, and defined by $g(e_1) = e_3 e_2 e_3 e_1$, and $g(e_2) = e_3 e_2 e_3$, and $g(e_3) = e_1 e_3$, as depicted to the right.

\begin{figure}[H]
    \centering
    \selectfont\fontsize{8pt}{8pt}
    \resizebox{!}{12cm}{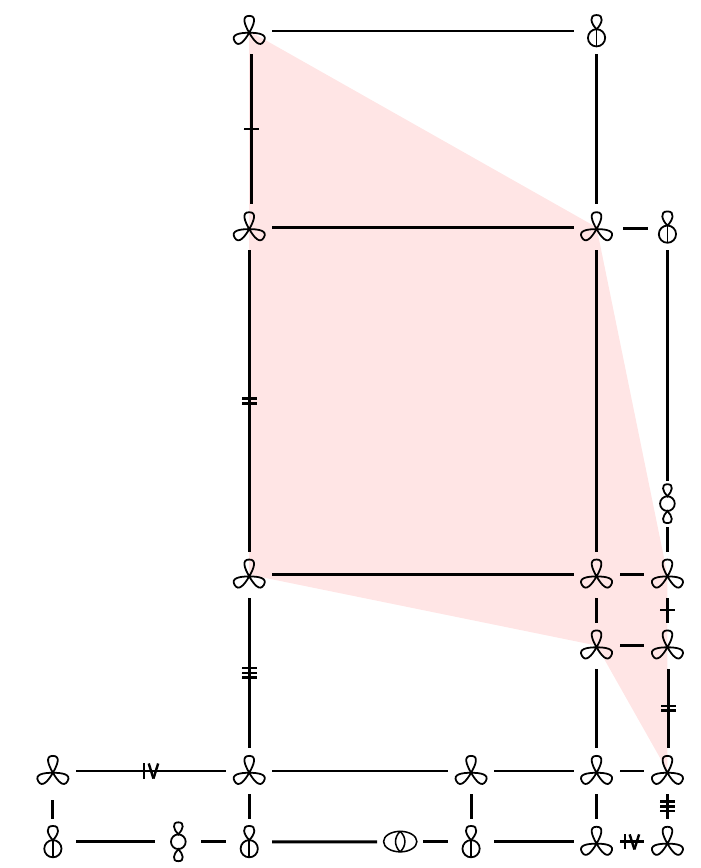}
    \caption{This is a fundamental domain of $\Av$ under the $\phi$-action, computed via the algorithm below. 1-cubes on the right are sent to the 1-cubes on the left by $\phi$ as indicated.}
    \label{fig:2dimbundlemultiartery}
\end{figure}

We computed the cubist decomposition of $\Av$ by the following algorithm:\\
1. Take $T \in \Av$ to be the fully preprincipal element that is the universal cover of the domain of the train track map above. Compute the branched cube $B_T \subset \Av$. Let $\mathcal{A}$ be the union of the $\phi$-translates of $B_T$.\\
2. For each fully preprincipal $T \in \mathcal{A} \subset \Av$, compute $B_T \subset \Av$. (There are infinitely many such $T$ but only finitely many $\phi$-orbits, so this is a finite time process.) 
\begin{itemize}
    \item[a.] If any such $B_T$ do not lie in $\mathcal{A}$, we add it to $\mathcal{A}$ and repeat this step. 
    \item[b.] If all such $B_T$ lie in $\mathcal{A}$, the algorithm terminates.
\end{itemize}
The output $\mathcal{A}$ of the algorithm is the axis bundle $\Av$: Suppose, for the sake of contradiction, there was a $T' \in \Av$ with $T' \not\in \mathcal{A}$. Then, as in the proof of \Cref{prop:branchedcubescoveraxisbundle}, there is a fold path $\alpha$ from a vertex $T$ of $\mathcal{A}$ to $T'$. Each time $\alpha$ exits a branched cube in $\mathcal{A}$ it enters another one (otherwise we could have added that branched cube to $\mathcal{A}$), so at the end of the path, $T'$ lies in $\mathcal{A}$.

\Cref{fig:2dimbundlemultiartery} depicts $\Av$. The length of each edge in the figure is the fold length, computed as follows:
The normalized eigenvector is $[0.45,0.29,0.26]$, so we start with $\ell(a)\approx0.45$, and $\ell(b)\approx0.29$, and $\ell(c)\approx0.26$. 
The vertical fold path from the node meeting the edges labeled III and IV on the left folds the turn $\{\bar{b}, \bar{c}\}$, so the fold length is $\approx 0.29(\frac{0.26}{0.26+0.29+0.26})$, or equivalently $\approx 0.26(\frac{0.26}{0.45+0.26})$. This is $\approx 0.095$. This says the lengths on the graph below are $\ell(a')\approx0.45$, and $\ell(b')\approx0.185$, and $\ell(c')\approx0.155$, and $\ell(d')\approx0.095$, where $d'$ is the edge created by the fold.
One then performs similar computations iteratively.

The cardiovascular system is drawn in red, with the arteries bold. In this example there are three arteries, demonstrating that axis bundles do not necessarily have a unique artery.
\end{ex}

\begin{ex}[Unique artery, branching] \label{ex:2dimbundlebranching}
Let $\phi \in \outt$ be defined by $\phi(a) = a\overline{c}\overline{a}$, and $\phi(b) = ac$, and
\parpic[r]{\selectfont\fontsize{9pt}{9pt} 
\begingroup%
  \makeatletter%
  \providecommand\color[2][]{%
    \errmessage{(Inkscape) Color is used for the text in Inkscape, but the package 'color.sty' is not loaded}%
    \renewcommand\color[2][]{}%
  }%
  \providecommand\transparent[1]{%
    \errmessage{(Inkscape) Transparency is used (non-zero) for the text in Inkscape, but the package 'transparent.sty' is not loaded}%
    \renewcommand\transparent[1]{}%
  }%
  \providecommand\rotatebox[2]{#2}%
  \newcommand*\fsize{\dimexpr\f@size pt\relax}%
  \newcommand*\lineheight[1]{\fontsize{\fsize}{#1\fsize}\selectfont}%
  \ifx\svgwidth\undefined%
    \setlength{\unitlength}{191.73688032bp}%
    \ifx\svgscale\undefined%
      \relax%
    \else%
      \setlength{\unitlength}{\unitlength * \real{\svgscale}}%
    \fi%
  \else%
    \setlength{\unitlength}{\svgwidth}%
  \fi%
  \global\let\svgwidth\undefined%
  \global\let\svgscale\undefined%
  \makeatother%
  \begin{picture}(1,0.328401)%
    \lineheight{1}%
    \setlength\tabcolsep{0pt}%
    \put(0,0){\includegraphics[width=\unitlength,page=1]{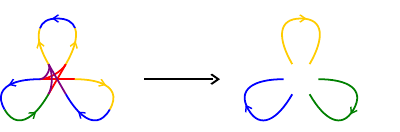}}%
    \put(0.7339123,0.30432004){\color[rgb]{1,0.8,0}\makebox(0,0)[lt]{\lineheight{1.25}\smash{\begin{tabular}[t]{l}$e_1$\end{tabular}}}}%
    \put(0.19835653,0.18369316){\color[rgb]{1,0.8,0}\makebox(0,0)[lt]{\lineheight{1.25}\smash{\begin{tabular}[t]{l}$e_1$\end{tabular}}}}%
    \put(0.24786861,0.14253456){\color[rgb]{1,0.8,0}\makebox(0,0)[lt]{\lineheight{1.25}\smash{\begin{tabular}[t]{l}$e_1$\end{tabular}}}}%
    \put(0.05150217,0.18368685){\color[rgb]{1,0.8,0}\makebox(0,0)[lt]{\lineheight{1.25}\smash{\begin{tabular}[t]{l}$e_1$\end{tabular}}}}%
    \put(0.14967812,0.005516){\color[rgb]{0,0,1}\makebox(0,0)[lt]{\lineheight{1.25}\smash{\begin{tabular}[t]{l}$e_3$\end{tabular}}}}%
    \put(0.00851023,0.14253118){\color[rgb]{0,0,1}\makebox(0,0)[lt]{\lineheight{1.25}\smash{\begin{tabular}[t]{l}$e_3$\end{tabular}}}}%
    \put(0.12236218,0.30288069){\color[rgb]{0,0,1}\makebox(0,0)[lt]{\lineheight{1.25}\smash{\begin{tabular}[t]{l}$e_3$\end{tabular}}}}%
    \put(0.90395154,0.02759912){\color[rgb]{0,0.50196078,0}\makebox(0,0)[lt]{\lineheight{1.25}\smash{\begin{tabular}[t]{l}$e_2$\end{tabular}}}}%
    \put(0.0909929,0.00551882){\color[rgb]{0,0.50196078,0}\makebox(0,0)[lt]{\lineheight{1.25}\smash{\begin{tabular}[t]{l}$e_2$\end{tabular}}}}%
    \put(0.55867757,0.02759912){\color[rgb]{0,0,1}\makebox(0,0)[lt]{\lineheight{1.25}\smash{\begin{tabular}[t]{l}$e_3$\end{tabular}}}}%
    \put(0,0){\includegraphics[width=\unitlength,page=2]{eg2dimbundlebranchingmap.pdf}}%
  \end{picture}%
\endgroup%
}
\noindent $\phi(c) = \overline{b}\overline{c}$. Then $\phi$ is represented by the train track map $g(e_1) = e_1 \overline{e_3} \overline{e_1}$, and $g(e_2) = e_1 e_3$, and $g(e_3) = \overline{e_2} \overline{e_3}$ on the 3-petaled rose depicted to the right. The marking is defined by $a = [e_1]$, $b = [e_2]$, $c = [e_3]$.
It is similarly straightforward to check $\phi$ is ageometric fully irreducible.

\begin{figure}[H]
    \centering
\selectfont\fontsize{8pt}{8pt}
    \resizebox{!}{10cm}{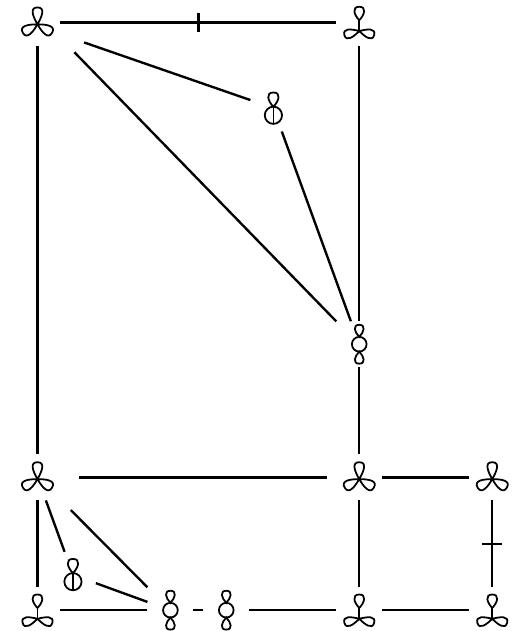}
    \caption{The axis bundle of the outer automorphism in \Cref{ex:2dimbundlebranching}. 
The 1-cube on the right is sent to the 1-cube on the \emph{top} by $\phi$ as indicated.}
    \label{fig:2dimbundlebranching}
\end{figure}

This example exhibits some phenomena that are different from \Cref{ex:2dimbundlemultiartery}:
\begin{enumerate}
    \item There is branching in the axis bundle $\mathcal{A}_\phi$. More precisely, there are two branched 2-cubes in $\mathcal{A}_\phi$, modulo the action of $\phi$.
    \item There is a (branched) 2-cube with 4 vertices on one of its folding faces. 
    \item There is a unique artery.
    \item The artery meets the boundary of $\mathcal{A}_\phi$.
\end{enumerate}
\end{ex}

\begin{ex}[Full \& stable axis bundles] \label{ex:2dimbundle1dimstable}
Define $\phi \in \outt$ by
$\phi(a) = ac$, and $\phi(b) = cbc$,
\parpic[r]{\selectfont\fontsize{9pt}{9pt} 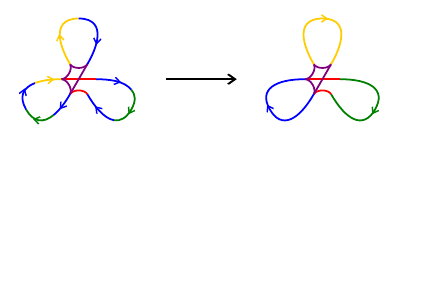}
\noindent and $\phi(c) = cbca$ and represented by the train track map $g(e_1) = e_1 e_3$, and $g(e_2) = e_3 e_2 e_3$, and $g(e_3) = e_3 e_2 e_3 e_1$ on the 3-petaled rose in the upper row of the image. The marking is defined by  $a = [e_1]$, $b = [e_2]$, $c = [e_3]$.

Straightforward computation shows this train track map has one PNP, namely $e_1*\overline{e''_2}$, where $e''_2$ is a suitable suffix of $e_2$. Collapsing this PNP yeilds the train track map in the bottom row: $g'(e_1) = e_1 e_3$, and $g'(e'_2) = e_3 e'_2$, and $g(e_3) = e_3 e'_2 e_1 e_3 e_1$.
The stable Whitehead graph of $g'$ is that of $g$, but with an additional edge.
One can see this train track map has no PNPs and apply \cite{IWGII,stablestrata}to show $\phi$ is ageometric fully irreducible.

\begin{figure}[H]
    \centering
    \selectfont\fontsize{8pt}{8pt}
    \resizebox{!}{12cm}{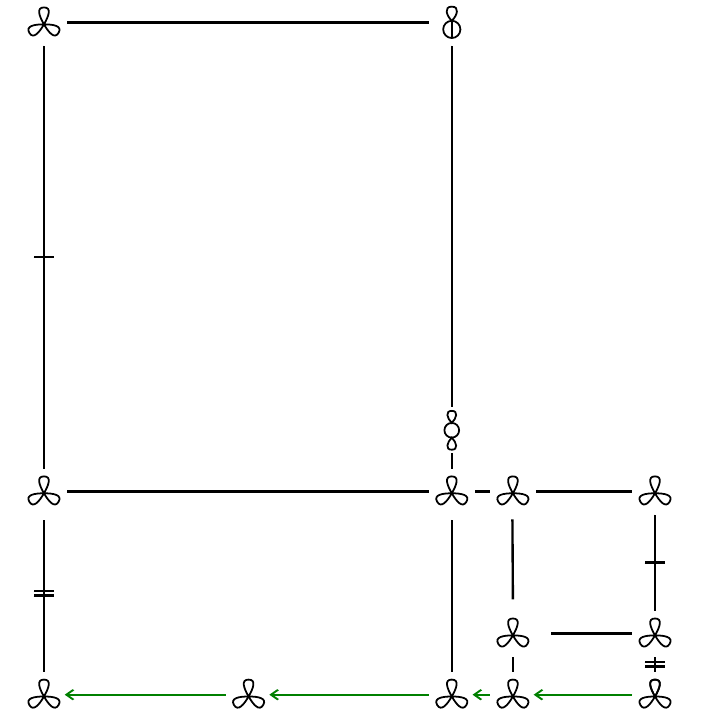}
    \caption{The axis bundle and stable axis bundle in \Cref{ex:2dimbundle1dimstable}. The 1-cubes on the right are sent to the 1-cubes on the left by $\phi$ as indicated.
The cardiovascular system of $\Av$ is drawn in red, with the arteries bold.}
    \label{fig:2dimbundle1dimstable}
\end{figure}

As reasoned above, the ideal Whitehead graph ${IW}(\phi)$ has a cut vertex. Hence, the stable axis bundle $\mathcal{SA}_\phi$ is a proper subset of $\Av$. 
In the figure, $\mathcal{SA}_\phi$ is the line on the bottom. 
The cardiovascular system of $\mathcal{SA}_\phi$ is drawn in green, with the arteries highlighted. This example demonstrates that the arteries of the stable axis bundle need not agree with the arteries of the full axis bundle.
\end{ex}

\begin{ex}[Non-constant local dimension] \label{ex:bottleneck}

Let $\phi \in \outt$ be defined by $\phi(a) = ac$, and $\phi(b) = ab$,
\parpic[r]{\selectfont\fontsize{9pt}{9pt} 
\begingroup%
  \makeatletter%
  \providecommand\color[2][]{%
    \errmessage{(Inkscape) Color is used for the text in Inkscape, but the package 'color.sty' is not loaded}%
    \renewcommand\color[2][]{}%
  }%
  \providecommand\transparent[1]{%
    \errmessage{(Inkscape) Transparency is used (non-zero) for the text in Inkscape, but the package 'transparent.sty' is not loaded}%
    \renewcommand\transparent[1]{}%
  }%
  \providecommand\rotatebox[2]{#2}%
  \newcommand*\fsize{\dimexpr\f@size pt\relax}%
  \newcommand*\lineheight[1]{\fontsize{\fsize}{#1\fsize}\selectfont}%
  \ifx\svgwidth\undefined%
    \setlength{\unitlength}{201.80779645bp}%
    \ifx\svgscale\undefined%
      \relax%
    \else%
      \setlength{\unitlength}{\unitlength * \real{\svgscale}}%
    \fi%
  \else%
    \setlength{\unitlength}{\svgwidth}%
  \fi%
  \global\let\svgwidth\undefined%
  \global\let\svgscale\undefined%
  \makeatother%
  \begin{picture}(1,0.17964375)%
    \lineheight{1}%
    \setlength\tabcolsep{0pt}%
    \put(0,0){\includegraphics[width=\unitlength,page=1]{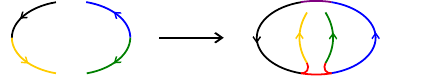}}%
    \put(0.65124309,0.08259615){\color[rgb]{1,0.8,0}\makebox(0,0)[lt]{\lineheight{1.25}\smash{\begin{tabular}[t]{l}$e_2$\end{tabular}}}}%
    \put(-0.00225016,0.00524075){\color[rgb]{1,0.8,0}\makebox(0,0)[lt]{\lineheight{1.25}\smash{\begin{tabular}[t]{l}$e_2$\end{tabular}}}}%
    \put(0.80909961,0.08259615){\color[rgb]{0,0.50196078,0}\makebox(0,0)[lt]{\lineheight{1.25}\smash{\begin{tabular}[t]{l}$e_3$\end{tabular}}}}%
    \put(0.22129786,0.11392055){\color[rgb]{0,0.50196078,0}\makebox(0,0)[lt]{\lineheight{1.25}\smash{\begin{tabular}[t]{l}$e_3$\end{tabular}}}}%
    \put(0.29662818,0.00524461){\color[rgb]{0,0.50196078,0}\makebox(0,0)[lt]{\lineheight{1.25}\smash{\begin{tabular}[t]{l}$e_3$\end{tabular}}}}%
    \put(0.90874458,0.08259615){\color[rgb]{0,0,1}\makebox(0,0)[lt]{\lineheight{1.25}\smash{\begin{tabular}[t]{l}$e_4$\end{tabular}}}}%
    \put(0.29662175,0.14736862){\color[rgb]{0,0,1}\makebox(0,0)[lt]{\lineheight{1.25}\smash{\begin{tabular}[t]{l}$e_4$\end{tabular}}}}%
    \put(0.0776864,0.11392092){\color[rgb]{0,0,1}\makebox(0,0)[lt]{\lineheight{1.25}\smash{\begin{tabular}[t]{l}$e_4$\end{tabular}}}}%
    \put(0.07768244,0.04441848){\color[rgb]{0,0,0}\makebox(0,0)[lt]{\lineheight{1.25}\smash{\begin{tabular}[t]{l}$e_1$\end{tabular}}}}%
    \put(0,0){\includegraphics[width=\unitlength,page=2]{egbottleneckmap.pdf}}%
    \put(0.55004852,0.08259739){\color[rgb]{0,0,0}\makebox(0,0)[lt]{\lineheight{1.25}\smash{\begin{tabular}[t]{l}$e_1$\end{tabular}}}}%
    \put(0.22129154,0.04441848){\color[rgb]{0,0,0}\makebox(0,0)[lt]{\lineheight{1.25}\smash{\begin{tabular}[t]{l}$e_1$\end{tabular}}}}%
    \put(-0.00224759,0.14736471){\color[rgb]{0,0,0}\makebox(0,0)[lt]{\lineheight{1.25}\smash{\begin{tabular}[t]{l}$e_1$\end{tabular}}}}%
  \end{picture}%
\endgroup%
}
\noindent and $\phi(c) = a\overline{b}c$. 
Then $\phi$ is represented by the train track map $g(e_1) = e_1e_2$, and $g(e_2) = e_1e_4$, $g(e_3) = e_1e_3$, and $g(e_4) = \overline{e_3}e_4$, illustrated to the right.
The marking of the train track is defined by $a = [e_1e_2], b = [e_1e_3], c = [e_1e_4]$.

A straightforward computation shows $\phi$ satisfies the \cite{IWGII,stablestrata} criterion, so is ageometric fully irreducible. \Cref{fig:egbottleneck} shows the cubist decomposition of $\Av$, where the 1-cube on the bottom is sent to the 1-cube on the top by $\phi$ as indicated.

\begin{figure}[ht!]
    \centering
    \selectfont\fontsize{8pt}{8pt}
    \resizebox{!}{12cm}{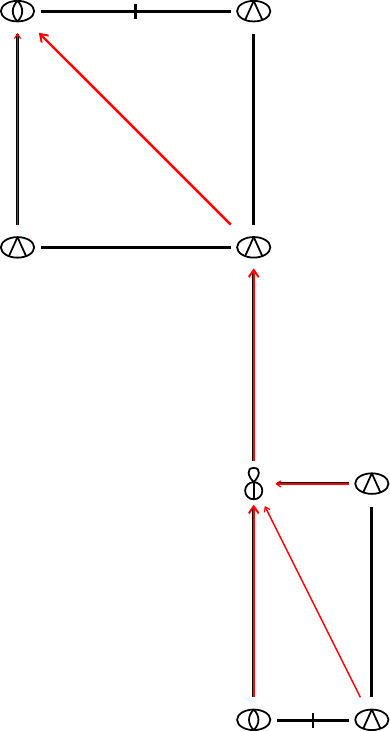}
    \caption{The axis bundle of the outer automorphism in \Cref{ex:bottleneck}.}
    \label{fig:egbottleneck}
\end{figure}

Note that the local dimension of $\Av$ is not constant: at some points it is $1$, while at other points it is $2$.

\end{ex}

\bigskip
\bigskip

\bibliography{PaperRefs} 

\bibliographystyle{alphaurl} 

\end{document}